\documentclass[11pt]{article}
\usepackage[margin=1.2in]{geometry}
\usepackage[utf8]{inputenc}
\usepackage[T1]{fontenc}
\usepackage[english]{babel}
\usepackage{subcaption}
\usepackage{authblk}
\usepackage{url}
\usepackage{algorithm} 
\usepackage{algorithmicx} 
\usepackage{algpseudocode} 
\usepackage{scalerel}
\usepackage{lipsum}
\usepackage{mathtools}
\usepackage{enumitem}
\usepackage[title]{appendix}
\usepackage{multirow}
\usepackage{cite}

\usepackage[running, displaymath, mathlines]{lineno}

\def\scfactor{0.25}
\def\scfactorx{0.25}

\usepackage{hyperref}
\hypersetup{
    colorlinks=true,
    linkcolor=blue,
    citecolor=magenta,      
    urlcolor=black
}

\usepackage{slashbox}
\usepackage{bbding}
\usepackage{colortbl}

\usepackage{tikz}
\usetikzlibrary{calc}
\usetikzlibrary{patterns}
\usepackage{mathrsfs}
\usetikzlibrary{arrows}

\usepackage[all]{xypic}
\usepackage{graphicx} 
\usepackage{amsthm}
\usepackage{amsfonts}
\usepackage{amsmath}
\usepackage{amssymb}
\usepackage{array}
\usepackage{blkarray}
\usepackage{multirow}
\usepackage{thm-restate}
\usepackage{xcolor}
\usepackage{subcaption}

\theoremstyle{plain}
\newtheorem{theorem}{Theorem}

\newtheorem{lemma}[theorem]{Lemma}
\newtheorem{corollary}[theorem]{Corollary}
\newtheorem{proposition}[theorem]{Proposition}

\theoremstyle{definition}
\newtheorem{definition}{Definition}

\newtheorem{example}{Example}

\def\x{\mathbf{q}}
\def\X{\mathbf{Q}}
\def\special{\mathbf{s}}

\def\0{\mathbf{0}}

\newcommand{\noop}[1]{}

\DeclareMathOperator{\altan}{\mathfrak{a}}

\DeclareMathOperator{\spn}{span}

\title{On the Nullity of Altans and Iterated Altans}

\author[1,2,3,$*$]{Nino Ba{\v s}i{\'c}}
\author[4,$*$]{Patrick~W.~Fowler}

\affil[1]{FAMNIT, University of Primorska, Koper, Slovenia}
\affil[2]{IAM, University of Primorska, Koper, Slovenia}
\affil[3]{Institute of Mathematics, Physics and Mechanics, Ljubljana, Slovenia}
\affil[4]{Department of Chemistry, University of Sheffield, Sheffield S3 7HF, UK}

\date{March 14, 2022}

\begin{document}

\maketitle

\begin{abstract}
Altanisation (formation of the {\it altan} of a parent structure) 
originated
in the chemical literature as a formal device for constructing 
generalised coronenes from smaller structures.  
The altan of graph $G$, denoted $\altan(G, H)$, depends on the choice of \emph{attachment set} $H$
(a cyclic $h$-tuple of vertices of $G$). From a given pair $(G, H)$, the altan construction
produces a pair $(G', H')$, where $H'$ is called the induced attachment set. Repetition of the
construction, using at each stage the attachment set induced in the previous step, 
defines the \emph{iterated altan}. 
Here, we prove sharp bounds  for the
nullity 
of {\it altan} and {\it iterated altan} 
graphs based on a general parent graph: for any attachment set with odd $h$,
nullities of altan and parent are equal;
for any $h$ and all $k \geq 1$, the $k$-th altan has the same nullity
as the first;
for any attachment set with even $h$, the nullity of the altan  
exceeds the nullity of the parent graph by one of the three values $\{0, 1, 2\}$.
The case of excess nullity $2$ has not been noticed before; for benzenoids with the 
\emph{natural} attachment set consisting of the CH sites, it occurs first for a parent structure with  $5$ hexagons. 
On the basis of extensive computation, it is conjectured that in fact no
altan of a \emph{convex} benzenoid has excess nullity~$2$.

\vspace{\baselineskip}

\vspace{\baselineskip}
\noindent
\textbf{Math.\ Subj.\ Class.\ (2020):} 
05C50, 
15A18, 
05C92 
\end{abstract}

\section{Introduction}

Aromaticity has been an influential 
concept in chemistry 
for a century and a half, and despite difficulties with precise definition
 is still 
invoked routinely in qualitative explanation of stability, reactivity and magnetic 
response of conjugated 
systems \cite{Schleyer2001,Schleyer2005,Martin2015,Peeks2020,Peeks2021}. 
On the widely accepted magnetic criterion, aromaticity is defined 
by the ability of a system to sustain 
{\it ring currents} (circulations of ($\pi$) electrons) induced by a perpendicular external
magnetic field \cite{London1937,Pople1958,McWeeny1958,Schleyer1996}, where a
multi-centre  $\pi$ current with diatropic/paratropic sense implies aromaticity/anti-aromaticity
  and absence of current implies non-aromaticity. 
The prediction of 
patterns of ring current and design of 
carbon nanostructures with specific magnetic properties are subjects of 
active research \cite{Monaco2012,Stepien2018}. One
strategy for design of carbon nanostructures that should support concentric 
currents is based on the venerable 
annulene-within-an-annulene 
model \cite{Lawton1971,Ege1972,Diederich1978,Aihara2014}, 
which predicts combinations 
of inner and outer currents on concentric 
cycles connected by `spoke' bonds. Depending on the lengths of  inner and 
outer circuits, and the strength of the spoke
coupling between them, any of the four combinations of diatropic and paratropic 
current may appear \cite{Ege1972,Steiner2001,Monaco2006,Monaco2007}. 
As with many chemical models, this approach 
gives a rule of thumb rather 
than a rigorous prediction \cite{coronene}, but it can serve as a starting point for more 
detailed explanations.

A development from this earlier picture is the formal strategy of `altanisation’ in 
which a central carbon framework is surrounded 
by an {\bf an}nulene perimeter in which carbon centres with two and 
three carbon 
neighbours {\bf alt}ernate \cite{Monaco2012,Monaco2013a,Zanasi2016}. 
The {\it altan} concept has spawned a sizeable quantum chemical 
literature, in which the construction is used to 
generate systems (often hypothetical)  for which properties can be calculated 
with empirical, semi-empirical or {\it ab initio} methods \cite{Monaco2006,Monaco2013b,Dickens2014a,
Dickens2014b,Dickens2014c,Dickens2015a,
Dickens2015b,Dickens2018,Piccardo2020,Dickens2020,Dickens2021}. 
Interest has centred on creating systems with unusual 
paratropic perimeter currents \cite{Monaco2007} or with 
unusual dependence of outer currents on the total charge \cite{Dickens2021}. 

In qualitative theories of electronic structure 
in organic chemistry, there is a 
conceptual split between localised and delocalised 
descriptions, based respectively on Kekul{\'e} structures 
(perfect matchings) or the balance of bonding, non-bonding and 
anti-bonding 
molecular orbitals 
(eigenvectors corresponding to positive,
zero and negative eigenvalues of the adjacency matrix). 
One important quantity is the \emph{nullity} of the graph -- the 
multiplicity of the zero eigenvalue, or
in chemical terms, the number of non-bonding $\pi$ molecular orbitals
(NBMOs).
The altan construction has
inspired mathematical investigations of matchings and spectra
\cite{Gutman2014a,Gutman2014b,Basic2015,Basic2016}: 
{altans} may be Kekulean (have a perfect matching) or not, and 
their graphs may be singular (have at least one zero 
eigenvalue of the adjacency matrix) or not.  

In the H{\"u}ckel model of unsaturated carbon frameworks, 
the number of NBMOs \cite{Trinajstic1992}
has consequences for stability, reactivity and magnetic properties. Nullity is also
of interest in the context of \emph{graph energy} \cite{Gutman2021a,Gutman2021b,Triantafillou2016}.
For the general theory of graph spectra, the reader may consult standard
references \cite{Cvetkovic1995,Haemers2012,Cvetkovic1997,Cvetkovic2010,Chung1997}, 
and for a survey on nullity see \cite{Borovicanin2011}.

Nullity has been studied for altans 
using a mixture of formal theorems and empirical 
observations on small sets of examples, and a certain amount of 
confusion about the limits on altan nullity has arisen as a result. 
We aim to clarify the situation here.
Specifically, in the present paper, we build 
on established mathematical results \cite{Gutman2014a,Gutman2014b} to 
derive the precise relationship
between the nullity of 
the altan and the parent graph, and between nullities 
of first and subsequent iterated altans.  
Some previously unsuspected cases of a higher increase in nullity 
on first altanisation are detected. These occur 
for surprisingly small chemical examples.

\section{Preliminaries}
\subsection{The altan construction}

In this section we give a formal graph-theoretical definition of the altan of a graph $G$
and develop some useful concepts.
We follow Gutman's definitions \cite{Gutman2014b} with some variations in terminology and notation.
The primary object of our study is the
pair $(G, H)$ where $G$ is a (simple) graph and $H$ is what we will call, for simplicity, an \emph{attachment set}. 
The attachment `set' is a cyclic $h$-tuple of vertices, in which any given vertex is allowed to occur more than once.
Note that in the case of cyclic tuples we do not distinguish between the tuple $(v_1, v_2, \ldots, v_h)$
and its circular shifts $(v_{i+1}, v_{i+2},\ldots, v_h, v_1, v_2, \ldots, v_i)$ for $1 \leq i < h$. In \cite{Basic2015},
the attachment set was called the peripheral root. Both terms have their disadvantages: the attachment set is
not strictly a set, and its vertices may lie in the interior of the graph. 

\begin{definition}
\label{def:altan}
Let $G = (V(G), E(G))$ be a graph and  $H = (v_1, v_2, \ldots, v_h)$ be an attachment set, i.e.,
a cyclic $h$-tuple with $h \geq 2$ and $v_i \in V(G)$ for $1 \leq i \leq h$. Let $G'$ be the
graph with
\begin{align*}
V(G') & = V(G) \sqcup \{ x_1, x_2, \ldots, x_h \} \sqcup \{ y_1, y_2, \ldots, y_h \}, \\
E(G') & = E(G) \sqcup \{ v_ix_i, x_iy_i : 1 \leq i \leq h \} \sqcup \{ y_ix_{i+1} : 1 \leq i < h\} \sqcup \{ y_h x_1 \}. 
\end{align*}
and let $H' = (y_1, y_2, \ldots, y_h)$. The altan of $(G, H)$, denoted $\altan(G, H)$, is the pair $(G', H')$,
where this particular choice of
$H'$ will be called the \emph{induced attachment set}.
\end{definition}

\begin{figure}[!htb]
\centering
\begin{tikzpicture}[scale=1.2]
\tikzstyle{edge}=[draw,thick]
\tikzstyle{every node}=[draw, circle, fill=blue!50!white, inner sep=1.5pt]
\draw[fill=green!20!white] (0,0) ellipse (2cm and 1.2cm);
\node[label={[yshift=6pt]-90:$v_{h-1}$}] (vh-1) at (170:1.6) {};
\node[label={[yshift=2pt]-90:$v_h$}] (vh) at (140:1.2) {};
\node[label={[yshift=2pt]-90:$v_1$}] (v1) at (90:0.95) {};
\node[label={[yshift=2pt]-90:$v_2$}] (v2) at (40:1.2) {};
\node[label={[yshift=2pt]-90:$v_3$}] (v3) at (10:1.6) {};
\node[label={[yshift=0pt]180:$x_{h-1}$}] (xh-1) at (170:1.6*1.5) {};
\node[label={[yshift=0pt]120:$x_h$}] (xh) at (140:1.2*1.7) {};
\node[label={[yshift=-2pt]90:$x_1$}] (x1) at (90:0.95*1.9) {};
\node[label={[yshift=0pt]60:$x_2$}] (x2) at (40:1.2*1.7) {};
\node[label={[yshift=0pt]0:$x_3$}] (x3) at (10:1.6*1.5) {};
\node[label={[yshift=0pt]180:$y_{h-2}$}] (yh-2) at (185:2.8) {};
\node[label={[yshift=0pt]180:$y_{h-1}$}] (yh-1) at ($ (xh-1)!0.5!(xh) + (150:0.5) $) {};
\node[label={[yshift=-2pt]90:$y_h$}] (yh) at ($ (xh)!0.5!(x1) + (105:0.4) $) {};
\node[label={[yshift=-2pt]90:$y_1$}] (y1) at ($ (x1)!0.5!(x2) + (75:0.4) $) {};
\node[label={[yshift=0pt]0:$y_2$}] (y2) at ($ (x2)!0.5!(x3) + (30:0.5) $) {};
\node[label={[yshift=0pt]0:$y_3$}] (y3) at (-5:2.8) {};
\path[edge] ($ (yh-2) + (-50:0.4) $) -- (yh-2) -- (xh-1) -- (yh-1) -- (xh) -- (yh) -- (x1) -- (y1) -- (x2) -- (y2) -- (x3) -- (y3) -- ($ (y3) + (230:0.4) $);
\path[edge] (vh-1) -- (xh-1);
\path[edge] (vh) -- (xh);
\path[edge] (v1) -- (x1);
\path[edge] (v2) -- (x2);
\path[edge] (v3) -- (x3);
\node[draw=none,fill=none] at (0, 0) {$G$};
\node[draw=none,fill=none] at (-20:1.2) {$\ldots$};
\node[draw=none,fill=none] at ($ (y3) + (230:0.8) $) {$\ldots$};
\end{tikzpicture}
\caption{The altan $\altan(G, H)$.
Attachment is via the set $\{ v_1, \ldots, v_h\}$, possibly
with multiple connections to some vertices, and the perimeter
is the cycle $\{x_1, y_1, \ldots , x_h, y_h\}$.}
\label{fig:altan}
\end{figure}
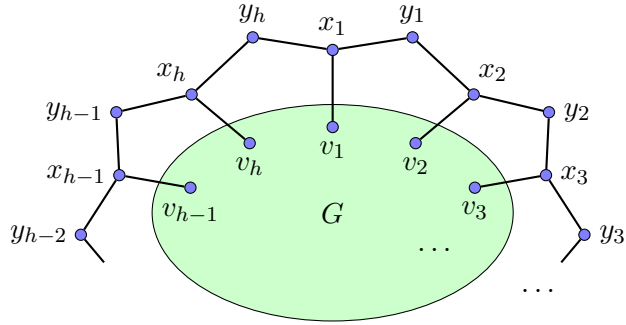

For an illustration of the definition see Figure~\ref{fig:altan}.
In the above definition,
the symbol $\sqcup$ denotes a disjoint union. Observe that $|V(G')| = |V(G)| + 2h$ and
$|E(G')| = |E(G)| + 3h$. The graph $G$ may be disconnected. The graph $G'$ is connected if and
only if the attachment set $H$ contains at least one vertex from each connected component of $G$.
Note that the induced attachment set $H'$ is only one possible choice of attachment set for a subsequent 
altan operation on $G'$; however,
this choice is conveniently used to define an \emph{iterated altan} construction:
\begin{align*}
\altan^k(G, H) = \begin{cases}
(G, H) & \text{if } k = 0, \\
\altan(G, H) & \text{if } k = 1, \\
\altan(\altan^{k - 1}(G, H)) & \text{if } k \geq 2.
\end{cases}
\end{align*}
This recursive definition implicitly defines the sequence of intermediate graphs and their respective 
attachment sets (see Figure~\ref{fig:iteraltan}). We call $\altan^k(G, H)$ the $k$-th altan of $(G, H)$.
\begin{figure}[!htb]
\centering
\begin{tikzpicture}[scale=0.8]
\tikzstyle{edge}=[draw,thick]
\tikzstyle{every node}=[draw, circle, fill=blue!50!white, inner sep=1.5pt]
\draw[fill=green!20!white] (0,0) ellipse (2cm and 1.2cm);
\node[fill=green!80!white] (vh) at (140:1.2) {};
\node[fill=green!80!white] (v1) at (90:0.95) {};
\node[fill=green!80!white] (v2) at (40:1.2) {};
\node[fill=green!80!white] (v3) at (10:1.6) {};
\node[fill=black] (xh) at (140:1.2*1.7) {};
\node[fill=black] (x1) at (90:0.95*1.9) {};
\node[fill=black] (x2) at (40:1.2*1.7) {};
\node[fill=black] (x3) at (10:1.6*1.5) {};
\node[fill=black] (yh) at ($ (xh)!0.5!(x1) + (105:0.4) $) {};
\node[fill=black] (y1) at ($ (x1)!0.5!(x2) + (75:0.4) $) {};
\node[fill=black] (y2) at ($ (x2)!0.5!(x3) + (30:0.5) $) {};
\node[fill=black] (y3) at (-5:2.8) {};
\path[edge] ($ (xh) + (210:0.4) $) -- (xh) -- (yh) -- (x1) -- (y1) -- (x2) -- (y2) -- (x3) -- (y3) -- ($ (y3) + (230:0.4) $);
\path[edge] (vh) -- (xh);
\path[edge] (v1) -- (x1);
\path[edge] (v2) -- (x2);
\path[edge] (v3) -- (x3);
\node[fill=magenta] (xxh) at ($ (yh) + (105:0.6) $) {};
\node[fill=magenta] (xx1) at ($ (y1) + (75:0.6) $) {};
\node[fill=magenta] (xx2) at ($ (y2) + (25:0.6) $) {};
\node[fill=magenta] (xx3) at ($ (y3) + (0:0.7) $) {};
\node[fill=magenta] (yyh) at (135:1.2*1.7*1.6) {};
\node[fill=magenta] (yy1) at (90:0.95*1.9*1.6) {};
\node[fill=magenta] (yy2) at (45:1.2*1.7*1.6) {};
\node[fill=magenta] (yy3) at (10:1.6*1.5*1.5) {};
\path[edge,color=magenta] ($ (yyh) + (230:0.4) $) -- (yyh) -- (xxh) -- (yy1) -- (xx1) -- (yy2) -- (xx2) -- (yy3) -- (xx3) -- ($ (xx3) + (-70:0.4) $);
\path[edge,color=magenta] (yh) -- (xxh);
\path[edge,color=magenta] (y1) -- (xx1);
\path[edge,color=magenta] (y2) -- (xx2);
\path[edge,color=magenta] (y3) -- (xx3);
\node[fill=blue!60!white] (xxxh) at (135:1.2*1.7*1.9) {};
\node[fill=blue!60!white] (xxx1) at (90:0.95*1.9*1.9) {};
\node[fill=blue!60!white] (xxx2) at (45:1.2*1.7*1.9) {};
\node[fill=blue!60!white] (xxx3) at (10:1.6*1.5*1.8) {};
\node[fill=blue!60!white] (yyyh) at ($ (xxxh)!0.5!(xxx1) + (105:0.45) $) {};
\node[fill=blue!60!white] (yyy1) at ($ (xxx1)!0.5!(xxx2) + (75:0.5) $) {};
\node[fill=blue!60!white] (yyy2) at ($ (xxx2)!0.5!(xxx3) + (30:0.5) $) {};
\node[fill=blue!60!white] (yyy3) at (-2:4.7) {};
\path[edge,color=blue!60!white] (yyh) -- (xxxh);
\path[edge,color=blue!60!white] (yy1) -- (xxx1);
\path[edge,color=blue!60!white] (yy2) -- (xxx2);
\path[edge,color=blue!60!white] (yy3) -- (xxx3);
\path[edge,color=blue!60!white] ($ (xxxh) + (180:0.4) $) -- (xxxh) -- (yyyh) -- (xxx1) -- (yyy1) -- (xxx2) -- (yyy2) -- (xxx3) -- (yyy3) -- ($ (yyy3) + (250:0.4) $);
\node[draw=none,fill=none] at (0, 0) {$G$};
\node[draw=none,fill=none] at (-20:1.2) {$\dots$};
\node[draw=none,fill=none] at ($ (yyh) + (230:0.8) $) {$\ldots$};
\node[draw=none,fill=none] at ($ (xx3) + (-90:0.7) $) {$\ldots$};
\node[draw=none,fill=none] at ($ (xx3) + (-90:0.7) $) {$\ldots$};
\end{tikzpicture}
\caption{The iterated altan $\altan^3(G, H)$. Colours indicate the new vertices and edges
that are added at each stage.}
\label{fig:iteraltan}
\end{figure}
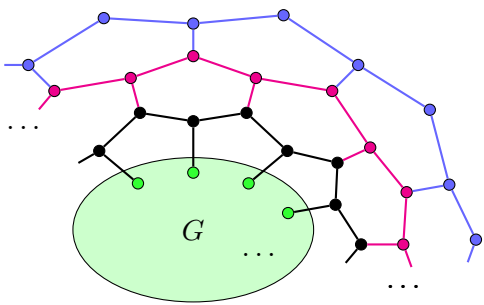

We recall some basic concepts from spectral graph theory. 
Let $A(G)$ denote the adjacency matrix of a graph $G$.
Graph $G$ is called \emph{singular} if $A(G)$ has a non-trivial kernel.
The nullity of $G$, denoted $\eta(G)$, is the algebraic multiplicity of the
number $0$ in the spectrum of $A(G)$, i.e.\ $\eta(G) = \dim \ker A(G)$.
An eigenvector $\x$ of $A(G)$ can be viewed as a weighting of vertices of $G$,
i.e., a mapping $\x\colon V(G) \to \mathbb{R}$. Let $\lambda$
be an eigenvalue of $A(G)$, let $v \in V(G)$ and let $N_G(v)$ denote
the neighbourhood of $v$. The equation
\begin{equation}
\label{eq:local}
\lambda \x(v) = \sum_{u \in N_G(v)} \x(u)
\end{equation}
is sometimes called the \emph{local condition} 
(with the pivot $v$). The notation $u \in N_G(v)$ will
be abbreviated $u \sim v$ for convenience. We will be using the local condition 
for the eigenvalue $\lambda = 0$, in particular.

Application of H{\"u}ckel theory in organic chemistry is usually limited to
systems where the  molecular graph is a  {\it chemical graph}, i.e. a graph 
that is simple, connected and with maximum degree at most three.
Altanisation of a chemical graph with an arbitrary 
attachment set
may lead to a non-chemical result, 
but this can be avoided by choosing a `natural' attachment set based on the implied hydrogen positions 
in the chemical formula for the parent chemical graph.
One interpretation of this chemically motivated construction is that it represents the formal replacement 
of all hydrogen atoms of an unsaturated hydrocarbon by vinyl groups, which then cyclise in a particular way \cite{Gutman2014b}.  
\begin{example}
Let $P_3$ be the path on three vertices and let $V(P_3) = \{u_1, u_2, u_3\}$. It is easy to
check that $\eta(P_3) = 1$. We will consider the following four attachment sets:
\begin{align*}
H_1 & = \{u_1, u_1, u_2, u_3, u_3\}, &
H_2 & = \{u_1, u_3\}, &
H_3 & = \{u_1, u_2, u_3, u_2\}, &
H_4 & = \{u_1, u_1, u_3, u_3\}.
\end{align*}
It is easy to verify that
\begin{align*}
\eta(\altan(P_3, H_1)) = \eta(\altan(P_3, H_2)) & = 1, &
\eta(\altan(P_3, H_3)) & = 2, &
\eta(\altan(P_3, H_4)) & = 3.
\end{align*}
Kernel eigenvectors of the parent graph $P_3$ and the
four altans are shown in Figure~\ref{fig:example}.

The altan $\altan(P_3, H_3)$ includes a vertex of degree four and is thus not chemical,
but it does illustrate the case where $\eta(\altan(P_3, H_3)) = \eta(P_3) + 1$.
The present example shows that even a small graph can produce
a range of nullities;
in fact, it turns out that this example demonstrates all the possibilities, 
as the theorems below will demonstrate.
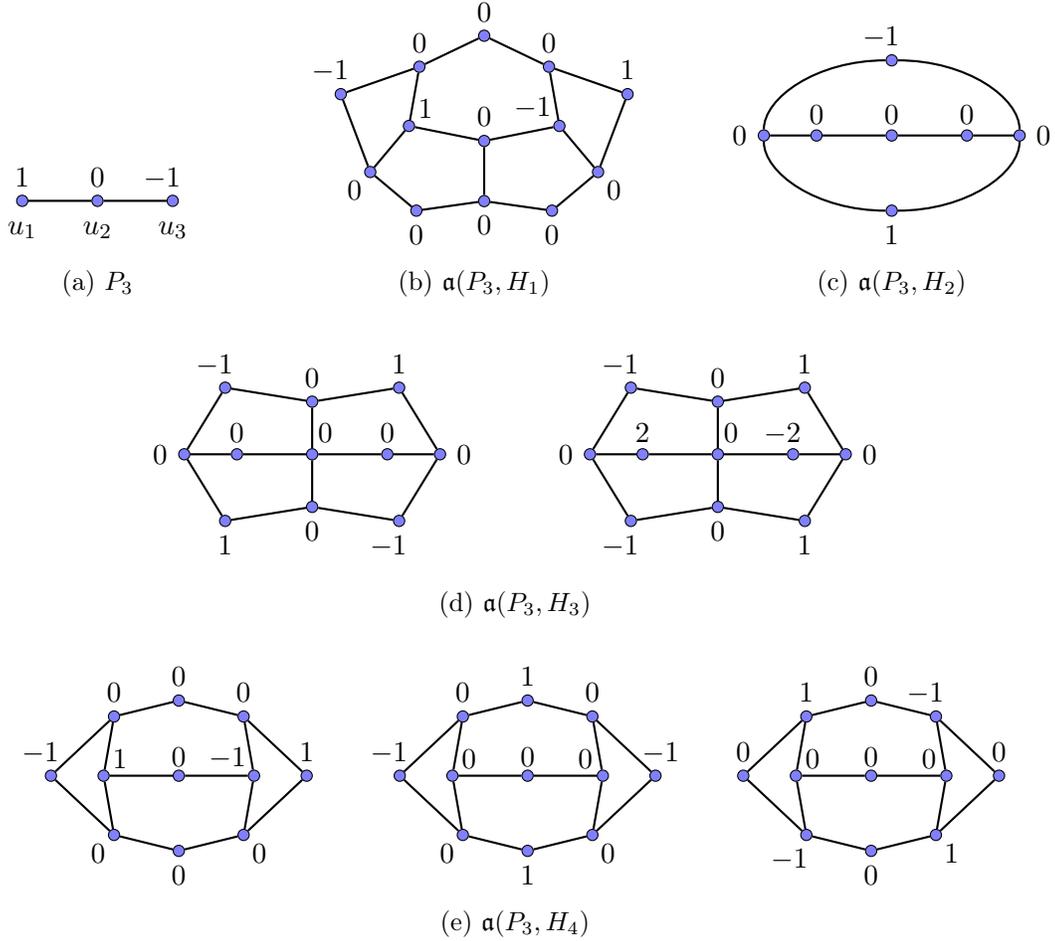
\begin{figure}[!htbp]
\centering
\begin{subfigure}[b]{0.28\textwidth}
\centering
\begin{tikzpicture}
\tikzstyle{edge}=[draw,thick]
\tikzstyle{every node}=[draw, circle, fill=blue!50!white, inner sep=1.5pt]
\node[label=-90:$u_1$,label={90:$1$}] (u1) at (-1, 0) {};
\node[label=-90:$u_2$,label={90:$0$}] (u2) at (-0, 0) {};
\node[label=-90:$u_3$,label={[xshift=-4pt,yshift=-4pt]90:$-1$}] (u3) at (1, 0) {};
\path[edge] (u1) -- (u2) -- (u3);
\end{tikzpicture}
\caption{$P_3$}
\label{fig:ex1a}
\end{subfigure}
\begin{subfigure}[b]{0.35\textwidth}
\centering
\begin{tikzpicture}
\tikzstyle{edge}=[draw,thick]
\tikzstyle{every node}=[draw, circle, fill=blue!50!white, inner sep=1.5pt]
\node[label={50:$1$}] (u1) at (-1, 0.2) {};
\node[label={90:$0$}] (u2) at (-0, 0) {};
\node[label={[xshift=-1pt,yshift=-2pt]140:$-1$}] (u3) at (1, 0.2) {};
\node[label={-120:$0$}] (x1) at ($ (u1) + (-130:0.8) $) {};
\node[label={90:$0$}] (x2) at ($ (u1) + (80:0.8) $) {};
\node[label={90:$0$}] (x3) at ($ (u3) + (100:0.8) $) {};
\node[label={-60:$0$}] (x4) at ($ (u3) + (-50:0.8) $) {};
\node[label={-90:$0$}] (x5) at ($ (u2) + (-90:0.8) $) {};
\node[label={[xshift=-4pt,yshift=-4pt]90:$-1$}] (y1) at ($ (u1) + (155:1.0) $) {};
\node[label={90:$0$}] (y2) at ($ (u2) + (90:1.4) $) {};
\node[label={90:$1$}] (y3) at ($ (u3) + (25:1.0) $) {};
\node[label={-90:$0$}] (y4) at ($ (x4) + (-140:0.8) $) {};
\node[label={-90:$0$}] (y5) at ($ (x1) + (-40:0.8) $) {};
\path[edge] (u1) -- (u2) -- (u3);
\path[edge] (x1) -- (u1) -- (x2);
\path[edge] (x3) -- (u3) -- (x4);
\path[edge] (x5) -- (u2);
\path[edge] (x1) -- (y1) -- (x2) -- (y2) -- (x3) -- (y3) -- (x4) -- (y4) -- (x5) -- (y5) -- (x1);
\end{tikzpicture}
\caption{$\altan(P_3, H_1)$}
\label{fig:ex1b}
\end{subfigure}
\begin{subfigure}[b]{0.35\textwidth}
\centering
\begin{tikzpicture}
\tikzstyle{edge}=[draw,thick]
\tikzstyle{every node}=[draw, circle, fill=blue!50!white, inner sep=1.5pt]
\node[label={[yshift=-1pt]90:$0$}] (u1) at (-1, 0.1) {};
\node[label={[yshift=-1pt]90:$0$}] (u2) at (-0, 0.1) {};
\node[label={[yshift=-1pt]90:$0$}] (u3) at (1, 0.1) {};
\node[label={-180:$0$}] (x1) at ($ (u1) + (180:0.7) $) {};
\node[label={0:$0$}] (x2) at ($ (u3) + (0:0.7) $) {};
\node[label={[xshift=-4pt,yshift=-4pt]90:$-1$}] (y1) at ($ (u2) + (90:1.0) $) {};
\node[label={-90:$1$}] (y2) at ($ (u2) + (-90:1.0) $) {};
\path[edge] (u1) -- (u2) -- (u3);
\path[edge] (x1) -- (u1);
\path[edge] (x2) -- (u3);
\path[edge] (x1) .. controls ($ (x1) + (90:0.5) $) and ($ (y1) + (180:1) $) .. (y1);
\path[edge] (x2) .. controls ($ (x2) + (90:0.5) $) and ($ (y1) + (0:1) $) .. (y1);
\path[edge] (x1) .. controls ($ (x1) + (-90:0.5) $) and ($ (y2) + (180:1) $) .. (y2);
\path[edge] (x2) .. controls ($ (x2) + (-90:0.5) $) and ($ (y2) + (0:1) $) .. (y2);
\end{tikzpicture}
\caption{$\altan(P_3, H_2)$}
\label{fig:ex1c}
\end{subfigure}

\vspace{\baselineskip}
\begin{subfigure}[b]{\textwidth}
\centering
\begin{tikzpicture}
\tikzstyle{edge}=[draw,thick]
\tikzstyle{every node}=[draw, circle, fill=blue!50!white, inner sep=1.5pt]
\node[label={[yshift=-1pt]90:$0$}] (u1) at (-1, 0.0) {};
\node[label={[yshift=-1pt,xshift=5pt]90:$0$}] (u2) at (-0, 0) {};
\node[label={[yshift=-1pt]90:$0$}] (u3) at (1, 0.0) {};
\node[label={-180:$0$}] (x1) at ($ (u1) + (180:0.7) $) {};
\node[label={90:$0$}] (x2) at ($ (u2) + (90:0.7) $) {};
\node[label={0:$0$}] (x3) at ($ (u3) + (0:0.7) $) {};
\node[label={-90:$0$}] (x4) at ($ (u2) + (-90:0.7) $) {};
\node[label={[xshift=-4pt,yshift=-4pt]90:$-1$}] (y1) at ($ (u1) + (100:0.9) $) {};
\node[label={90:$1$}] (y2) at ($ (u3) + (80:0.9) $) {};
\node[label={[xshift=-4pt,yshift=3pt]-90:$-1$}] (y3) at ($ (u3) + (-80:0.9) $) {};
\node[label={-90:$1$}] (y4) at ($ (u1) + (-100:0.9) $) {};
\path[edge] (u1) -- (u2) -- (u3);
\path[edge] (x1) -- (u1);
\path[edge] (x2) -- (u2) -- (x4);
\path[edge] (x3) -- (u3);
\path[edge] (x1) -- (y1) -- (x2) -- (y2) -- (x3) -- (y3) -- (x4) -- (y4) -- (x1);
\end{tikzpicture}
\qquad
\begin{tikzpicture}
\tikzstyle{edge}=[draw,thick]
\tikzstyle{every node}=[draw, circle, fill=blue!50!white, inner sep=1.5pt]
\node[label={[yshift=-1pt]90:$2$}] (u1) at (-1, 0.0) {};
\node[label={[yshift=-1pt,xshift=5pt]90:$0$}] (u2) at (-0, 0) {};
\node[label={[xshift=-4pt,yshift=-5pt]90:$-2$}] (u3) at (1, 0.0) {};
\node[label={-180:$0$}] (x1) at ($ (u1) + (180:0.7) $) {};
\node[label={90:$0$}] (x2) at ($ (u2) + (90:0.7) $) {};
\node[label={0:$0$}] (x3) at ($ (u3) + (0:0.7) $) {};
\node[label={-90:$0$}] (x4) at ($ (u2) + (-90:0.7) $) {};
\node[label={[xshift=-4pt,yshift=-4pt]90:$-1$}] (y1) at ($ (u1) + (100:0.9) $) {};
\node[label={90:$1$}] (y2) at ($ (u3) + (80:0.9) $) {};
\node[label={-90:$1$}] (y3) at ($ (u3) + (-80:0.9) $) {};
\node[label={[xshift=-4pt,yshift=3pt]-90:$-1$}] (y4) at ($ (u1) + (-100:0.9) $) {};
\path[edge] (u1) -- (u2) -- (u3);
\path[edge] (x1) -- (u1);
\path[edge] (x2) -- (u2) -- (x4);
\path[edge] (x3) -- (u3);
\path[edge] (x1) -- (y1) -- (x2) -- (y2) -- (x3) -- (y3) -- (x4) -- (y4) -- (x1);
\end{tikzpicture}
\caption{$\altan(P_3, H_3)$}
\label{fig:ex1d}
\end{subfigure}

\vspace{\baselineskip}
\begin{subfigure}[b]{\textwidth}
\centering
\begin{tikzpicture}
\tikzstyle{edge}=[draw,thick]
\tikzstyle{every node}=[draw, circle, fill=blue!50!white, inner sep=1.5pt]
\node[label={50:$1$}] (u1) at (-1, 0.0) {};
\node[label={[yshift=-2pt]90:$0$}] (u2) at (-0, 0) {};
\node[label={[xshift=-1pt,yshift=-2pt]140:$-1$}] (u3) at (1, 0.0) {};
\node[label={-120:$0$}] (x1) at ($ (u1) + (-80:0.8) $) {};
\node[label={90:$0$}] (x2) at ($ (u1) + (80:0.8) $) {};
\node[label={90:$0$}] (x3) at ($ (u3) + (100:0.8) $) {};
\node[label={-60:$0$}] (x4) at ($ (u3) + (-100:0.8) $) {};
\node[label={[xshift=-4pt,yshift=-4pt]90:$-1$}] (y1) at ($ (u1) + (180:0.7) $) {};
\node[label={90:$0$}] (y2) at ($ (u2) + (90:1.0) $) {};
\node[label={90:$1$}] (y3) at ($ (u3) + (0:0.7) $) {};
\node[label={-90:$0$}] (y4) at ($ (u2) + (-90:1.0) $) {};
\path[edge] (u1) -- (u2) -- (u3);
\path[edge] (x1) -- (u1) -- (x2);
\path[edge] (x3) -- (u3) -- (x4);
\path[edge] (x1) -- (y1) -- (x2) -- (y2) -- (x3) -- (y3) -- (x4) -- (y4) -- (x1);
\end{tikzpicture}
\quad
\begin{tikzpicture}
\tikzstyle{edge}=[draw,thick]
\tikzstyle{every node}=[draw, circle, fill=blue!50!white, inner sep=1.5pt]
\node[label={50:$0$}] (u1) at (-1, 0.0) {};
\node[label={[yshift=-2pt]90:$0$}] (u2) at (-0, 0) {};
\node[label={140:$0$}] (u3) at (1, 0.0) {};
\node[label={-120:$0$}] (x1) at ($ (u1) + (-80:0.8) $) {};
\node[label={90:$0$}] (x2) at ($ (u1) + (80:0.8) $) {};
\node[label={90:$0$}] (x3) at ($ (u3) + (100:0.8) $) {};
\node[label={-60:$0$}] (x4) at ($ (u3) + (-100:0.8) $) {};
\node[label={[xshift=-4pt,yshift=-4pt]90:$-1$}] (y1) at ($ (u1) + (180:0.7) $) {};
\node[label={90:$1$}] (y2) at ($ (u2) + (90:1.0) $) {};
\node[label={[xshift=2pt,yshift=-4pt]90:$-1$}] (y3) at ($ (u3) + (0:0.7) $) {};
\node[label={-90:$1$}] (y4) at ($ (u2) + (-90:1.0) $) {};
\path[edge] (u1) -- (u2) -- (u3);
\path[edge] (x1) -- (u1) -- (x2);
\path[edge] (x3) -- (u3) -- (x4);
\path[edge] (x1) -- (y1) -- (x2) -- (y2) -- (x3) -- (y3) -- (x4) -- (y4) -- (x1);
\end{tikzpicture}
\quad
\begin{tikzpicture}
\tikzstyle{edge}=[draw,thick]
\tikzstyle{every node}=[draw, circle, fill=blue!50!white, inner sep=1.5pt]
\node[label={50:$0$}] (u1) at (-1, 0.0) {};
\node[label={[yshift=-2pt]90:$0$}] (u2) at (-0, 0) {};
\node[label={140:$0$}] (u3) at (1, 0.0) {};
\node[label={[xshift=2pt]-120:$-1$}] (x1) at ($ (u1) + (-80:0.8) $) {};
\node[label={90:$1$}] (x2) at ($ (u1) + (80:0.8) $) {};
\node[label={[xshift=-4pt,yshift=-4pt]90:$-1$}] (x3) at ($ (u3) + (100:0.8) $) {};
\node[label={-60:$1$}] (x4) at ($ (u3) + (-100:0.8) $) {};
\node[label={90:$0$}] (y1) at ($ (u1) + (180:0.7) $) {};
\node[label={90:$0$}] (y2) at ($ (u2) + (90:1.0) $) {};
\node[label={90:$0$}] (y3) at ($ (u3) + (0:0.7) $) {};
\node[label={-90:$0$}] (y4) at ($ (u2) + (-90:1.0) $) {};
\path[edge] (u1) -- (u2) -- (u3);
\path[edge] (x1) -- (u1) -- (x2);
\path[edge] (x3) -- (u3) -- (x4);
\path[edge] (x1) -- (y1) -- (x2) -- (y2) -- (x3) -- (y3) -- (x4) -- (y4) -- (x1);
\end{tikzpicture}
\caption{$\altan(P_3, H_4)$}
\label{fig:ex1e}
\end{subfigure}
\caption{Kernel eigenvectors of (a)  the path graph $P_3$ and (b)\,--\,(e) altans for different
choices of attachment set.
The vectors are chosen to be orthogonal, with integer entries and  distinct 
symmetries in the point group of the altan graph.
In (b), the altan has a vertical line of reflection and the kernel eigenvector is
$(-)$, i.e.\ anti-symmetric.
In (c) to (e), the altan has horizontal and vertical lines of reflection, and the
various kernel eigenvectors have symmetries as follows:
(c) $(-, +)$; (d) $(-, -)$ and $(+, -)$; (e) $(+, -)$, $(+, +)$ and $(-, -)$,
where the $+/-$ signs indicate symmetry or anti-symmetry with respect to horizontal and 
vertical lines of reflection, respectively.}
\label{fig:example}
\end{figure}
\end{example}

\subsection{Patches}
Most applications of the altan construction in the chemical literature 
are expansions of a particular kind of polycyclic chemical graph. We
call this a \emph{patch}.

\begin{definition}
A \emph{patch} $\Pi$ is a sub-cubic $2$-connected plane graph that has every
degree-$2$ vertex incident with the distinguished outer face.
The \emph{natural attachment set} of $\Pi$
is the $h$-tuple $H$ that contains all the degree-$2$ vertices of $\Pi$ in the order
induced by clockwise traversal of the perimeter (i.e.\ the outer face).
When $H$ is the natural attachment set of a patch $\Pi$, we can simplify
the notation $\altan(\Pi, H)$ to $\altan(\Pi)$.
\end{definition}

\begin{definition}
A patch that contains only hexagonal interior
faces is called a \emph{fusene}. 
A fusene patch that is a simply connected subgraph of the hexagonal tessellation of the plane is a 
{\it benzenoid}.  Every benzenoid is a fusene, but not every
fusene is a benzenoid.
\end{definition}

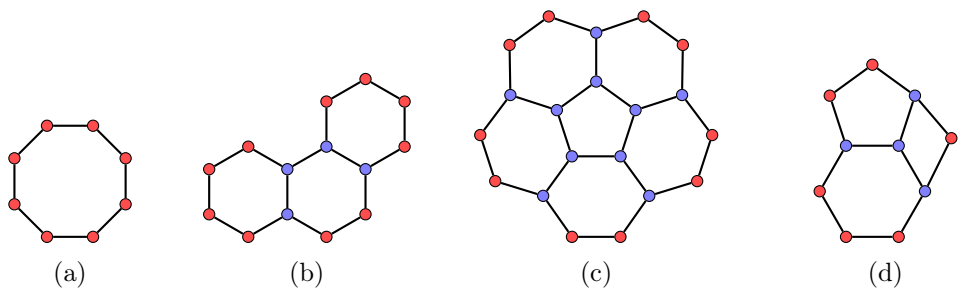
\begin{figure}[!htbp]
\centering
\begin{subfigure}[b]{0.15\textwidth}
\centering
\begin{tikzpicture}
\tikzstyle{edge}=[draw,thick]
\tikzstyle{every node}=[draw, circle, fill=blue!50!white, inner sep=1.5pt]
\foreach \k in {0,1,...,7} {
\node[fill=red!70!white] (u\k) at ({22.5 + 45 * \k}:0.8) {};
}
\path[edge] (u0) -- (u1) -- (u2) -- (u3) -- (u4) -- (u5) -- (u6) -- (u7) -- (u0);
\end{tikzpicture}
\caption{}
\end{subfigure}
\begin{subfigure}[b]{0.24\textwidth}
\centering
\begin{tikzpicture}
\tikzstyle{edge}=[draw,thick]
\tikzstyle{every node}=[draw, circle, fill=blue!50!white, inner sep=1.5pt]
\foreach \k in {0,1,2,3} {
\node[fill=red!70!white] (u\k) at ({90 + 60 * \k}:0.6) {};
}
\foreach \k in {4,5} {
\node (u\k) at ({90 + 60 * \k}:0.6) {};
}
\foreach \k in {3,4} {
\node[fill=red!70!white] (v\k) at ($ ({90 + 60 * \k}:0.6) + ({0.6 * sqrt(3)},0) $) {};
}
\foreach \k in {0,5} {
\node (v\k) at ($ ({90 + 60 * \k}:0.6) + ({0.6 * sqrt(3)},0) $) {};
}
\foreach \k in {0,1,4,5} {
\node[fill=red!70!white] (w\k) at ($ ({90 + 60 * \k}:0.6) + ({0.6 * sqrt(3)},0) + (60:{0.6 * sqrt(3)}) $) {};
}
\path[edge] (u0) -- (u1) -- (u2) -- (u3) -- (u4) -- (u5) -- (u0);
\path[edge] (u4) -- (v3) -- (v4) -- (v5) -- (v0) -- (u5);
\path[edge] (v5) -- (w4) -- (w5) -- (w0) -- (w1) -- (v0);
\end{tikzpicture}
\caption{}
\end{subfigure}
\begin{subfigure}[b]{0.24\textwidth}
\centering
\begin{tikzpicture}
\tikzstyle{edge}=[draw,thick]
\tikzstyle{every node}=[draw, circle, fill=blue!50!white, inner sep=1.5pt]
\foreach \k in {0,1,2,3,4} {
\node (u\k) at ({90 + 72 * \k}:0.55) {};
\node (v\k) at ({90 + 72 * \k}:1.2) {};
}
\foreach \k in {0,1,2,3,4} {
\node[fill=red!70!white] (a\k) at ({90 + 36 - 12 + 72 * \k}:1.55) {};
\node[fill=red!70!white] (b\k) at ({90 + 36 + 12 + 72 * \k}:1.55) {};
}
\path[edge] (u0) -- (u1) -- (u2) -- (u3) -- (u4) -- (u0);
\path[edge] (u0) -- (v0) -- (a0) -- (b0) -- (v1) -- (u1);
\path[edge] (v1) -- (a1) -- (b1) -- (v2) -- (u2);
\path[edge] (v2) -- (a2) -- (b2) -- (v3) -- (u3);
\path[edge] (v3) -- (a3) -- (b3) -- (v4) -- (u4);
\path[edge] (v4) -- (a4) -- (b4) -- (v0);
\end{tikzpicture}
\caption{}
\end{subfigure}
\begin{subfigure}[b]{0.24\textwidth}
\centering
\begin{tikzpicture}
\tikzstyle{edge}=[draw,thick]
\tikzstyle{every node}=[draw, circle, fill=blue!50!white, inner sep=1.5pt]
\node[fill=red!70!white] (u0) at (0, 0) {};
\node[fill=red!70!white] (u1) at ($ (u0) + (0:0.7) $) {};
\node (u2) at ($ (u1) + (60:0.7) $) {};
\node (u3) at ($ (u2) + (120:0.7) $) {};
\node (u4) at ($ (u2) + (120:0.7) + (180:0.7) $) {};
\node[fill=red!70!white] (u5) at ($ (u4) + (240:0.7) $) {};
\node[fill=red!70!white] (v0) at ($ (u3) + (0:0.7) + (90:0.1) $) {};
\node[fill=red!70!white] (w0) at ($ (u4) + (108:0.7) $) {};
\node[fill=red!70!white] (w1) at ($ (w0) + (36:0.7) $) {};
\node (w2) at ($ (w1) + (-36:0.7) $) {};
\path[edge] (u0) -- (u1) -- (u2) -- (v0) -- (w2) -- (w1) -- (w0) -- (u4) -- (u5) -- (u0);
\path[edge] (u2) -- (u3) -- (u4);
\path[edge] (u3) -- (w2);
\end{tikzpicture}
\caption{}
\end{subfigure}
\caption{Examples of patches. The vertices of the natural attachment set are coloured red.
Patch (b) is a fusene (and a benzenoid).  Patches (b) and (c) are fullerene patches  in the sense that they can appear as induced subgraphs in fullerenes \cite{Graver2010,Graver2014}.}
\label{fig:patches}
\end{figure}

Figure~\ref{fig:patches} contains several examples of patches. A patch is a chemical graph
and it corresponds to an unsaturated hydrocarbon of chemical formula $\mathrm{C}_n \mathrm{H}_{3n - 2m}$
where $n$ is the number of vertices and $m$ is the number of edges in the patch.
When  the altan construction is performed on a patch (which is, by definition, embedded in the plane), the
new perimeter cycle can be drawn in the outer face of the patch and connected to the natural attachment
set in such a way that no crossings are introduced. The graph
obtained is evidently plane.
It is easy to see that the altan of a patch is a patch.

\begin{example}
\label{ex:ex2}
The molecular graph of pentalene is an example of a small patch, consisting in this case of two fused pentagons. 
Pentalene has a unique kernel eigenvector/NBMO which is illustrated in Figure~\ref{fig:example2}(\subref{fig:example2a}).
Altan-pentalene has nullity $2$, and a pair of eigenvectors/NBMOs spanning the kernel is shown in Figure~\ref{fig:example2}(\subref{fig:example2b}).
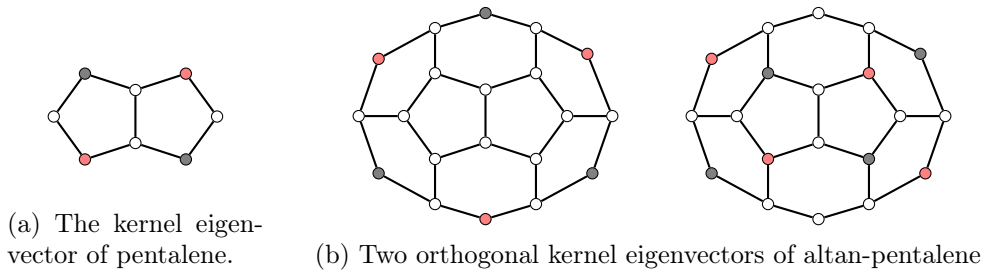
\begin{figure}[!b]
\centering
\begin{subfigure}[b]{0.22\linewidth}
\centering
\begin{tikzpicture}
\tikzstyle{edge}=[draw,thick]
\tikzstyle{every node}=[draw, circle, fill=white, inner sep=1.5pt]
\tikzstyle{pozit}=[fill=black!50!white]
\tikzstyle{negat}=[fill=red!50!white]
\node[] (u0) at ({36 + 72 * 0}:0.6) {};
\node[pozit] (u1) at ({36 + 72 * 1}:0.6) {};
\node[] (u2) at ({36 + 72 * 2}:0.6) {};
\node[negat] (u3) at ({36 + 72 * 3}:0.6) {};
\node[] (u4) at ({36 + 72 * 4}:0.6) {};
\node[draw=none,fill=none] at (-90:0.95) {}; 
\node[negat] (v1) at ($ ({0.6 * cos(36) + 0.6 * sin(36) / tan(36)}, 0) + ({180 - 36 - 72}: {0.6 * sin(36) / sin(36)}) $) {};
\node[] (v2) at ($ ({0.6 * cos(36) + 0.6 * sin(36) / tan(36)}, 0) + ({180 - 36 - 2 * 72}: {0.6 * sin(36) / sin(36)}) $) {};
\node[pozit] (v3) at ($ ({0.6 * cos(36) + 0.6 * sin(36) / tan(36)}, 0) + ({180 - 36 - 3 * 72}: {0.6 * sin(36) / sin(36)}) $) {};
\path[edge] (u0) -- (u1) -- (u2) -- (u3) -- (u4) -- (u0);
\path[edge] (u0) -- (v1) -- (v2) -- (v3) -- (u4);
\end{tikzpicture}
\caption{The kernel eigenvector of pentalene.}
\label{fig:example2a}
\end{subfigure}
\begin{subfigure}[b]{0.65\linewidth}
\centering
\begin{tikzpicture}
\tikzstyle{edge}=[draw,thick]
\tikzstyle{every node}=[draw, circle, fill=white, inner sep=1.5pt]
\tikzstyle{pozit}=[fill=black!50!white]
\tikzstyle{negat}=[fill=red!50!white]
\node[] (u0) at ({36 + 72 * 0}:0.6) {};
\node[] (u1) at ({36 + 72 * 1}:0.6) {};
\node[] (u2) at ({36 + 72 * 2}:0.6) {};
\node[] (u3) at ({36 + 72 * 3}:0.6) {};
\node[] (u4) at ({36 + 72 * 4}:0.6) {};
\node[] (v1) at ($ ({0.6 * cos(36) + 0.6 * sin(36) / tan(36)}, 0) + ({180 - 36 - 72}: {0.6 * sin(36) / sin(36)}) $) {};
\node[] (v2) at ($ ({0.6 * cos(36) + 0.6 * sin(36) / tan(36)}, 0) + ({180 - 36 - 2 * 72}: {0.6 * sin(36) / sin(36)}) $) {};
\node[] (v3) at ($ ({0.6 * cos(36) + 0.6 * sin(36) / tan(36)}, 0) + ({180 - 36 - 3 * 72}: {0.6 * sin(36) / sin(36)}) $) {};
\node[] (uu1) at ($ (u1) + ({90}:0.6)$) {};
\node[] (uu2) at ($ (u2) + (180:0.6)$) {};
\node[] (uu3) at ($ (u3) + ({-90}:0.6)$) {};
\node[] (vv1) at ($ (v1) + ({90}:0.6)$) {};
\node[] (vv2) at ($ (v2) + (0:0.6)$) {};
\node[] (vv3) at ($ (v3) + ({-90}:0.6)$) {};
\node[negat] (aa1) at ($ (uu1)!0.5!(uu2) + (144:0.3) $) {};
\node[pozit] (aa2) at ($ (uu2)!0.5!(uu3) + (216:0.3) $) {};
\node[pozit] (bb1) at ($ (uu1)!0.5!(vv1) + (90:0.2) $) {};
\node[negat] (bb2) at ($ (vv1)!0.5!(vv2) + (54:0.3) $) {};
\node[pozit] (bb3) at ($ (vv2)!0.5!(vv3) + (-36:0.3) $) {};
\node[negat] (bb4) at ($ (uu3)!0.5!(vv3) + (-90:0.2) $) {};
\path[edge] (u0) -- (u1) -- (u2) -- (u3) -- (u4) -- (u0);
\path[edge] (u0) -- (v1) -- (v2) -- (v3) -- (u4);
\path[edge] (uu1) -- (aa1) -- (uu2) -- (aa2) -- (uu3) -- (bb4) -- (vv3) -- (bb3) -- (vv2) -- (bb2) --
(vv1) -- (bb1) -- (uu1);
\path[edge] (u1) -- (uu1);
\path[edge] (u2) -- (uu2);
\path[edge] (u3) -- (uu3);
\path[edge] (v1) -- (vv1);
\path[edge] (v2) -- (vv2);
\path[edge] (v3) -- (vv3);
\end{tikzpicture}
\qquad
\begin{tikzpicture}
\tikzstyle{edge}=[draw,thick]
\tikzstyle{every node}=[draw, circle, fill=white, inner sep=1.5pt]
\tikzstyle{pozit}=[fill=black!50!white]
\tikzstyle{negat}=[fill=red!50!white]
\node[] (u0) at ({36 + 72 * 0}:0.6) {};
\node[pozit] (u1) at ({36 + 72 * 1}:0.6) {};
\node[] (u2) at ({36 + 72 * 2}:0.6) {};
\node[negat] (u3) at ({36 + 72 * 3}:0.6) {};
\node[] (u4) at ({36 + 72 * 4}:0.6) {};
\node[negat] (v1) at ($ ({0.6 * cos(36) + 0.6 * sin(36) / tan(36)}, 0) + ({180 - 36 - 72}: {0.6 * sin(36) / sin(36)}) $) {};
\node[] (v2) at ($ ({0.6 * cos(36) + 0.6 * sin(36) / tan(36)}, 0) + ({180 - 36 - 2 * 72}: {0.6 * sin(36) / sin(36)}) $) {};
\node[pozit] (v3) at ($ ({0.6 * cos(36) + 0.6 * sin(36) / tan(36)}, 0) + ({180 - 36 - 3 * 72}: {0.6 * sin(36) / sin(36)}) $) {};
\node[] (uu1) at ($ (u1) + ({90}:0.6)$) {};
\node[] (uu2) at ($ (u2) + (180:0.6)$) {};
\node[] (uu3) at ($ (u3) + ({-90}:0.6)$) {};
\node[] (vv1) at ($ (v1) + ({90}:0.6)$) {};
\node[] (vv2) at ($ (v2) + (0:0.6)$) {};
\node[] (vv3) at ($ (v3) + ({-90}:0.6)$) {};
\node[negat] (aa1) at ($ (uu1)!0.5!(uu2) + (144:0.3) $) {};
\node[pozit] (aa2) at ($ (uu2)!0.5!(uu3) + (216:0.3) $) {};
\node[] (bb1) at ($ (uu1)!0.5!(vv1) + (90:0.2) $) {};
\node[pozit] (bb2) at ($ (vv1)!0.5!(vv2) + (54:0.3) $) {};
\node[negat] (bb3) at ($ (vv2)!0.5!(vv3) + (-36:0.3) $) {};
\node[] (bb4) at ($ (uu3)!0.5!(vv3) + (-90:0.2) $) {};
\path[edge] (u0) -- (u1) -- (u2) -- (u3) -- (u4) -- (u0);
\path[edge] (u0) -- (v1) -- (v2) -- (v3) -- (u4);
\path[edge] (uu1) -- (aa1) -- (uu2) -- (aa2) -- (uu3) -- (bb4) -- (vv3) -- (bb3) -- (vv2) -- (bb2) --
(vv1) -- (bb1) -- (uu1);
\path[edge] (u1) -- (uu1);
\path[edge] (u2) -- (uu2);
\path[edge] (u3) -- (uu3);
\path[edge] (v1) -- (vv1);
\path[edge] (v2) -- (vv2);
\path[edge] (v3) -- (vv3);
\end{tikzpicture}
\caption{Two orthogonal kernel eigenvectors of altan-pentalene.}
\label{fig:example2b}
\end{subfigure}
\caption{The molecular graph of pentalene has nullity $1$, whilst the molecular graph of altan-pentalene has nullity $2$.
Eigenvector entries are colour-coded: white corresponds to value $0$, black to $1$ and red to $-1$ in the unnormalised vector.
In the normalised vectors, entries are multiplied by $\sqrt{1/4}, \sqrt{1/6}$ and $\sqrt{1/8}$, respectively.
}
\label{fig:example2}
\end{figure}
\end{example}

Straightforward considerations based on the Euler theorem give
relations that constrain the composition of  the altan of a patch in terms 
of the faces that it contains, and of 
the composition of the cyclic strip of faces that is added 
to the patch at each stage of altanisation.

\begin{proposition}
Let $\Pi$ be a patch. Let $f_r$ denote the number of faces of length $r$  in $\altan(\Pi)$, excluding the infinite (unbounded) face.
Then
\begin{equation}
\label{eq:eulereq1}
\sum_{r \geq 3} (6 - r) f_r = 3f_3 + 2f_4 + f_5 - f_7 -2f_8 - \cdots = 6.
\end{equation}
Let $n_k$ be the number of vertices of degree $k$ in  $\Pi$ and let $n^{b}_k$ be the
number of vertices of degree $k$ on the perimeter of $\Pi$.
Let $\widetilde{f}_r$ denote the number of new faces of length $r$ that are added to the patch by the altan
construction. Then
\begin{equation}
\label{eq:altanfaces}
\sum_{r \geq 3} (6 - r) \widetilde{f}_r = 3\widetilde{f}_3 + 2\widetilde{f}_4 + \widetilde{f}_5 - \widetilde{f}_7 -2\widetilde{f}_8 - \cdots = n_2 - n^{b}_3
\end{equation}
with $\widetilde{f}_3  = \widetilde{f}_4 = 0$.
\end{proposition}

\begin{proof}
Let $f$ be the number of faces of $\altan(\Pi)$, excluding the infinite face. Then $f = f_3 + f_4 + f_5 + \cdots = \sum_{r \geq 3} f_r$.
Let $h$ be the size of the natural attachment set. Note that $h = n^b_2 = n_2$.

By applying the Euler formula on the altan   $\altan(\Pi)$ we obtain
\begin{equation}
\label{eq:proofef}
(n_2 + n_3 + 2h) - (n_2 + \tfrac{3}{2}n_3 + 3h) + (\sum_{r \geq 3} f_r + 1) = 2,
\end{equation}
and the Handshaking Lemma for faces gives
\begin{equation}
\label{eq:proofhlf}
\sum_{r \geq 3} r f_r + 2h = 2n_2 + 3n_3 + 6h.
\end{equation}
Combining equations \eqref{eq:proofef} and \eqref{eq:proofhlf} gives  \eqref{eq:eulereq1}.

By applying the Euler formula on the patch $\Pi$ we obtain
\begin{equation}
\label{eq:proofef2}
(n_2 + n_3) - (n_2 + \tfrac{3}{2}n_3) + (\sum_{r \geq 3} (f_r - \widetilde{f}_r) + 1) = 2,
\end{equation}
while the Handshaking Lemma for faces gives 
\begin{equation}
\label{eq:proofhlf2}
\sum_{r \geq 3} r (f_r - \widetilde{f}_r) + (n^{b}_3 + n^{b}_2) = 2n_2 + 3n_3.
\end{equation}
Combining equations \eqref{eq:eulereq1}, \eqref{eq:proofef2} and \eqref{eq:proofhlf2} gives \eqref{eq:altanfaces}, and by definition of the altan, it is clear that the length of each new face is at least $5$, hence $\widetilde{f}_3  = \widetilde{f}_4 = 0$.
\end{proof}

If $n^{b}_3 = 0$ then $\widetilde{f}_5 = n_2$ and $\widetilde{f}_7 = \widetilde{f}_8 = \cdots = 0$.
Note that in the iterated altan, the boundary code is $(23)^h$ for all altanisations after the first, and
$\widetilde{f}_6 = h$ while $\widetilde{f}_7 = \widetilde{f}_8 = \widetilde{f}_9 = \cdots = 0$.
Iteration will lead to a graphitic tube corresponding to a capped nanotube (a nanocone with $6$ disclinations in the nomenclature of Klein and Balaban \cite{Klein2006}).

Reasoning of this type is often used to describe the composition of various types of generalised patch\cite{Brinkmann2005,Gutman2014a}. 
In \cite{Gutman2014a} it is used to specify the composition of the altans of benzenoids. 
In the familiar geographical analogy where the perimeter of a benzenoid is likened to the coastline of an island continent, 
regions of increasing concavity are described as fissures, bays, coves and fjords \cite{Gutman1989}, corresponding respectively 
to degree sequences on the boundary of $232$, $2332$, $23332$, and $233332$. 
The counts of the four types of feature are denoted $b_1$, $b_2$, $b_3$, and $b_4$, respectively. 
The \emph{bay number} of the benzenoid is then $b = b_2 +2b_3 + 3b_4$, 
and the number of adjacencies of type $22$ on the perimeter of the benzenoid is 
$n_{22}=6+b$ \cite{Gutman1989}, 
which is also the number of pentagons added to the benzenoid by altanisation. 
Each fissure, bay, cove or fjord adds respectively a hexagon, heptagon, octagon or 
nonagon in the altan \cite{Gutman2014a}. 
Hence, when $G$ is a benzenoid, in our notation we have
\begin{equation}
\widetilde{f}_5 -\widetilde{f}_7 - 2\widetilde{f}_8 -3\widetilde{f}_9 = 6,
\end{equation}
which is the special case of \eqref{eq:altanfaces} for a benzenoid parent.
Notice also that \emph{convex} benzenoids \cite{gutman2012} have $b =0$, so that the strip of faces added by the 
altan operation to a convex benzenoid consists of $6$ pentagons and $b_1$ hexagons. 

Simple counting considerations can also be used to derive a useful relation between nullity and size of 
the attachment set in the case where the patch is bipartite:

\begin{theorem}
\label{thm:bipartitepatchnullity}
Let $\Pi$ be a bipartite patch and let $H$ be the natural attachment set. Then
$$\eta(\Pi) \equiv |H| \pmod{2}.$$
\end{theorem}

\begin{proof}
Let $V(\Pi)$ and $E(\Pi)$ be the vertex and edge sets of the patch $\Pi$, respectively.
Let $n_2$ and $n_3$ denote the numbers of degree-$2$ and degree-$3$ vertices of the patch.
Moreover, let $n_3^{i}$ and $n_3^{b}$ denote the numbers of internal and boundary degree-$3$ vertices of
the patch, respectively. Note that $n_3 = n_3^{i} + n_3^{b}$ and $n = |V(\Pi)| = n_2 + n_3$.
Furthermore, let $p$ denote the length of the perimeter. Note that $p = n_2 + n_3^{b}$.
Observe that $h = |H| = n_2$.

Since $\Pi$ is a bipartite graph, all its cycles are of even length and therefore $p$ is even.
From $p = n_2 + n_3^{b} \equiv 0 \pmod{2}$ it follows that $n_2 \equiv n_3^{b} \pmod{2}$.
From $n = n_2 + n_3^{b} + n_3^{i}$ it follows that $n \equiv n_3^{i} \pmod{2}$.
By the Handshaking Lemma, $2|E(\Pi)| = 3n_3 + 2n_2$ and therefore $n_3 \equiv 0 \pmod{2}$. 
Hence, $n_3^{b} \equiv n_3^{i} \pmod{2}$. 
By the Pairing Theorem, $n \equiv \eta(\Pi) \pmod{2}$.
Summarising the above observations we obtain
\[
 h \equiv n_3^{i} \equiv n_3^{b} \equiv n \equiv \eta(\Pi) \pmod{2}. \qedhere
\]
\end{proof}

\begin{corollary}
\label{cor:corollary3}
Let $\mathcal{B}$ be a benzenoid on $n$ vertices, let $\eta = \eta(\mathcal{B})$ be the nullity of $\mathcal{B}$ and let $h$ be the size of the natural attachment set.
Then $$h \equiv n \equiv \eta \pmod {2}.$$
\end{corollary}

In the proof of Theorem~\ref{thm:bipartitepatchnullity} we have seen that $n_2 \equiv n_3^{b} \pmod{2}$ for a bipartite patch.
This implies that the altan of a bipartite patch has a constraint on the numbers of faces of odd lengths. In this case,
 the RHS of Equation~\eqref{eq:altanfaces} is even and 
$\widetilde{f}_5 - \widetilde{f}_7 - 3\widetilde{f}_9 - 5\widetilde{f}_{11} - \cdots \equiv 0 \pmod{2}.$

Although we have concentrated here on altanisation 
of patched as models for simply
connected polycyclic aromatic hydrocarbon (PAH) nanostructures,
the chemical literature also includes examples of altans of structures with holes, 
such as kekulene~\cite{Monaco2015}, where altanisation proceeds on the outer perimeter only. 
Altans (inner and outer) of general coronoid structures such as kekulene have been considered in \cite{Basic2016}
under a comprehensive scheme for altanisation of perforated patches.

\section{Main results: Theorems for nullities of altans}
\label{sec:mainsec}

We can now state the main  results
of our mathematical investigation of the nullity of altans.

\begin{definition}
\label{def:special}
Let $(G', H')$ and vertex labeling be as in Definition~\ref{def:altan} (see also Figure~\ref{fig:altan}). The vector $\special$, defined as
\begin{equation}
\label{eq:ubiq}
\special(u) = \begin{cases}
\phantom{-}1 & \text{if }u = y_i \text{ and } i \text{ even}, \\
-1 & \text{if }u = y_i \text{ and } i \text{ odd}, \\
\phantom{-}0 & \text{otherwise},
\end{cases}
\end{equation}
will be called the \emph{special vector} of $\altan(G, H)$, or sometimes, for short, the special one.
\end{definition}

An example of this vector has already been encountered in Example~\ref{ex:ex2} (see the left panel of Figure~\ref{fig:example2}(\subref{fig:example2b})).

\begin{lemma}
\label{lem:special}
Let $(G', H')$ be as in Definition~\ref{def:special}.
If $h = |H|$ is even, then the special vector
$\special$ is a kernel eigenvector of $\altan(G, H)$.
\end{lemma}

\begin{proof}
Apply the local condition \eqref{eq:local} with $\lambda = 0$. If $h$ is even then $\{ \special(u) : u \in N(x_i) \} = \{-1, 0, 1\}$
for all $1 \leq i \leq h$. For all other vertices $w \neq x_i$ we have $\{ \special(u) : u \in N(w) \} = \{0\}$.
\end{proof}

The following theorem was already proved by Gutman in \cite{Gutman2014b}
using the Sachs Theorem. Here, we give an elementary proof.

\begin{theorem}[Gutman \cite{Gutman2014b}]
\label{thm:minigutman}
Let $G$ be a graph and let $H$ be an even attachment set.
Then $\eta(\altan(G, H)) \geq 1$.
\end{theorem}

\begin{proof}
Follows directly from Lemma~\ref{lem:special}.
\end{proof}

\begin{theorem}
\label{thm:armes1}
Let $G$ be a graph and let $H$ be an attachment set. Then 
\begin{equation}
\eta(\altan(G, H)) \geq \eta(G).
\end{equation}
\end{theorem}

\begin{proof}
Let $(G', H') = \altan(G, H)$ and let $G'$ be labelled as in Definition~\ref{def:altan}. Let $\x \in \ker(G)$.
We will extend the vector $\x$ to  $\widetilde{\x} \in \ker(G')$.
Let $\widetilde{\x}(u) = \x(u)$ for $u \in V(G)$ and let $\widetilde{\x}(x_i) = 0$ for all $1 \leq i \leq h$.
Note that all vertices $v \in V(G)$ and all vertices $y_i$ for $1 \leq i \leq h$ satisfy the local condition
\eqref{eq:local} with $\lambda = 0$.

Now, let $\widetilde{\x}(y_h) = t \in \mathbb{R}$. By pivoting at vertices $x_1, x_2, \ldots, x_{h - 1}$, respectively, we obtain
\begin{align*}
\widetilde{\x}(y_1) & = \hspace{8.5pt} -t - \widetilde{\x}(v_1) \\
\widetilde{\x}(y_2) & = \hspace{8.5pt} \phantom{-}t + \widetilde{\x}(v_1) - \widetilde{\x}(v_2) \\
\widetilde{\x}(y_3) & = \hspace{8.5pt} -t - \widetilde{\x}(v_1) + \widetilde{\x}(v_2) - \widetilde{\x}(v_3)  \\
& \ldots \\
\widetilde{\x}(y_{h - 1}) & = \begin{cases}
-t - \widetilde{\x}(v_1) + \widetilde{\x}(v_2) - \widetilde{\x}(v_3) + \cdots - \widetilde{\x}(v_{h-1}) & \text{if } h \text{ even}, \\
\phantom{-}t + \widetilde{\x}(v_1) - \widetilde{\x}(v_2) + \widetilde{\x}(v_3) - \cdots - \widetilde{\x}(v_{h-1}) & \text{if } h \text{ odd}.
\end{cases}
\end{align*}
Let us define 
\begin{equation}
\mathcal{C}(\x) = \sum_{i=1}^{h} (-1)^{i} \x(v_i).
\end{equation}
It is easy to see that $\mathcal{C}$ is a linear functional. By pivoting at vertex $x_h$, we get the
following condition
\begin{equation}
\label{eq:extcond}
\begin{cases}
\mathcal{C}(\x) = 0 & \text{if } h \text{ even}, \\
t = \tfrac{1}{2}\mathcal{C}(\x) & \text{if }h \text{ odd}.
\end{cases}
\end{equation}

Now, let $\eta = \eta(G)$ and let $\{ \x^{(1)}, \x^{(2)}, \ldots, \x^{(\eta)} \}$ be a basis of $\ker (G)$.
In other words, $\ker(G) = \spn \{ \x^{(1)}, \x^{(2)}, \ldots, \x^{(\eta)} \}$ and vectors $\x^{(1)}, \x^{(2)}, \ldots, \x^{(\eta)}$
are linearly independent.

First, suppose that $h$ is odd. From \eqref{eq:extcond} it follows that each vector $\x^{(k)}$ can be extended to $\widetilde{\x}^{(k)}$
 by setting $t = \tfrac{1}{2}\mathcal{C}(\x^{(k)})$, as described above. By construction, $\{ \widetilde{\x}^{(1)}, \widetilde{\x}^{(2)}, \ldots, \widetilde{\x}^{(\eta)} \} \subseteq \ker(G')$.
It is easy to see that these vectors
are linearly independent. Therefore, $\eta(G') \geq \eta(G)$.

Suppose instead that $h$ is even. A kernel eigenvector $\x$ will be called \emph{extendable} if $\mathcal{C}(\x) = 0$, i.e., if it satisfies 
condition \eqref{eq:extcond}. We consider two cases:
\begin{enumerate}[label=(\roman*)]
\item\label{proof_case1} 
Suppose that vectors $\x^{(1)}, \x^{(2)}, \ldots, \x^{(\eta)}$ are all extendable. Then their extensions $\widetilde{\x}^{(1)}, \linebreak
\widetilde{\x}^{(2)}, \ldots, \widetilde{\x}^{(\eta)}$ are obtained as described above by choosing, say, $t = 0$.
It is easy to see that the vectors $\widetilde{\x}^{(1)}, 
\widetilde{\x}^{(2)}, \ldots, \widetilde{\x}^{(\eta)}$ are linearly independent. Moreover, the special vector $\special$ is yet another kernel eigenvector.
It is easy to check that $\special$ is linearly independent of all vectors $\widetilde{\x}^{(1)}, \widetilde{\x}^{(2)}, \ldots, \widetilde{\x}^{(\eta)}$ (by inspecting
the entry $y_h$). Hence, in this case we get $\eta(G') \geq \eta(G) + 1$.
\item 
Suppose that at least one of the vectors $\x^{(1)}, \x^{(2)}, \ldots, \x^{(\eta)}$ is \emph{not} extendable.
Without loss of generality, assume that $\x^{(1)}$ is not extendable. In other words, $\mathcal{C}(\x^{(1)}) \neq 0$.
We will replace the basis $\{\x^{(1)}, \x^{(2)}, \ldots, \x^{(\eta)}\}$ with an alternative basis
\[
\{ \x^{(1)}, \x^{(2)} + \lambda_2 \x^{(1)}, \x^{(3)} + \lambda_3 \x^{(1)}, \ldots, \x^{(\eta)} + \lambda_{\eta} \x^{(1)} \},
\]
where scalars $\lambda_k$ are chosen so that $\mathcal{C}(\x^{(k)} + \lambda_k \x^{(1)}) = 0$. Namely,
 $\lambda_k = -{\mathcal{C}(\x^{(k)})}/{\mathcal{C}(\x^{(1)})}$.

Let $\widetilde{\x}^{(2)}, \ldots, \widetilde{\x}^{(\eta)}$ be extensions of $\x^{(2)} + \lambda_2 \x^{(1)}, \ldots, \x^{(\eta)} + \lambda_{\eta} \x^{(1)}$,
respectively. It is easy to see that the vectors $\widetilde{\x}^{(2)}, \ldots, \widetilde{\x}^{(\eta)}$  are linearly independent. Moreover, the special vector $\special$ is independent
of them.  Again, we get $\eta(G') \geq \eta(G)$, as the special one compensates for the loss of $\x^{(1)}$. \qedhere
\end{enumerate}
\end{proof}

Note that the kernel eigenvector of pentalene (see Figure~\ref{fig:example2}(\subref{fig:example2a})) is extendable, and altan-pentalene is an example of 
case \ref{proof_case1} in the proof above.

\begin{theorem}
\label{thm:armes2}
Let $G$ be a graph and $H$  an attachment set. Then
\begin{equation}
\eta(\altan(G, H)) \leq \begin{cases}
\eta(G) + 2 & \text{if } h = |H| \text{ even}, \\
\eta(G) & \text{if } h = |H| \text{ odd}.
\end{cases}
\end{equation}
\end{theorem}

\par\noindent
To prepare for the proof of the above theorem, we state several technical lemmas.

\begin{lemma}
\label{techlem:a}
Let $(G', H')$ be as in Definition~\ref{def:altan}.
Let $\widetilde{\x} \in \ker (G')$. If $h$ is odd then $\widetilde{\x}(x_i) = 0$  for all $1 \leq i \leq h$.
\end{lemma}

\begin{proof}
Let $\widetilde{\x}(x_1) = t \in \mathbb{R}$. By pivoting at vertices $y_1, y_2, \ldots, y_{h - 1}$, respectively, we obtain
\begin{align*}
\widetilde{\x}(x_1) & = \widetilde{\x}(x_3) = \widetilde{\x}(x_5) = \cdots = \widetilde{\x}(x_h) = t, \\
\widetilde{\x}(x_2) & = \widetilde{\x}(x_4) = \widetilde{\x}(x_4) = \cdots = \widetilde{\x}(x_{h-1}) = -t. 
\end{align*}
By pivoting at vertex $y_{h}$ we obtain $t = 0$.
\end{proof}

\begin{lemma}
\label{techlem:b}
Let $(G', H')$ be as in Definition~\ref{def:altan}.
Let $\widetilde{\x} \in \ker (G')$ such that $\widetilde{\x}(u) = 0$ for all $u \in V(G)$ and
$\widetilde{\x}(x_i) = 0$ for all $1 \leq i \leq h$.
If $h$ is odd then $\widetilde{\x}(y_i) = 0$ for all $1 \leq i \leq h$.
\end{lemma}

\begin{proof}
Let $\widetilde{\x}(y_1) = t \in \mathbb{R}$. By pivoting at vertices $x_2, x_3, \ldots, x_{h}$, respectively, we obtain
\begin{align*}
\widetilde{\x}(y_1) & = \widetilde{\x}(y_3) = \widetilde{\x}(y_5) = \cdots = \widetilde{\x}(y_h) = t, \\
\widetilde{\x}(y_2) & = \widetilde{\x}(y_4) = \widetilde{\x}(y_4) = \cdots = \widetilde{\x}(y_{h-1}) = -t.
\end{align*}
By pivoting at vertex $x_{1}$ we obtain $t = 0$.
\end{proof}

\begin{lemma}
\label{techlem:c}
Let $(G', H')$ be as in Definition~\ref{def:altan}.
Let $\widetilde{\x} \in \ker (G')$. If $h$ is even then $\widetilde{\x}(x_i) = \widetilde{\x}(x_1)$ for all odd $i$
and $\widetilde{\x}(x_i) = -\widetilde{\x}(x_1)$ for all even $i$.
\end{lemma}
\begin{proof}
The proof is similar to that of Lemma~\ref{techlem:a} and is left to the reader.
\end{proof}

\begin{lemma}
\label{techlem:d}
Let $(G', H')$ be as in Definition~\ref{def:altan}.
Let $\widetilde{\x} \in \ker (G')$ such that $\widetilde{\x}(u) = 0$ for all $u \in V(G)$ and
$\widetilde{\x}(x_i) = 0$ for all $1 \leq i \leq h$.
If $h$ is even then $\widetilde{\x}(y_i) = \widetilde{\x}(y_1)$ for all odd $i$
and $\widetilde{\x}(y_i) = -\widetilde{\x}(y_1)$ for all even $i$.
\end{lemma}
\begin{proof}
The proof is similar to that of Lemma~\ref{techlem:b} and is left to the reader.
\end{proof}

\par\noindent
Now we can give the proof of Theorem~\ref{thm:armes2}.

\begin{proof}[Proof of Theorem~\ref{thm:armes2}]
Let $(G', H') = \altan(G, H)$ and let $G'$ be labelled as in Definition~\ref{def:altan}.

First, suppose that $h$ is odd.
Let $\xi = \eta(G')$ and let $\{ \widetilde{\x}^{(1)}, \widetilde{\x}^{(2)}, \ldots, \widetilde{\x}^{(\xi)} \}$ be a basis for $\ker(G')$.
A kernel eigenvector $\widetilde{\x}$ of $G'$ will be called \emph{contractible} if $\widetilde{\x}(x_i) = 0$ for all
$1 \leq i \leq h$. The \emph{contraction} of $\widetilde{\x}$ is the vector $\x$ of $G$ defined by $\x(u) = \widetilde{\x}(u)$ for $u \in V(G)$.
Note that $\x \in \ker(G)$ if $\widetilde{\x}$ is contractible.

By Lemma~\ref{techlem:a}, the vectors $\widetilde{\x}^{(1)}, \widetilde{\x}^{(2)}, \ldots, \widetilde{\x}^{(\xi)}$
are contractible.  Let $\x^{(1)}, \x^{(2)}, \ldots, \x^{(\xi)}$ be contractions of 
$\widetilde{\x}^{(1)}, \widetilde{\x}^{(2)}, \ldots, \widetilde{\x}^{(\xi)}$, respectively. We will prove that 
$\x^{(1)}, \x^{(2)}, \ldots, \x^{(\xi)}$ are linearly independent. For contradiction, suppose that they are not
linearly independent. Then there exist scalars $\mu_1, \mu_2, \ldots, \mu_\xi$ (at least one of which is non-zero) such that
\begin{equation}
\label{eq:linkombi}
\mu_1 \x^{(1)} + \mu_2 \x^{(2)} + \cdots + \mu_\xi \x^{(\xi)} = \mathbf{0}.
\end{equation}
Let us define 
\[
\X := \mu_1 \widetilde{\x}^{(1)} + \mu_2 \widetilde{\x}^{(2)} + \cdots + \mu_\xi \widetilde{\x}^{(\xi)}.
\]
Observe that $\X(u) = (\mu_1 \x^{(1)} + \mu_2 \x^{(2)} + \cdots + \mu_\xi \x^{(\xi)})(u)$ for all $u \in V(G)$. From 
\eqref{eq:linkombi} it follows that $\X(u) = 0$ for all $u \in V(G)$. By Lemma~\ref{techlem:b}, $\X = \mathbf{0}$. This
contradicts the fact that $\{ \widetilde{\x}^{(1)}, \widetilde{\x}^{(2)}, \ldots, \widetilde{\x}^{(\xi)}\}$ is a basis, so
$\x^{(1)}, \x^{(2)}, \ldots, \x^{(\xi)}$ are linearly independent and hence $\eta(G) \geq \eta(G')$.

Suppose instead that $h$ is even. Let $\xi = \eta(G')$. By Lemma~\ref{lem:special}, we have $\special \in \ker(G')$.
We can choose a basis that contains the special one, so let $\{ \special, \widetilde{\x}^{(2)}, \ldots, \widetilde{\x}^{(\xi)} \}$ 
be a basis for $\ker(G')$.

Note that $\special$ is contractible, but its contraction is the trivial vector $\mathbf{0}$.
Let us assume to begin with
that at least one of the vectors $\widetilde{\x}^{(2)}, \ldots, \widetilde{\x}^{(\xi)}$ is \emph{not} contractible;
without loss of generality, we take $\widetilde{\x}^{(2)}$ to be non-contractible. The vectors
$\widetilde{\x}^{(3)}, \ldots, \widetilde{\x}^{(\xi)}$ can be replaced by
\begin{equation}
\label{eq:altbasis}
\widetilde{\x}^{(3)} + \lambda_3 \widetilde{\x}^{(2)}, \widetilde{\x}^{(4)} + \lambda_4 \widetilde{\x}^{(2)}, \ldots, 
\widetilde{\x}^{(\xi)} + \lambda_\xi \widetilde{\x}^{(2)},
\end{equation}
where scalars $\lambda_3, \lambda_4, \ldots, \lambda_\xi$ can be chosen in such a way that the vectors specified in
\eqref{eq:altbasis} are
all contractible; namely, by setting $\lambda_k = - \widetilde{\x}^{(k)}(x_1) /  \widetilde{\x}^{(2)}(x_1)$. 
(Note that $\widetilde{\x}^{(2)}(x_1)$ is non-zero, because $\widetilde{\x}^{(2)}$ would otherwise be
contractible by Lemma~\ref{techlem:c}.)
The vector 
$\widetilde{\x}^{(k)} + \lambda_k \widetilde{\x}^{(2)}$ is then contractible by Lemma~\ref{techlem:c}.
Therefore, we can assume that all basis vectors, except possibly the vector $\widetilde{\x}^{(2)}$, are contractible.

First, assume that $\widetilde{\x}^{(2)}$ is non-contractible. Let $\x^{(3)}, \x^{(4)}, \ldots, \x^{(\xi)}$ be contractions
of vectors $\widetilde{\x}^{(3)}, \widetilde{\x}^{(4)}, \ldots, \widetilde{\x}^{(\xi)}$, respectively. We will prove that
$\x^{(3)}, \x^{(4)}, \ldots, \x^{(\xi)}$ are linearly independent.
For contradiction, suppose that they are not
linearly independent. Then there exist scalars $\mu_3, \mu_4, \ldots, \mu_\xi$ (at least one of which is non-zero) such that
\begin{equation}
\label{eq:linkombi2}
\mu_3 \x^{(3)} + \mu_4 \x^{(4)} + \cdots + \mu_\xi \x^{(\xi)} = \mathbf{0}.
\end{equation}
Let us define 
\[
\X := \mu_3 \widetilde{\x}^{(3)} + \mu_4 \widetilde{\x}^{(4)} + \cdots + \mu_\xi \widetilde{\x}^{(\xi)}.
\]
Observe that $\X(u) = (\mu_3 \x^{(3)} + \mu_4 \x^{(4)} + \cdots + \mu_\xi \x^{(\xi)})(u)$ for all $u \in V(G)$. From 
\eqref{eq:linkombi2} it follows that $\X(u) = 0$ for all $u \in V(G)$. As $\widetilde{\x}^{(3)}, \widetilde{\x}^{(4)}, \ldots, \widetilde{\x}^{(\xi)}$
are contractible, it also follows that $\X(x_i) = 0$ for all $1 \leq i \leq h$. By Lemma~\ref{techlem:d}, it follows that
$\X = \sigma \special$ for some scalar $\sigma$. Therefore,
\[
\mu_3 \widetilde{\x}^{(3)} + \mu_4 \widetilde{\x}^{(4)} + \cdots + \mu_\xi \widetilde{\x}^{(\xi)} - \sigma \special = \mathbf{0}.
\]
This
contradicts the fact that $\{ \special, \widetilde{\x}^{(2)}, \ldots, \widetilde{\x}^{(\xi)}\}$ is a basis, so
$\x^{(3)}, \x^{(4)}, \ldots, \x^{(\xi)}$ are linearly independent and hence $\eta(G) \geq \eta(G') - 2$.

The case where $\widetilde{\x}^{(2)}$ is contractible is analogous, except that we obtain $\eta(G) \geq \eta(G') - 1$.
Either way, $\eta(G') \leq \eta(G) + 2$, as desired.
\end{proof}

\par\noindent
Combining the bounds in Theorems \ref{thm:armes1} and \ref{thm:armes2}, we obtain our main result for the possible values of the nullity of an altan:

\begin{corollary}
\label{cor:10}
Let $G$ be a graph and $H$ an attachment set.
\begin{enumerate}[label=(\roman*)]
\item
If $h = |H|$ is even then $\eta(G) \leq \eta(\altan(G, H)) \leq \eta(G) + 2$.
\item
If $h = |H|$ is odd then $\eta(\altan(G, H)) = \eta(G)$.
\end{enumerate}
\end{corollary}

\par\noindent
Cases where $\eta(\altan(G, H)) = \eta(G)$ or $\eta(\altan(G, H)) = \eta(G) + 1$ are well known,
but, apparently, the case $\eta(\altan(G, H)) = \eta(G) + 2$ had not been noticed before.
Exhaustive search shows that the smallest benzenoid that gains two NBMOs on altanisation
is a $5$-hexagon molecule (benzo[{\it a}]tetracene, a.k.a.\ benzo[{\it a}]naphthacene), as shown in Figure~\ref{fig:bendbenznegl}/Example~\ref{ex:surprise}. 

\begin{example}
\label{ex:surprise}
Figure \ref{fig:bendbenznegl}(\subref{subfig:bendbenznegla}) shows the molecular graph
of benzo[{\it a}]tetracene.
As a catafused benzenoid, this molecule has no NBMOs (nullity equals $0$).
Surprisingly, the molecular graph of its altan, on the other hand, has nullity $2$.
Figures \ref{fig:bendbenznegl}(\subref{subfig:bendbenzneglb}) and \ref{fig:bendbenznegl}(\subref{subfig:bendbenzneglc}) show a pair of
independent eigenvectors $\x^{(1)}$ and $\x^{(2)}$ spanning the nullspace. Note that the vector $\x^{(1)}$ is, in fact, the special one,
which arises whenever the attachment set is even.
Orthonormal molecular orbitals
can be found by taking combinations 
$(1/\sqrt{14}) \x^{(1)}$ and $(1/\sqrt{162526})(9 \x^{(1)} - 14  \x^{(2)})$.
This is the smallest benzenoid for which the excess nullity of the altan is $2$.

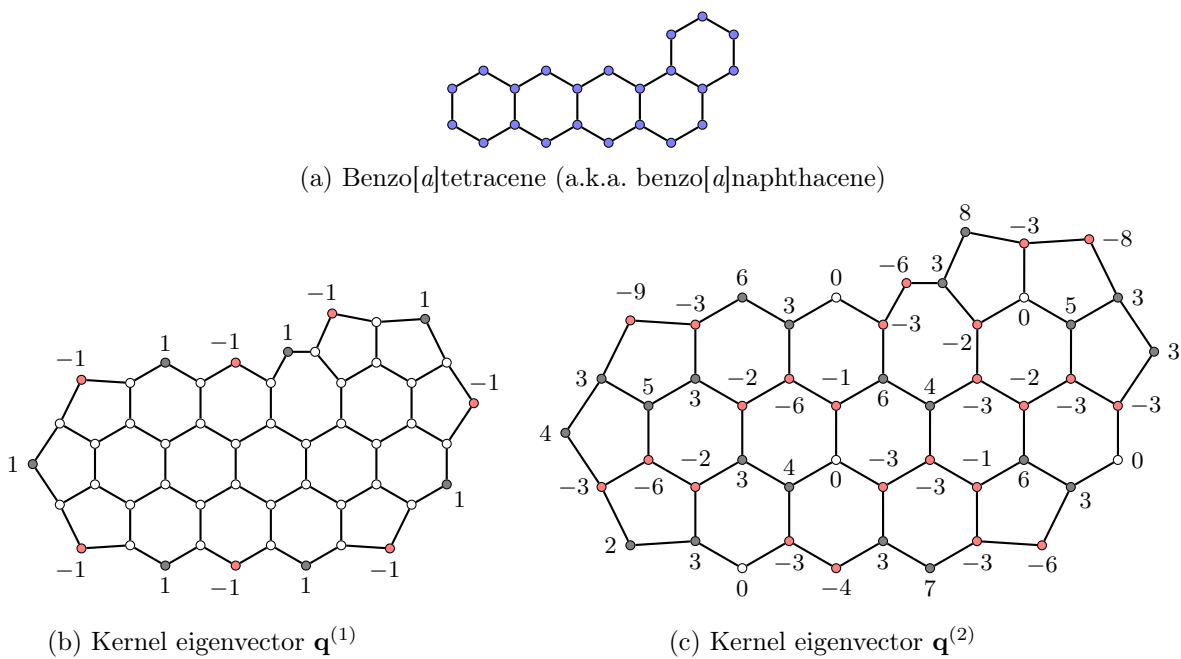
\begin{figure}[!htbp]
\centering
\begin{subfigure}[b]{0.55\linewidth}
\centering
\begin{tikzpicture}[scale=0.8]
\tikzstyle{edge}=[draw,thick]
\tikzstyle{every node}=[draw, circle, fill=blue!50!white, inner sep=1.2pt]
\node (u1) at (90:0.6) {};
\node (u2) at ({90 + 60}:0.6) {};
\node (u3) at ({90 + 2 * 60}:0.6) {};
\node (u4) at ({90 + 3 * 60}:0.6) {};
\node (u5) at ({90 + 4 * 60}:0.6) {};
\node (u6) at ({90 + 5 * 60}:0.6) {};
\node (v1) at ($ (u5) + (-30:0.6) $) {};
\node (v2) at ($ (v1) + (30:0.6) $) {};
\node (v3) at ($ (v2) + (90:0.6) $) {};
\node (v4) at ($ (v3) + (150:0.6) $) {};
\node (w1) at ($ (v2) + (-30:0.6) $) {};
\node (w2) at ($ (w1) + (30:0.6) $) {};
\node (w3) at ($ (w2) + (90:0.6) $) {};
\node (w4) at ($ (w3) + (150:0.6) $) {};
\node (x1) at ($ (w2) + (-30:0.6) $) {};
\node (x2) at ($ (x1) + (30:0.6) $) {};
\node (x3) at ($ (x2) + (90:0.6) $) {};
\node (x4) at ($ (x3) + (150:0.6) $) {};
\node (y1) at ($ (x3) + (30:0.6) $) {};
\node (y2) at ($ (y1) + (90:0.6) $) {};
\node (y3) at ($ (y2) + (150:0.6) $) {};
\node (y4) at ($ (y3) + (210:0.6) $) {};
\path[edge] (u1) -- (u2) -- (u3) -- (u4) -- (u5) -- (u6) -- (u1);
\path[edge] (u5) -- (v1) -- (v2) -- (v3) -- (v4) -- (u6);
\path[edge] (v2) -- (w1) -- (w2) -- (w3) -- (w4) -- (v3);
\path[edge] (w2) -- (x1) -- (x2) -- (x3) -- (x4) -- (w3);
\path[edge] (x3) -- (y1) -- (y2) -- (y3) -- (y4) -- (x4);
\end{tikzpicture}
\caption{Benzo[{\it a}]tetracene (a.k.a.\ benzo[{\it a}]naphthacene)}
\label{subfig:bendbenznegla}
\end{subfigure}

\hspace{-1.1cm}
\begin{subfigure}[b]{0.35\linewidth}
\centering
\begin{tikzpicture}[scale=1.2*0.75]
\tikzstyle{edge}=[draw,thick]
\tikzstyle{every node}=[draw, circle, fill=white, inner sep=1.2pt]
\tikzstyle{pozit}=[fill=black!50!white]
\tikzstyle{negat}=[fill=red!50!white]
\node (u1) at (90:0.6) {};
\node (u2) at ({90 + 60}:0.6) {};
\node (u3) at ({90 + 2 * 60}:0.6) {};
\node (u4) at ({90 + 3 * 60}:0.6) {};
\node (u5) at ({90 + 4 * 60}:0.6) {};
\node (u6) at ({90 + 5 * 60}:0.6) {};
\node (v1) at ($ (u5) + (-30:0.6) $) {};
\node (v2) at ($ (v1) + (30:0.6) $) {};
\node (v3) at ($ (v2) + (90:0.6) $) {};
\node (v4) at ($ (v3) + (150:0.6) $) {};
\node (w1) at ($ (v2) + (-30:0.6) $) {};
\node (w2) at ($ (w1) + (30:0.6) $) {};
\node (w3) at ($ (w2) + (90:0.6) $) {};
\node (w4) at ($ (w3) + (150:0.6) $) {};
\node (x1) at ($ (w2) + (-30:0.6) $) {};
\node (x2) at ($ (x1) + (30:0.6) $) {};
\node (x3) at ($ (x2) + (90:0.6) $) {};
\node (x4) at ($ (x3) + (150:0.6) $) {};
\node (y1) at ($ (x3) + (30:0.6) $) {};
\node (y2) at ($ (y1) + (90:0.6) $) {};
\node (y3) at ($ (y2) + (150:0.6) $) {};
\node (y4) at ($ (y3) + (210:0.6) $) {};
\node (a1) at ($ (u1) + (90:0.6) $) {}; \path[edge] (a1) -- (u1);
\node (a2) at ($ (v4) + (90:0.6) $) {}; \path[edge] (a2) -- (v4);
\node (a3) at ($ (w4) + (90:0.6) $) {}; \path[edge] (a3) -- (w4);
\node (a4) at ($ (y4) + (130:0.6) $) {}; \path[edge] (a4) -- (y4);
\node (a5) at ($ (y3) + (90:0.6) $) {}; \path[edge] (a5) -- (y3);
\node (a6) at ($ (y2) + (30:0.6) $) {}; \path[edge] (a6) -- (y2);
\node (a7) at ($ (y1) + (-30:0.6) $) {}; \path[edge] (a7) -- (y1);
\node (a8) at ($ (x2) + (-30:0.6) $) {}; \path[edge] (a8) -- (x2);
\node (a9) at ($ (x1) + (-90:0.6) $) {}; \path[edge] (a9) -- (x1);
\node (a10) at ($ (w1) + (-90:0.6) $) {}; \path[edge] (a10) -- (w1);
\node (a11) at ($ (v1) + (-90:0.6) $) {}; \path[edge] (a11) -- (v1);
\node (a12) at ($ (u4) + (-90:0.6) $) {}; \path[edge] (a12) -- (u4);
\node (a13) at ($ (u3) + (-150:0.6) $) {}; \path[edge] (a13) -- (u3);
\node (a14) at ($ (u2) + (150:0.6) $) {}; \path[edge] (a14) -- (u2);
\node[pozit,label=90:{\footnotesize $1$}] (b1) at ($ (a1)!0.5!(a2) + (90:0.3) $) {};
\node[negat,label={[yshift=-3pt,xshift=-4pt]90:{\footnotesize $-1$}}] (b2) at ($ (a2)!0.5!(a3) + (90:0.3) $) {};
\node[pozit,label=90:{\footnotesize $1$}] (b3) at ($ (a4) + (180:0.4) $) {};
\node[negat,label={[yshift=-3pt,xshift=-4pt]90:{\footnotesize $-1$}}] (b4) at ($ (a4)!0.5!(a5) + (120:0.4) $) {};
\node[pozit,label=90:{\footnotesize $1$}] (b5) at ($ (a5)!0.5!(a6) + (60:0.4) $) {};
\node[negat,label={[yshift=-3pt,xshift=4pt]90:{\footnotesize $-1$}}] (b6) at ($ (a6)!0.5!(a7) + (0:0.4) $) {};
\node[pozit,label=-60:{\footnotesize $1$}] (b7) at ($ (a7)!0.5!(a8) + (-30:0.3) $) {};
\node[negat,label={[yshift=2pt,xshift=-2pt]-90:{\footnotesize $-1$}}] (b8) at ($ (a8)!0.5!(a9) + (-60:0.4) $) {};
\node[pozit,label=-90:{\footnotesize $1$}] (b9) at ($ (a9)!0.5!(a10) + (-90:0.3) $) {};
\node[negat,label={[yshift=2pt,xshift=-4pt]-90:{\footnotesize $-1$}}] (b10) at ($ (a10)!0.5!(a11) + (-90:0.3) $) {};
\node[pozit,,label=-90:{\footnotesize $1$}] (b11) at ($ (a11)!0.5!(a12) + (-90:0.3) $) {};
\node[negat,label={[yshift=2pt,xshift=-4pt]-90:{\footnotesize $-1$}}] (b12) at ($ (a12)!0.5!(a13) + (-120:0.4) $) {};
\node[pozit,label=180:{\footnotesize $1$}] (b13) at ($ (a13)!0.5!(a14) + (-180:0.4) $) {};
\node[negat,label={[yshift=-3pt,xshift=-4pt]90:{\footnotesize $-1$}}] (b14) at ($ (a14)!0.5!(a1) + (120:0.4) $) {};
\path[edge] (u1) -- (u2) -- (u3) -- (u4) -- (u5) -- (u6) -- (u1);
\path[edge] (u5) -- (v1) -- (v2) -- (v3) -- (v4) -- (u6);
\path[edge] (v2) -- (w1) -- (w2) -- (w3) -- (w4) -- (v3);
\path[edge] (w2) -- (x1) -- (x2) -- (x3) -- (x4) -- (w3);
\path[edge] (x3) -- (y1) -- (y2) -- (y3) -- (y4) -- (x4);
\path[edge] (a1) -- (b1) -- (a2) -- (b2) -- (a3) -- (b3) -- (a4) -- (b4) -- (a5) -- (b5) -- (a6) -- (b6)
 -- (a7) -- (b7) -- (a8) -- (b8) -- (a9) -- (b9) -- (a10) -- (b10) -- (a11) -- (b11) -- (a12) -- (b12)
 -- (a13) -- (b13) -- (a14) -- (b14) -- (a1);
\end{tikzpicture}
\caption{Kernel eigenvector $\x^{(1)}$}
\label{subfig:bendbenzneglb}
\end{subfigure}
\qquad\qquad
\begin{subfigure}[b]{0.50\linewidth}
\centering
\begin{tikzpicture}[scale=1.2]
\tikzstyle{edge}=[draw,thick]
\tikzstyle{every node}=[draw, circle, fill=white, inner sep=1.2pt]
\tikzstyle{pozit}=[fill=black!50!white]
\tikzstyle{negat}=[fill=red!50!white]
\node[pozit,label={[yshift=0,xshift=0]-90:{\footnotesize $3$}}] (u1) at (90:0.6) {};
\node[pozit,label={[yshift=0,xshift=0]90:{\footnotesize $5$}}] (u2) at ({90 + 60}:0.6) {};
\node[negat,label={[yshift=0,xshift=0]-90:{\footnotesize $-6$}}] (u3) at ({90 + 2 * 60}:0.6) {};
\node[negat,label={[yshift=0,xshift=0]90:{\footnotesize $-2$}}] (u4) at ({90 + 3 * 60}:0.6) {};
\node[pozit,label={[yshift=0,xshift=0]-90:{\footnotesize $3$}}] (u5) at ({90 + 4 * 60}:0.6) {};
\node[negat,label={[yshift=0,xshift=0]90:{\footnotesize $-2$}}] (u6) at ({90 + 5 * 60}:0.6) {};
\node[pozit,label={[yshift=0,xshift=0]90:{\footnotesize $4$}}] (v1) at ($ (u5) + (-30:0.6) $) {};
\node[label={[yshift=0,xshift=0]-90:{\footnotesize $0$}}] (v2) at ($ (v1) + (30:0.6) $) {};
\node[negat,label={[yshift=0,xshift=0]90:{\footnotesize $-1$}}] (v3) at ($ (v2) + (90:0.6) $) {};
\node[negat,label={[yshift=0,xshift=0]-90:{\footnotesize $-6$}}] (v4) at ($ (v3) + (150:0.6) $) {};
\node[negat,label={[yshift=0,xshift=0]90:{\footnotesize $-3$}}] (w1) at ($ (v2) + (-30:0.6) $) {};
\node[negat,label={[yshift=0,xshift=0]-90:{\footnotesize $-3$}}] (w2) at ($ (w1) + (30:0.6) $) {};
\node[pozit,label={[yshift=0,xshift=0]90:{\footnotesize $4$}}] (w3) at ($ (w2) + (90:0.6) $) {};
\node[pozit,label={[yshift=0,xshift=0]-90:{\footnotesize $6$}}] (w4) at ($ (w3) + (150:0.6) $) {};
\node[negat,label={[yshift=0,xshift=0]90:{\footnotesize $-1$}}] (x1) at ($ (w2) + (-30:0.6) $) {};
\node[pozit,label={[yshift=0,xshift=0]-90:{\footnotesize $6$}}] (x2) at ($ (x1) + (30:0.6) $) {};
\node[negat,label={[yshift=0,xshift=0]90:{\footnotesize $-2$}}] (x3) at ($ (x2) + (90:0.6) $) {};
\node[negat,label={[yshift=0,xshift=0]-90:{\footnotesize $-3$}}] (x4) at ($ (x3) + (150:0.6) $) {};
\node[negat,label={[yshift=0,xshift=0]-90:{\footnotesize $-3$}}] (y1) at ($ (x3) + (30:0.6) $) {};
\node[pozit,label={[yshift=0,xshift=0]90:{\footnotesize $5$}}] (y2) at ($ (y1) + (90:0.6) $) {};
\node[label={[yshift=0,xshift=0]-90:{\footnotesize $0$}}] (y3) at ($ (y2) + (150:0.6) $) {};
\node[negat,label={[yshift=0,xshift=0]210:{\footnotesize $-2$}}] (y4) at ($ (y3) + (210:0.6) $) {};
\node[negat,label={[yshift=-3pt,xshift=-2pt]90:{\footnotesize $-3$}}] (a1) at ($ (u1) + (90:0.6) $) {}; \path[edge] (a1) -- (u1);
\node[pozit,label={[yshift=0,xshift=0]90:{\footnotesize $3$}}] (a2) at ($ (v4) + (90:0.6) $) {}; \path[edge] (a2) -- (v4);
\node[negat,label={[yshift=0pt,xshift=-2pt]0:{\footnotesize $-3$}}]  (a3) at ($ (w4) + (90:0.6) $) {}; \path[edge] (a3) -- (w4);
\node[pozit,label={[yshift=0,xshift=-2pt]90:{\footnotesize $3$}}] (a4) at ($ (y4) + (130:0.6) $) {}; \path[edge] (a4) -- (y4);
\node[negat,label={[yshift=-4pt,xshift=0]90:{\footnotesize $-3$}}]  (a5) at ($ (y3) + (90:0.6) $) {}; \path[edge] (a5) -- (y3);
\node[pozit,label={[yshift=0,xshift=0]0:{\footnotesize $3$}}] (a6) at ($ (y2) + (30:0.6) $) {}; \path[edge] (a6) -- (y2);
\node[negat,label={[yshift=0,xshift=0]0:{\footnotesize $-3$}}] (a7) at ($ (y1) + (-30:0.6) $) {}; \path[edge] (a7) -- (y1);
\node[pozit,label={[yshift=0,xshift=0]-30:{\footnotesize $3$}}] (a8) at ($ (x2) + (-30:0.6) $) {}; \path[edge] (a8) -- (x2);
\node[negat,label={[yshift=2pt,xshift=0]-90:{\footnotesize $-3$}}] (a9) at ($ (x1) + (-90:0.6) $) {}; \path[edge] (a9) -- (x1);
\node[pozit,label={[yshift=0,xshift=0]-90:{\footnotesize $3$}}]  (a10) at ($ (w1) + (-90:0.6) $) {}; \path[edge] (a10) -- (w1);
\node[negat,label={[yshift=2pt,xshift=0]-90:{\footnotesize $-3$}}] (a11) at ($ (v1) + (-90:0.6) $) {}; \path[edge] (a11) -- (v1);
\node[pozit,label={[yshift=0,xshift=0]-90:{\footnotesize $3$}}] (a12) at ($ (u4) + (-90:0.6) $) {}; \path[edge] (a12) -- (u4);
\node[negat,label={[yshift=0,xshift=0]180:{\footnotesize $-3$}}] (a13) at ($ (u3) + (-150:0.6) $) {}; \path[edge] (a13) -- (u3);
\node[pozit,label={[yshift=0,xshift=0]180:{\footnotesize $3$}}] (a14) at ($ (u2) + (150:0.6) $) {}; \path[edge] (a14) -- (u2);
\node[pozit,label={[yshift=0,xshift=0]90:{\footnotesize $6$}}] (b1) at ($ (a1)!0.5!(a2) + (90:0.3) $) {};
\node[label={[yshift=0,xshift=0]90:{\footnotesize $0$}}]  (b2) at ($ (a2)!0.5!(a3) + (90:0.3) $) {};
\node[negat,label={[yshift=-4pt,xshift=-5pt]90:{\footnotesize $-6$}}] (b3) at ($ (a4) + (180:0.4) $) {};
\node[pozit,label={[yshift=0,xshift=0]90:{\footnotesize $8$}}] (b4) at ($ (a4)!0.5!(a5) + (120:0.4) $) {};
\node[negat,label={[yshift=0,xshift=0]0:{\footnotesize $-8$}}] (b5) at ($ (a5)!0.5!(a6) + (60:0.4) $) {};
\node[pozit,label={[yshift=0,xshift=0]0:{\footnotesize $3$}}]  (b6) at ($ (a6)!0.5!(a7) + (0:0.4) $) {};
\node[label={[yshift=0,xshift=0]0:{\footnotesize $0$}}]  (b7) at ($ (a7)!0.5!(a8) + (-30:0.3) $) {};
\node[negat,label={[yshift=3pt,xshift=0]-90:{\footnotesize $-6$}}] (b8) at ($ (a8)!0.5!(a9) + (-60:0.4) $) {};
\node[pozit,label={[yshift=0,xshift=0]-90:{\footnotesize $7$}}] (b9) at ($ (a9)!0.5!(a10) + (-90:0.3) $) {};
\node[negat,label={[yshift=3pt,xshift=0]-90:{\footnotesize $-4$}}] (b10) at ($ (a10)!0.5!(a11) + (-90:0.3) $) {};
\node[label={[yshift=0,xshift=0]-90:{\footnotesize $0$}}] (b11) at ($ (a11)!0.5!(a12) + (-90:0.3) $) {};
\node[pozit,label={[yshift=0,xshift=0]180:{\footnotesize $2$}}]  (b12) at ($ (a12)!0.5!(a13) + (-120:0.4) $) {};
\node[pozit,label={[yshift=0,xshift=0]180:{\footnotesize $4$}}] (b13) at ($ (a13)!0.5!(a14) + (-180:0.4) $) {};
\node[negat,label={[yshift=0,xshift=0]90:{\footnotesize $-9$}}] (b14) at ($ (a14)!0.5!(a1) + (120:0.4) $) {};
\path[edge] (u1) -- (u2) -- (u3) -- (u4) -- (u5) -- (u6) -- (u1);
\path[edge] (u5) -- (v1) -- (v2) -- (v3) -- (v4) -- (u6);
\path[edge] (v2) -- (w1) -- (w2) -- (w3) -- (w4) -- (v3);
\path[edge] (w2) -- (x1) -- (x2) -- (x3) -- (x4) -- (w3);
\path[edge] (x3) -- (y1) -- (y2) -- (y3) -- (y4) -- (x4);
\path[edge] (a1) -- (b1) -- (a2) -- (b2) -- (a3) -- (b3) -- (a4) -- (b4) -- (a5) -- (b5) -- (a6) -- (b6)
 -- (a7) -- (b7) -- (a8) -- (b8) -- (a9) -- (b9) -- (a10) -- (b10) -- (a11) -- (b11) -- (a12) -- (b12)
 -- (a13) -- (b13) -- (a14) -- (b14) -- (a1);
\end{tikzpicture}
\caption{Kernel eigenvector $\x^{(2)}$}
\label{subfig:bendbenzneglc}
\end{subfigure}
\caption{The molecular graph of benzo[{\it a}]tetracene and 
two independent kernel eigenvectors
of the altan. In (b) and (c) all vertices represented by white-filled circles carry entry $0$.}
\label{fig:bendbenznegl}
\end{figure}
\end{example}

So far, the complete list of the possible cases for the difference in nullity between a parent graph and its altan 
was derived from general considerations of linear algebra 
concerning the possibilities for 
extension of kernel vectors from parent to altan, and contraction
from altan to parent.
These considerations can also be used to derive an interesting result for this difference in the case where the parent graph is itself an altan.
Again, a preliminary technical lemma is needed.
\begin{lemma}
\label{lem:everycontractible}
Let $(G', H')$ be as in Definition~\ref{def:altan}. Let $(G'', H'') = \altan(G', H')$ and let vertices of
the perimeter of $G''$ be labeled $x'_1, x'_2, \ldots, x'_h$ and $y'_1, y'_2, \ldots, y'_h$ so that
$N(x'_1) = \{ y'_h, y_1, y'_1 \}$,  $N(x'_2) = \{ y'_1, y_2, y'_2 \}$,  $N(x'_3) = \{ y'_2, y_3, y'_3 \}$ and so on (see Figure~\ref{fig:altan2nd}).
Let $\x \in \ker (G'')$ and let $\mathcal{D}(\x) = \sum_{i=1}^{h}(-1)^{i}\x(y_i)$. 
\begin{enumerate}[label=(\roman*)]
\item
If $h$ is even then $\x(x'_i) = 0$ for all $1 \leq i \leq h$. 
\item
If $h$ is even then $\mathcal{D}(\x) = 0$.
\end{enumerate}
\end{lemma}

\begin{figure}[!htb]
\centering
\begin{tikzpicture}[scale=1.2]
\tikzstyle{edge}=[draw,thick]
\tikzstyle{every node}=[draw, circle, fill=blue!50!white, inner sep=1.5pt]
\draw[fill=green!20!white] (0,0) ellipse (2cm and 1.2cm);
\node[label={[yshift=6pt]-90:$v_{h-1}$}] (vh-1) at (170:1.6) {};
\node[label={[yshift=2pt]-90:$v_h$}] (vh) at (140:1.2) {};
\node[label={[yshift=2pt]-90:$v_1$}] (v1) at (90:0.95) {};
\node[label={[yshift=2pt]-90:$v_2$}] (v2) at (40:1.2) {};
\node[label={[yshift=2pt]-90:$v_3$}] (v3) at (10:1.6) {};
\node[label={[yshift=0pt]180:$x_{h-1}$}] (xh-1) at (170:1.6*1.5) {};
\node[label={[yshift=0pt]120:$x_h$}] (xh) at (140:1.2*1.7) {};
\node[label={[yshift=-2pt]90:$x_1$}] (x1) at (90:0.95*1.9) {};
\node[label={[yshift=0pt]60:$x_2$}] (x2) at (40:1.2*1.7) {};
\node[label={[yshift=0pt]0:$x_3$}] (x3) at (10:1.6*1.5) {};
\node[label={[yshift=6pt,xshift=-6pt]-90:$y_{h-2}$}] (yh-2) at (185:2.8) {};
\node[label={[yshift=0pt]180:$y_{h-1}$}] (yh-1) at ($ (xh-1)!0.5!(xh) + (150:0.5) $) {};
\node[label={[yshift=-4pt,xshift=-10pt]90:$y_h$}] (yh) at ($ (xh)!0.5!(x1) + (105:0.4) $) {};
\node[label={[yshift=-4pt,xshift=11pt]90:$y_1$}] (y1) at ($ (x1)!0.5!(x2) + (75:0.4) $) {};
\node[label={[yshift=-2pt]0:$y_2$}] (y2) at ($ (x2)!0.5!(x3) + (30:0.5) $) {};
\node[label={[yshift=0pt,xshift=2pt]-90:$y_3$}] (y3) at (-5:2.8) {};
\path[edge] ($ (yh-2) + (-50:0.4) $) -- (yh-2) -- (xh-1) -- (yh-1) -- (xh) -- (yh) -- (x1) -- (y1) -- (x2) -- (y2) -- (x3) -- (y3) -- ($ (y3) + (230:0.4) $);
\path[edge] (vh-1) -- (xh-1);
\path[edge] (vh) -- (xh);
\path[edge] (v1) -- (x1);
\path[edge] (v2) -- (x2);
\path[edge] (v3) -- (x3);
\node[label={[yshift=-2pt]90:$x'_1$}] (xp1) at ($ (y1) + (80:0.9) $) {};
\node[label={[yshift=-2pt,xshift=-2pt]60:$x'_2$}] (xp2) at ($ (y2) + (40:0.9) $) {};
\node[label={[yshift=0pt,xshift=-2pt]0:$x'_3$}] (xp3) at ($ (y3) + (0:0.9) $) {};
\node[label={[yshift=-2pt]90:$x'_h$}] (xph) at ($ (yh) + (100:0.9) $) {};
\node[label={[yshift=-3pt,xshift=2pt]120:$x'_{h-1}$}] (xph-1) at ($ (yh-1) + (140:0.9) $) {};
\node[label={[xshift=2pt]180:$x'_{h-2}$}] (xph-2) at ($ (yh-2) + (180:0.9) $) {};
\path[edge] (y1) -- (xp1);
\path[edge] (y2) -- (xp2);
\path[edge] (y3) -- (xp3);
\path[edge] (yh) -- (xph);
\path[edge] (yh-1) -- (xph-1);
\path[edge] (yh-2) -- (xph-2);
\node[label={[yshift=0pt]60:$y'_1$}] (yp1) at ($ (xp1)!0.5!(xp2) + (60:0.3) $) {};
\node[label={[xshift=-2pt]0:$y'_2$}] (yp2) at ($ (xp2)!0.5!(xp3) + (10:0.3) $) {};
\node[label={[yshift=-2pt]90:$y'_h$}] (yph) at ($ (xp1)!0.5!(xph) + (90:0.3) $) {};
\node[label={[yshift=-2pt,xshift=6pt]120:$y'_{h-1}$}] (yph-1) at ($ (xph-1)!0.5!(xph) + (120:0.3) $) {};
\node[label={[xshift=2pt]180:$y'_{h-2}$}] (yph-2) at ($ (xph-2)!0.5!(xph-1) + (170:0.3) $) {};
\path[edge] ($ (xph-2) + (260:0.3) $) -- (xph-2) -- (yph-2) -- (xph-1) -- (yph-1) -- (xph) -- (yph) -- (xp1) -- (yp1) -- (xp2) -- (yp2) -- (xp3) -- ($ (xp3) + (-80:0.3) $);
\node[draw=none,fill=none] at (0, 0) {$G$};
\node[draw=none,fill=none] at (-20:1.2) {$\ldots$};
\node[draw=none,fill=none] at ($ (y3) + (230:0.8) $) {$\ldots$};
\end{tikzpicture}
\vspace{6pt}
\caption{The double altan $(G'', H'') = \altan(G', H') = \altan^2(G, H)$ as described in Lemma~\ref{lem:everycontractible}.}
\label{fig:altan2nd}
\end{figure}
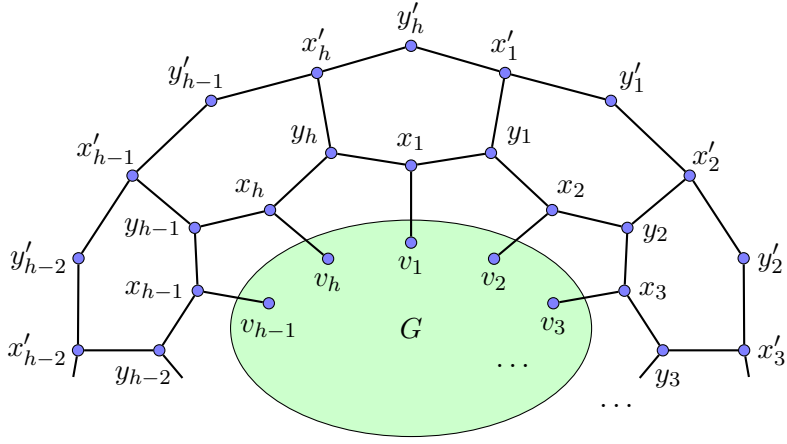

\begin{proof}
(i): By Lemma~\ref{techlem:c}, we can write $\x(x'_i) = a$ for all even $i$ and $\x(x'_i) = -a$ for all odd $i$. By pivoting at vertices
$y_1, y_2, \ldots, y_h$ we get the following conditions:
\begin{equation}
\label{eq:pivokonds}
\begin{aligned}
x_1 + x_2 \phantom{+ x_3 + x_4 + x_{h-1} + x_h} & = \phantom{-} a\\
\phantom{x_1 +} x_2 + x_3 \phantom{+ x_4 + x_{h-1} + x_h} & = -a \\
\phantom{x_1 + x_2 +} x_3 + x_4 \phantom{+ x_{h-1} + x_h} & = \phantom{-} a \\
& \ldots \\
\phantom{x_1 + x_2 + x_3 + x_4 +} x_{h-1} + x_h  & = \phantom{-} a\\
x_1 \phantom{+ x_2 + x_3 + x_4 + x_{h-1}} + x_h  & = - a 
\end{aligned}
\end{equation}
By summing all rows of \eqref{eq:pivokonds} we get $\sum_{i=1}^{h} x_i = 0$. By summing all odd rows of
\eqref{eq:pivokonds}, we get $\sum_{i=1}^{h} x_i = \frac{h}{2}a$. Hence, $a = 0$.

(ii): It is easy to see that $\mathcal{D}$ is a linear functional.
Now, let $\x(y'_h) = t \in \mathbb{R}$. By pivoting at vertices $x'_1, x'_2, \ldots, x'_{h-1}$ we obtain
\begin{align*}
\x(y'_1) & =  -t - \x(y_1) \\
\x(y'_2) & =  \phantom{-}t + \x(y_1) - \x(y_2) \\
\x(y'_3) & =  -t - \x(y_1) + \x(y_2) - \x(y_3)  \\
& \ldots \\
\x(y'_{h - 1}) & = -t - \x(y_1) + \x(y_2) - \x(y_3) + \cdots - \x(y_{h-1}).
\end{align*}
Finally, by pivoting at vertex $x'_h$ we obtain
$\x(y'_{h - 1}) + t + \x(y_{h}) = \mathcal{D}(\x) = 0$, as desired.
\end{proof}

\begin{theorem}
\label{thm:armes3}
Let $G$ be a graph and $H$ an attachment set. Then
\begin{equation}
\eta(\altan(G, H)) = \eta(\altan^2(G, H)).
\end{equation}
\end{theorem}

\begin{proof}
If $h = |H|$ is odd then the statement of the theorem follows from Corollary~\ref{cor:10}(ii).

Suppose that $h = |H|$ is even. Let $\xi = \eta(G'')$. Let $\widetilde{\special} \in \ker(G'')$ denote the special
vector of $G''$ and let $\special \in \ker(G')$ denote the special vector of $G'$. (The vector $\special$ has non-zero entries
on vertices $y_i$ only, whilst the vector $\widetilde{\special}$ has non-zero entries on vertices $y'_i$ only.)
Vectors $\special$ and $\widetilde{\special}$ both exist by Lemma~\ref{lem:special}.
We can choose a basis for $G''$ that contains the special one, so let $\{ \widetilde{\special}, \widetilde{\x}^{(2)}, \ldots, \widetilde{\x}^{(\xi)} \}$ 
be a basis for $\ker(G'')$.

By Lemma~\ref{lem:everycontractible}(i), the vectors $\widetilde{\x}^{(2)}, \widetilde{\x}^{(3)}, \ldots, \widetilde{\x}^{(\xi)}$
are contractible.  Let $\x^{(2)}, \x^{(3)}, \ldots, \x^{(\xi)}$ be contractions of 
$\widetilde{\x}^{(2)}, \widetilde{\x}^{(3)}, \ldots, \widetilde{\x}^{(\xi)}$, respectively. Note that $\x^{(2)}, \x^{(3)}, \ldots, \x^{(\xi)} \in \ker (G')$.

Since $\mathcal{D}$ (defined in Lemma~\ref{lem:everycontractible} for vectors $\widetilde{\x} \in \ker(G'')$) 
involves only the entries $y_1, y_2, \ldots,\linebreak
 y_h$, it remains well defined for vectors $\x \in \ker(G')$. 
Observe that $\mathcal{D}(\x) = \mathcal{D}(\widetilde{\x})$ if $\x$ is a contraction of $\widetilde{\x}$.
The vector $\special$ is linearly independent of $\x^{(2)}, \x^{(3)}, \ldots, \x^{(\xi)}$, since $\mathcal{D}(\x^{(i)}) = 0$ for all $2 \leq i \leq \xi$
by Lemma~\ref{lem:everycontractible}(ii) and $\mathcal{D}(\special) = h$.

It remains to be proved that the vectors $\x^{(2)}, \x^{(3)}, \ldots, \x^{(\xi)}$ are linearly independent of each other.
For contradiction, suppose that they are not
linearly independent. Then there exist scalars $\mu_2, \mu_3, \ldots, \mu_\xi$ (at least one of which is non-zero) such that
\begin{equation}
\label{eq:linkombi3}
\mu_2 \x^{(2)} + \mu_3 \x^{(3)} + \cdots + \mu_\xi \x^{(\xi)} = \mathbf{0}.
\end{equation}
Now define 
\[
\X := \mu_2 \widetilde{\x}^{(2)} + \mu_3 \widetilde{\x}^{(3)} + \cdots + \mu_\xi \widetilde{\x}^{(\xi)}.
\]
Observe that $\X(u) = (\mu_2 \x^{(2)} + \mu_3 \x^{(3)} + \cdots + \mu_\xi \x^{(\xi)})(u)$ for all $u \in V(G')$.
From \eqref{eq:linkombi3} it follows that $\X(u) = 0$ for all $u \in V(G')$. From Lemma~\ref{lem:everycontractible}(i)
it follows that $\X(x'_i) = 0$ for all $1 \leq i \leq h$.
By Lemma~\ref{techlem:d}, it follows that
$\X = \sigma \widetilde{\special}$ for some scalar $\sigma$. Therefore,
\[
\mu_2 \widetilde{\x}^{(2)} + \mu_3 \widetilde{\x}^{(3)} + \cdots + \mu_\xi \widetilde{\x}^{(\xi)} - \sigma \widetilde{\special} = \mathbf{0}.
\]
This
contradicts the fact that $\{ \widetilde{\special}, \widetilde{\x}^{(2)}, \ldots, \widetilde{\x}^{(\xi)}\}$ is a basis for $\ker(G'')$.

To summarise, vectors $\special, \x^{(2)}, \x^{(3)}, \ldots, \x^{(\xi)}$ are linearly independent and hence $\eta(G') \geq \eta(G'')$.
Theorem~\ref{thm:armes1} implies $\eta(G') \leq \eta(G'')$. It follows that $\eta(G') = \eta(G'')$, as claimed.
\end{proof}

Iteration of Theorem~\ref{thm:armes3} gives directly:

\begin{corollary}
Let $G$ be a graph and $H$ an attachment set. Then
\begin{equation}
\eta(\altan(G, H)) = \eta(\altan^n(G, H))
\end{equation}
for any $n \geq 1$.
\end{corollary}

Hence, successive altanisations surround the original patch and its first-altan penumbra 
of faces with rings of $h$ hexagons, but do not change the nullity of the graph.
Given the restrictive range of bond lengths available to carbon nanostructures,
this implies that at some level of iterated altanisation the system will pop out
of the plane and form a tube-like structure, closed at one end by 
the patch and its first altan acting together as a generalised hemispherical cap. 
This change in
geometrical structure has implications for the chemistry of any such molecule (see
Section~\ref{sec:conclusion}).

\section{Computational results}

The preceding sections have presented mathematical theorems for the limits on excess nullity.
The excess can take values $0$, $1$ or $2$ only.
It is of interest to check how these cases are distributed for examples of chemical relevance.
We made a computational survey of small examples, to get some information about the relative 
frequencies of the different allowed values. The chosen families of molecular graphs are
 benzenoids (general, catafused and convex) and patches composed of various combinations of 5-, 6-
and 7-membered rings (solely pentagonal, solely heptagonal, pentagonal-and-hexagonal and pentagonal-and-heptagonal).

General and catacondensed benzenoids were generated using the \texttt{catacondensed} generator which is
included in the CaGe package \cite{CaGe}. For convex benzenoids we used in-house software that was developed
for stratified enumeration reported in \cite{Basic2018}. The \texttt{ngons} generator from the CaGe package was
used for the other families of patches.
To determine nullity of parent and altan graphs we developed a program in SageMath \cite{SageMath}
that uses the \texttt{rank} method for matrices over the integer ring (exact linear algebra computation).
Calculations are carried out for the natural attachment set only.
Calculations of the altan nullity are needed only for patches with (natural) attachment sets of even size, 
as odd attachment sets give trivially zero excess nullity (Corollary~\ref{cor:10}).

Table~\ref{table:benz} presents the results for benzenoids on up to $15$ hexagons. 
Table~\ref{table:catabenz} shows the results for the equivalent computation limited to the catafused benzenoids. 
The most striking conclusion 
from both tables is that by far the majority of 
altan benzenoids have minimum excess nullity (i.e.\ $1$ for Kekulean benzenoids and $0$ for non-Kekulean benzenoids).
Cases with the next allowed value soon start to appear, and of these the most interesting is for Kekulean (i.e.\ non-singular) benzenoid parents that 
give rise to altans with nullity $2$. 
As mentioned earlier, this possibility had not been noticed before. 
Here we see that it occurs for benzenoids with 5 and 7
hexagons, and then, apparently, for all numbers of hexagons from $9$ on.
Figure~\ref{fig:small_excess2_benz} shows the smallest 
examples of catafused and perifused benzenoids for which the nullity jumps by two on altanisation.
Figure~\ref{fig:nullity2_excess2_benz} shows the smallest examples for which 
the nullity jumps by two for non-Kekulean benzenoids.

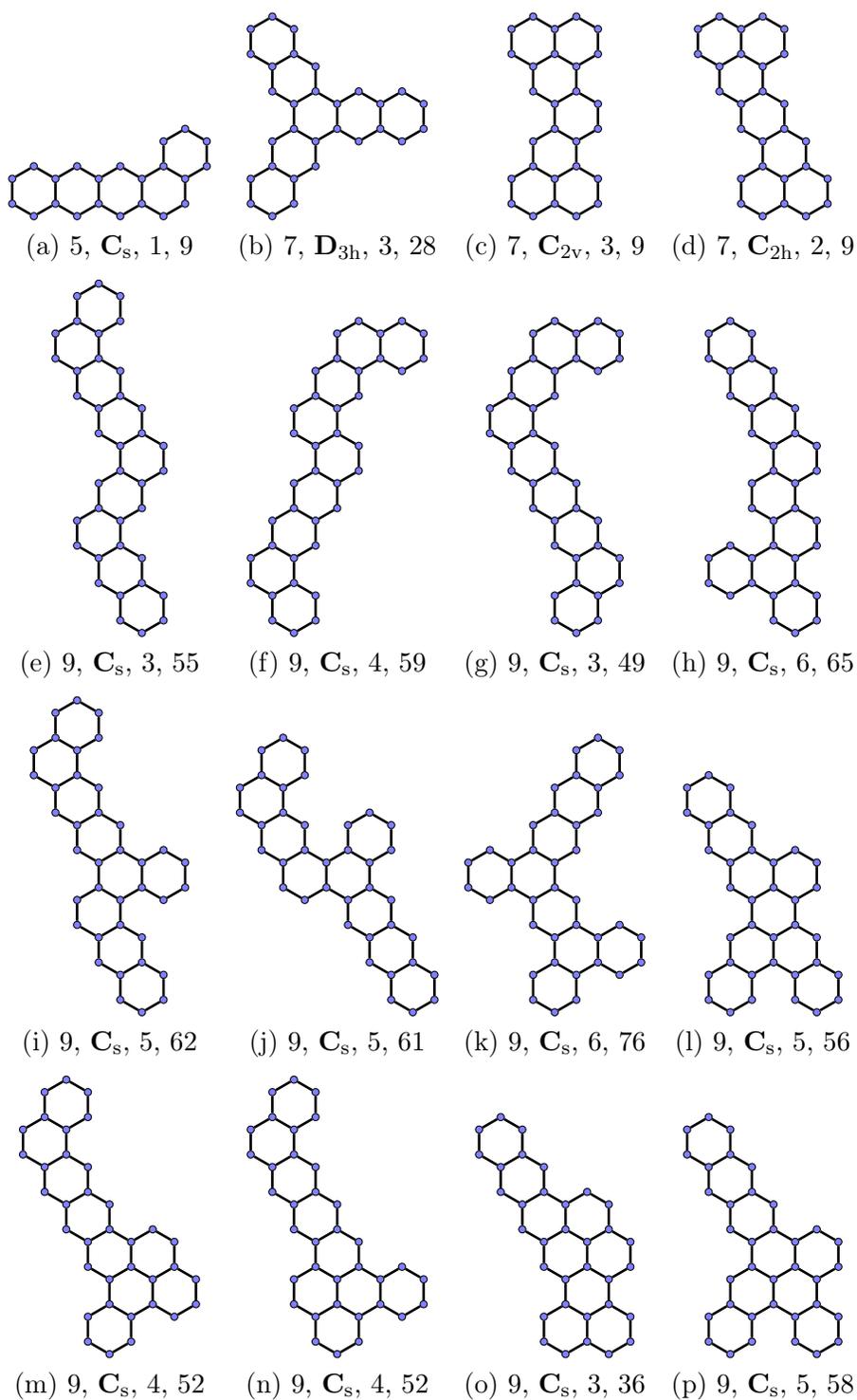
\begin{figure}[!htbp]
\centering
$\begin{array}{cccc}
\begin{tikzpicture}[scale=\scfactor,rotate=-120,xscale=-1]
\tikzstyle{every node} = [inner sep=1, draw, circle, fill=blue!50!white]
\tikzstyle{edge} = [draw, line width=1.0]
\tikzstyle{periedge} = [draw, line width=1.0]
\node (0_0_0) at (0.000000, 1.400000) {};
\node (0_0_1) at (1.212436, 0.700000) {};
\node (1_-1_0) at (1.212436, -0.700000) {};
\node (0_-1_1) at (0.000000, -1.400000) {};
\node (0_-1_0) at (-1.212436, -0.700000) {};
\node (-1_0_1) at (-1.212436, 0.700000) {};
\node (-1_-3_0) at (-6.062178, -4.900000) {};
\node (-1_-3_1) at (-4.849742, -5.600000) {};
\node (0_-4_0) at (-4.849742, -7.000000) {};
\node (-1_-4_1) at (-6.062178, -7.700000) {};
\node (-1_-4_0) at (-7.274613, -7.000000) {};
\node (-2_-3_1) at (-7.274613, -5.600000) {};
\node (-1_-2_0) at (-4.849742, -2.800000) {};
\node (-1_-2_1) at (-3.637307, -3.500000) {};
\node (0_-3_0) at (-3.637307, -4.900000) {};
\node (-2_-2_1) at (-6.062178, -3.500000) {};
\node (-1_-1_0) at (-3.637307, -0.700000) {};
\node (-1_-1_1) at (-2.424871, -1.400000) {};
\node (0_-2_0) at (-2.424871, -2.800000) {};
\node (-2_-1_1) at (-4.849742, -1.400000) {};
\node (-1_0_0) at (-2.424871, 1.400000) {};
\node (-2_0_1) at (-3.637307, 0.700000) {};
\draw[periedge] (-1_0_1) -- (0_0_0);
\draw[periedge] (0_0_0) -- (0_0_1);
\draw[periedge] (0_0_1) -- (1_-1_0);
\draw[periedge] (0_-1_1) -- (1_-1_0);
\draw[periedge] (0_-1_0) -- (0_-1_1);
\draw[edge] (-1_0_1) -- (0_-1_0);
\draw[periedge] (-2_-3_1) -- (-1_-3_0);
\draw[edge] (-1_-3_0) -- (-1_-3_1);
\draw[periedge] (-1_-3_1) -- (0_-4_0);
\draw[periedge] (-1_-4_1) -- (0_-4_0);
\draw[periedge] (-1_-4_0) -- (-1_-4_1);
\draw[periedge] (-2_-3_1) -- (-1_-4_0);
\draw[periedge] (-2_-2_1) -- (-1_-2_0);
\draw[edge] (-1_-2_0) -- (-1_-2_1);
\draw[periedge] (-1_-2_1) -- (0_-3_0);
\draw[periedge] (-1_-3_1) -- (0_-3_0);
\draw[periedge] (-2_-2_1) -- (-1_-3_0);
\draw[periedge] (-2_-1_1) -- (-1_-1_0);
\draw[edge] (-1_-1_0) -- (-1_-1_1);
\draw[periedge] (-1_-1_1) -- (0_-2_0);
\draw[periedge] (-1_-2_1) -- (0_-2_0);
\draw[periedge] (-2_-1_1) -- (-1_-2_0);
\draw[periedge] (-2_0_1) -- (-1_0_0);
\draw[periedge] (-1_0_0) -- (-1_0_1);
\draw[periedge] (-1_-1_1) -- (0_-1_0);
\draw[periedge] (-2_0_1) -- (-1_-1_0);
\end{tikzpicture} & 
\begin{tikzpicture}[scale=\scfactor]
\tikzstyle{every node} = [inner sep=1, draw, circle, fill=blue!50!white]
\tikzstyle{edge} = [draw, line width=1.0]
\tikzstyle{periedge} = [draw, line width=1.0]
\node (0_0_0) at (0.000000, 1.400000) {};
\node (0_0_1) at (1.212436, 0.700000) {};
\node (1_-1_0) at (1.212436, -0.700000) {};
\node (0_-1_1) at (0.000000, -1.400000) {};
\node (0_-1_0) at (-1.212436, -0.700000) {};
\node (-1_0_1) at (-1.212436, 0.700000) {};
\node (-1_0_0) at (-2.424871, 1.400000) {};
\node (-1_-1_1) at (-2.424871, -1.400000) {};
\node (-1_-1_0) at (-3.637307, -0.700000) {};
\node (-2_0_1) at (-3.637307, 0.700000) {};
\node (-2_-2_0) at (-7.274613, -2.800000) {};
\node (-2_-2_1) at (-6.062178, -3.500000) {};
\node (-1_-3_0) at (-6.062178, -4.900000) {};
\node (-2_-3_1) at (-7.274613, -5.600000) {};
\node (-2_-3_0) at (-8.487049, -4.900000) {};
\node (-3_-2_1) at (-8.487049, -3.500000) {};
\node (-2_-1_0) at (-6.062178, -0.700000) {};
\node (-2_-1_1) at (-4.849742, -1.400000) {};
\node (-1_-2_0) at (-4.849742, -2.800000) {};
\node (-3_-1_1) at (-7.274613, -1.400000) {};
\node (-2_0_0) at (-4.849742, 1.400000) {};
\node (-3_0_1) at (-6.062178, 0.700000) {};
\node (-3_1_0) at (-6.062178, 3.500000) {};
\node (-3_1_1) at (-4.849742, 2.800000) {};
\node (-3_0_0) at (-7.274613, 1.400000) {};
\node (-4_1_1) at (-7.274613, 2.800000) {};
\node (-4_2_0) at (-7.274613, 5.600000) {};
\node (-4_2_1) at (-6.062178, 4.900000) {};
\node (-4_1_0) at (-8.487049, 3.500000) {};
\node (-5_2_1) at (-8.487049, 4.900000) {};
\draw[periedge] (-1_0_1) -- (0_0_0);
\draw[periedge] (0_0_0) -- (0_0_1);
\draw[periedge] (0_0_1) -- (1_-1_0);
\draw[periedge] (0_-1_1) -- (1_-1_0);
\draw[periedge] (0_-1_0) -- (0_-1_1);
\draw[edge] (-1_0_1) -- (0_-1_0);
\draw[periedge] (-2_0_1) -- (-1_0_0);
\draw[periedge] (-1_0_0) -- (-1_0_1);
\draw[periedge] (-1_-1_1) -- (0_-1_0);
\draw[periedge] (-1_-1_0) -- (-1_-1_1);
\draw[edge] (-2_0_1) -- (-1_-1_0);
\draw[periedge] (-3_-2_1) -- (-2_-2_0);
\draw[edge] (-2_-2_0) -- (-2_-2_1);
\draw[periedge] (-2_-2_1) -- (-1_-3_0);
\draw[periedge] (-2_-3_1) -- (-1_-3_0);
\draw[periedge] (-2_-3_0) -- (-2_-3_1);
\draw[periedge] (-3_-2_1) -- (-2_-3_0);
\draw[periedge] (-3_-1_1) -- (-2_-1_0);
\draw[edge] (-2_-1_0) -- (-2_-1_1);
\draw[periedge] (-2_-1_1) -- (-1_-2_0);
\draw[periedge] (-2_-2_1) -- (-1_-2_0);
\draw[periedge] (-3_-1_1) -- (-2_-2_0);
\draw[edge] (-3_0_1) -- (-2_0_0);
\draw[periedge] (-2_0_0) -- (-2_0_1);
\draw[periedge] (-2_-1_1) -- (-1_-1_0);
\draw[periedge] (-3_0_1) -- (-2_-1_0);
\draw[edge] (-4_1_1) -- (-3_1_0);
\draw[periedge] (-3_1_0) -- (-3_1_1);
\draw[periedge] (-3_1_1) -- (-2_0_0);
\draw[periedge] (-3_0_0) -- (-3_0_1);
\draw[periedge] (-4_1_1) -- (-3_0_0);
\draw[periedge] (-5_2_1) -- (-4_2_0);
\draw[periedge] (-4_2_0) -- (-4_2_1);
\draw[periedge] (-4_2_1) -- (-3_1_0);
\draw[periedge] (-4_1_0) -- (-4_1_1);
\draw[periedge] (-5_2_1) -- (-4_1_0);
\end{tikzpicture} & 
\begin{tikzpicture}[scale=\scfactor,rotate=60]
\tikzstyle{every node} = [inner sep=1, draw, circle, fill=blue!50!white]
\tikzstyle{edge} = [draw, line width=1.0]
\tikzstyle{periedge} = [draw, line width=1.0]
\node (0_0_0) at (0.000000, 1.400000) {};
\node (0_0_1) at (1.212436, 0.700000) {};
\node (1_-1_0) at (1.212436, -0.700000) {};
\node (0_-1_1) at (0.000000, -1.400000) {};
\node (0_-1_0) at (-1.212436, -0.700000) {};
\node (-1_0_1) at (-1.212436, 0.700000) {};
\node (-1_-1_0) at (-3.637307, -0.700000) {};
\node (-1_-1_1) at (-2.424871, -1.400000) {};
\node (0_-2_0) at (-2.424871, -2.800000) {};
\node (-1_-2_1) at (-3.637307, -3.500000) {};
\node (-1_-2_0) at (-4.849742, -2.800000) {};
\node (-2_-1_1) at (-4.849742, -1.400000) {};
\node (-1_0_0) at (-2.424871, 1.400000) {};
\node (-2_0_1) at (-3.637307, 0.700000) {};
\node (-1_1_0) at (-1.212436, 3.500000) {};
\node (-1_1_1) at (0.000000, 2.800000) {};
\node (-2_1_1) at (-2.424871, 2.800000) {};
\node (-2_-2_0) at (-7.274613, -2.800000) {};
\node (-2_-2_1) at (-6.062178, -3.500000) {};
\node (-1_-3_0) at (-6.062178, -4.900000) {};
\node (-2_-3_1) at (-7.274613, -5.600000) {};
\node (-2_-3_0) at (-8.487049, -4.900000) {};
\node (-3_-2_1) at (-8.487049, -3.500000) {};
\node (-2_-1_0) at (-6.062178, -0.700000) {};
\node (-3_-1_1) at (-7.274613, -1.400000) {};
\node (-3_-1_0) at (-8.487049, -0.700000) {};
\node (-3_-2_0) at (-9.699485, -2.800000) {};
\node (-4_-1_1) at (-9.699485, -1.400000) {};
\draw[edge] (-1_0_1) -- (0_0_0);
\draw[periedge] (0_0_0) -- (0_0_1);
\draw[periedge] (0_0_1) -- (1_-1_0);
\draw[periedge] (0_-1_1) -- (1_-1_0);
\draw[periedge] (0_-1_0) -- (0_-1_1);
\draw[edge] (-1_0_1) -- (0_-1_0);
\draw[periedge] (-2_-1_1) -- (-1_-1_0);
\draw[edge] (-1_-1_0) -- (-1_-1_1);
\draw[periedge] (-1_-1_1) -- (0_-2_0);
\draw[periedge] (-1_-2_1) -- (0_-2_0);
\draw[periedge] (-1_-2_0) -- (-1_-2_1);
\draw[edge] (-2_-1_1) -- (-1_-2_0);
\draw[periedge] (-2_0_1) -- (-1_0_0);
\draw[edge] (-1_0_0) -- (-1_0_1);
\draw[periedge] (-1_-1_1) -- (0_-1_0);
\draw[periedge] (-2_0_1) -- (-1_-1_0);
\draw[periedge] (-2_1_1) -- (-1_1_0);
\draw[periedge] (-1_1_0) -- (-1_1_1);
\draw[periedge] (-1_1_1) -- (0_0_0);
\draw[periedge] (-2_1_1) -- (-1_0_0);
\draw[edge] (-3_-2_1) -- (-2_-2_0);
\draw[edge] (-2_-2_0) -- (-2_-2_1);
\draw[periedge] (-2_-2_1) -- (-1_-3_0);
\draw[periedge] (-2_-3_1) -- (-1_-3_0);
\draw[periedge] (-2_-3_0) -- (-2_-3_1);
\draw[periedge] (-3_-2_1) -- (-2_-3_0);
\draw[periedge] (-3_-1_1) -- (-2_-1_0);
\draw[periedge] (-2_-1_0) -- (-2_-1_1);
\draw[periedge] (-2_-2_1) -- (-1_-2_0);
\draw[edge] (-3_-1_1) -- (-2_-2_0);
\draw[periedge] (-4_-1_1) -- (-3_-1_0);
\draw[periedge] (-3_-1_0) -- (-3_-1_1);
\draw[periedge] (-3_-2_0) -- (-3_-2_1);
\draw[periedge] (-4_-1_1) -- (-3_-2_0);
\end{tikzpicture} & 
\begin{tikzpicture}[scale=\scfactor,rotate=60]
\tikzstyle{every node} = [inner sep=1, draw, circle, fill=blue!50!white]
\tikzstyle{edge} = [draw, line width=1.0]
\tikzstyle{periedge} = [draw, line width=1.0]
\node (0_0_0) at (0.000000, 1.400000) {};
\node (0_0_1) at (1.212436, 0.700000) {};
\node (1_-1_0) at (1.212436, -0.700000) {};
\node (0_-1_1) at (0.000000, -1.400000) {};
\node (0_-1_0) at (-1.212436, -0.700000) {};
\node (-1_0_1) at (-1.212436, 0.700000) {};
\node (-1_-3_0) at (-6.062178, -4.900000) {};
\node (-1_-3_1) at (-4.849742, -5.600000) {};
\node (0_-4_0) at (-4.849742, -7.000000) {};
\node (-1_-4_1) at (-6.062178, -7.700000) {};
\node (-1_-4_0) at (-7.274613, -7.000000) {};
\node (-2_-3_1) at (-7.274613, -5.600000) {};
\node (-1_-2_0) at (-4.849742, -2.800000) {};
\node (-1_-2_1) at (-3.637307, -3.500000) {};
\node (0_-3_0) at (-3.637307, -4.900000) {};
\node (-2_-2_1) at (-6.062178, -3.500000) {};
\node (-1_-1_0) at (-3.637307, -0.700000) {};
\node (-1_-1_1) at (-2.424871, -1.400000) {};
\node (0_-2_0) at (-2.424871, -2.800000) {};
\node (-2_-1_1) at (-4.849742, -1.400000) {};
\node (-1_0_0) at (-2.424871, 1.400000) {};
\node (-2_0_1) at (-3.637307, 0.700000) {};
\node (-1_1_0) at (-1.212436, 3.500000) {};
\node (-1_1_1) at (0.000000, 2.800000) {};
\node (-2_1_1) at (-2.424871, 2.800000) {};
\node (-2_-2_0) at (-7.274613, -2.800000) {};
\node (-2_-3_0) at (-8.487049, -4.900000) {};
\node (-3_-2_1) at (-8.487049, -3.500000) {};
\draw[edge] (-1_0_1) -- (0_0_0);
\draw[periedge] (0_0_0) -- (0_0_1);
\draw[periedge] (0_0_1) -- (1_-1_0);
\draw[periedge] (0_-1_1) -- (1_-1_0);
\draw[periedge] (0_-1_0) -- (0_-1_1);
\draw[edge] (-1_0_1) -- (0_-1_0);
\draw[edge] (-2_-3_1) -- (-1_-3_0);
\draw[edge] (-1_-3_0) -- (-1_-3_1);
\draw[periedge] (-1_-3_1) -- (0_-4_0);
\draw[periedge] (-1_-4_1) -- (0_-4_0);
\draw[periedge] (-1_-4_0) -- (-1_-4_1);
\draw[periedge] (-2_-3_1) -- (-1_-4_0);
\draw[periedge] (-2_-2_1) -- (-1_-2_0);
\draw[edge] (-1_-2_0) -- (-1_-2_1);
\draw[periedge] (-1_-2_1) -- (0_-3_0);
\draw[periedge] (-1_-3_1) -- (0_-3_0);
\draw[edge] (-2_-2_1) -- (-1_-3_0);
\draw[periedge] (-2_-1_1) -- (-1_-1_0);
\draw[edge] (-1_-1_0) -- (-1_-1_1);
\draw[periedge] (-1_-1_1) -- (0_-2_0);
\draw[periedge] (-1_-2_1) -- (0_-2_0);
\draw[periedge] (-2_-1_1) -- (-1_-2_0);
\draw[periedge] (-2_0_1) -- (-1_0_0);
\draw[edge] (-1_0_0) -- (-1_0_1);
\draw[periedge] (-1_-1_1) -- (0_-1_0);
\draw[periedge] (-2_0_1) -- (-1_-1_0);
\draw[periedge] (-2_1_1) -- (-1_1_0);
\draw[periedge] (-1_1_0) -- (-1_1_1);
\draw[periedge] (-1_1_1) -- (0_0_0);
\draw[periedge] (-2_1_1) -- (-1_0_0);
\draw[periedge] (-3_-2_1) -- (-2_-2_0);
\draw[periedge] (-2_-2_0) -- (-2_-2_1);
\draw[periedge] (-2_-3_0) -- (-2_-3_1);
\draw[periedge] (-3_-2_1) -- (-2_-3_0);
\end{tikzpicture} \\
\text{(a)}\ 5,\, \mathbf{C}_\mathrm{s},\, 1,\, 9 & \text{(b)}\ 7,\, \mathbf{D}_\mathrm{3h},\, 3,\, 28 & \text{(c)}\ 7,\, \mathbf{C}_\mathrm{2v},\, 3,\, 9 & \text{(d)}\ 7,\, \mathbf{C}_\mathrm{2h},\, 2,\, 9 \\[6pt]
\begin{tikzpicture}[scale=\scfactor,rotate=60]
\tikzstyle{every node} = [inner sep=1.0, draw, circle, fill=blue!50!white]
\tikzstyle{edge} = [draw, line width=1.0]
\tikzstyle{periedge} = [draw, line width=1.0]
\node (0_0_0) at (0.000000, 1.400000) {};
\node (0_0_1) at (1.212436, 0.700000) {};
\node (1_-1_0) at (1.212436, -0.700000) {};
\node (0_-1_1) at (0.000000, -1.400000) {};
\node (0_-1_0) at (-1.212436, -0.700000) {};
\node (-1_0_1) at (-1.212436, 0.700000) {};
\node (-1_-3_0) at (-6.062178, -4.900000) {};
\node (-1_-3_1) at (-4.849742, -5.600000) {};
\node (0_-4_0) at (-4.849742, -7.000000) {};
\node (-1_-4_1) at (-6.062178, -7.700000) {};
\node (-1_-4_0) at (-7.274613, -7.000000) {};
\node (-2_-3_1) at (-7.274613, -5.600000) {};
\node (-1_-2_0) at (-4.849742, -2.800000) {};
\node (-1_-2_1) at (-3.637307, -3.500000) {};
\node (0_-3_0) at (-3.637307, -4.900000) {};
\node (-2_-2_1) at (-6.062178, -3.500000) {};
\node (-1_-1_0) at (-3.637307, -0.700000) {};
\node (-1_-1_1) at (-2.424871, -1.400000) {};
\node (0_-2_0) at (-2.424871, -2.800000) {};
\node (-2_-1_1) at (-4.849742, -1.400000) {};
\node (-1_0_0) at (-2.424871, 1.400000) {};
\node (-2_0_1) at (-3.637307, 0.700000) {};
\node (-2_-3_0) at (-8.487049, -4.900000) {};
\node (-2_-4_1) at (-8.487049, -7.700000) {};
\node (-2_-4_0) at (-9.699485, -7.000000) {};
\node (-3_-3_1) at (-9.699485, -5.600000) {};
\node (-3_-5_0) at (-13.336791, -9.100000) {};
\node (-3_-5_1) at (-12.124356, -9.800000) {};
\node (-2_-6_0) at (-12.124356, -11.200000) {};
\node (-3_-6_1) at (-13.336791, -11.900000) {};
\node (-3_-6_0) at (-14.549227, -11.200000) {};
\node (-4_-5_1) at (-14.549227, -9.800000) {};
\node (-3_-4_0) at (-12.124356, -7.000000) {};
\node (-3_-4_1) at (-10.911920, -7.700000) {};
\node (-2_-5_0) at (-10.911920, -9.100000) {};
\node (-4_-4_1) at (-13.336791, -7.700000) {};
\node (-3_-3_0) at (-10.911920, -4.900000) {};
\node (-4_-3_1) at (-12.124356, -5.600000) {};
\draw[periedge] (-1_0_1) -- (0_0_0);
\draw[periedge] (0_0_0) -- (0_0_1);
\draw[periedge] (0_0_1) -- (1_-1_0);
\draw[periedge] (0_-1_1) -- (1_-1_0);
\draw[periedge] (0_-1_0) -- (0_-1_1);
\draw[edge] (-1_0_1) -- (0_-1_0);
\draw[periedge] (-2_-3_1) -- (-1_-3_0);
\draw[edge] (-1_-3_0) -- (-1_-3_1);
\draw[periedge] (-1_-3_1) -- (0_-4_0);
\draw[periedge] (-1_-4_1) -- (0_-4_0);
\draw[periedge] (-1_-4_0) -- (-1_-4_1);
\draw[edge] (-2_-3_1) -- (-1_-4_0);
\draw[periedge] (-2_-2_1) -- (-1_-2_0);
\draw[edge] (-1_-2_0) -- (-1_-2_1);
\draw[periedge] (-1_-2_1) -- (0_-3_0);
\draw[periedge] (-1_-3_1) -- (0_-3_0);
\draw[periedge] (-2_-2_1) -- (-1_-3_0);
\draw[periedge] (-2_-1_1) -- (-1_-1_0);
\draw[edge] (-1_-1_0) -- (-1_-1_1);
\draw[periedge] (-1_-1_1) -- (0_-2_0);
\draw[periedge] (-1_-2_1) -- (0_-2_0);
\draw[periedge] (-2_-1_1) -- (-1_-2_0);
\draw[periedge] (-2_0_1) -- (-1_0_0);
\draw[periedge] (-1_0_0) -- (-1_0_1);
\draw[periedge] (-1_-1_1) -- (0_-1_0);
\draw[periedge] (-2_0_1) -- (-1_-1_0);
\draw[periedge] (-3_-3_1) -- (-2_-3_0);
\draw[periedge] (-2_-3_0) -- (-2_-3_1);
\draw[periedge] (-2_-4_1) -- (-1_-4_0);
\draw[periedge] (-2_-4_0) -- (-2_-4_1);
\draw[edge] (-3_-3_1) -- (-2_-4_0);
\draw[periedge] (-4_-5_1) -- (-3_-5_0);
\draw[edge] (-3_-5_0) -- (-3_-5_1);
\draw[periedge] (-3_-5_1) -- (-2_-6_0);
\draw[periedge] (-3_-6_1) -- (-2_-6_0);
\draw[periedge] (-3_-6_0) -- (-3_-6_1);
\draw[periedge] (-4_-5_1) -- (-3_-6_0);
\draw[periedge] (-4_-4_1) -- (-3_-4_0);
\draw[edge] (-3_-4_0) -- (-3_-4_1);
\draw[periedge] (-3_-4_1) -- (-2_-5_0);
\draw[periedge] (-3_-5_1) -- (-2_-5_0);
\draw[periedge] (-4_-4_1) -- (-3_-5_0);
\draw[periedge] (-4_-3_1) -- (-3_-3_0);
\draw[periedge] (-3_-3_0) -- (-3_-3_1);
\draw[periedge] (-3_-4_1) -- (-2_-4_0);
\draw[periedge] (-4_-3_1) -- (-3_-4_0);
\end{tikzpicture} &
\begin{tikzpicture}[scale=\scfactor,rotate=0]
\tikzstyle{every node} = [inner sep=1, draw, circle, fill=blue!50!white]
\tikzstyle{edge} = [draw, line width=1.0]
\tikzstyle{periedge} = [draw, line width=1.0]
\node (0_-6_0) at (-7.274613, -11.200000) {};
\node (0_-6_1) at (-6.062178, -11.900000) {};
\node (1_-7_0) at (-6.062178, -13.300000) {};
\node (0_-7_1) at (-7.274613, -14.000000) {};
\node (0_-7_0) at (-8.487049, -13.300000) {};
\node (-1_-6_1) at (-8.487049, -11.900000) {};
\node (0_-5_0) at (-6.062178, -9.100000) {};
\node (0_-5_1) at (-4.849742, -9.800000) {};
\node (1_-6_0) at (-4.849742, -11.200000) {};
\node (-1_-5_1) at (-7.274613, -9.800000) {};
\node (0_-4_0) at (-4.849742, -7.000000) {};
\node (0_-4_1) at (-3.637307, -7.700000) {};
\node (1_-5_0) at (-3.637307, -9.100000) {};
\node (-1_-4_1) at (-6.062178, -7.700000) {};
\node (0_-3_0) at (-3.637307, -4.900000) {};
\node (0_-3_1) at (-2.424871, -5.600000) {};
\node (1_-4_0) at (-2.424871, -7.000000) {};
\node (-1_-3_1) at (-4.849742, -5.600000) {};
\node (0_0_0) at (0.000000, 1.400000) {};
\node (0_0_1) at (1.212436, 0.700000) {};
\node (1_-1_0) at (1.212436, -0.700000) {};
\node (0_-1_1) at (0.000000, -1.400000) {};
\node (0_-1_0) at (-1.212436, -0.700000) {};
\node (-1_0_1) at (-1.212436, 0.700000) {};
\node (-1_-2_0) at (-4.849742, -2.800000) {};
\node (-1_-2_1) at (-3.637307, -3.500000) {};
\node (-1_-3_0) at (-6.062178, -4.900000) {};
\node (-2_-2_1) at (-6.062178, -3.500000) {};
\node (-1_-1_0) at (-3.637307, -0.700000) {};
\node (-1_-1_1) at (-2.424871, -1.400000) {};
\node (0_-2_0) at (-2.424871, -2.800000) {};
\node (-2_-1_1) at (-4.849742, -1.400000) {};
\node (-1_0_0) at (-2.424871, 1.400000) {};
\node (-2_0_1) at (-3.637307, 0.700000) {};
\node (1_-7_1) at (-4.849742, -14.000000) {};
\node (2_-8_0) at (-4.849742, -15.400000) {};
\node (1_-8_1) at (-6.062178, -16.100000) {};
\node (1_-8_0) at (-7.274613, -15.400000) {};
\draw[periedge] (-1_-6_1) -- (0_-6_0);
\draw[edge] (0_-6_0) -- (0_-6_1);
\draw[periedge] (0_-6_1) -- (1_-7_0);
\draw[edge] (0_-7_1) -- (1_-7_0);
\draw[periedge] (0_-7_0) -- (0_-7_1);
\draw[periedge] (-1_-6_1) -- (0_-7_0);
\draw[periedge] (-1_-5_1) -- (0_-5_0);
\draw[edge] (0_-5_0) -- (0_-5_1);
\draw[periedge] (0_-5_1) -- (1_-6_0);
\draw[periedge] (0_-6_1) -- (1_-6_0);
\draw[periedge] (-1_-5_1) -- (0_-6_0);
\draw[periedge] (-1_-4_1) -- (0_-4_0);
\draw[edge] (0_-4_0) -- (0_-4_1);
\draw[periedge] (0_-4_1) -- (1_-5_0);
\draw[periedge] (0_-5_1) -- (1_-5_0);
\draw[periedge] (-1_-4_1) -- (0_-5_0);
\draw[edge] (-1_-3_1) -- (0_-3_0);
\draw[periedge] (0_-3_0) -- (0_-3_1);
\draw[periedge] (0_-3_1) -- (1_-4_0);
\draw[periedge] (0_-4_1) -- (1_-4_0);
\draw[periedge] (-1_-3_1) -- (0_-4_0);
\draw[periedge] (-1_0_1) -- (0_0_0);
\draw[periedge] (0_0_0) -- (0_0_1);
\draw[periedge] (0_0_1) -- (1_-1_0);
\draw[periedge] (0_-1_1) -- (1_-1_0);
\draw[periedge] (0_-1_0) -- (0_-1_1);
\draw[edge] (-1_0_1) -- (0_-1_0);
\draw[periedge] (-2_-2_1) -- (-1_-2_0);
\draw[edge] (-1_-2_0) -- (-1_-2_1);
\draw[periedge] (-1_-2_1) -- (0_-3_0);
\draw[periedge] (-1_-3_0) -- (-1_-3_1);
\draw[periedge] (-2_-2_1) -- (-1_-3_0);
\draw[periedge] (-2_-1_1) -- (-1_-1_0);
\draw[edge] (-1_-1_0) -- (-1_-1_1);
\draw[periedge] (-1_-1_1) -- (0_-2_0);
\draw[periedge] (-1_-2_1) -- (0_-2_0);
\draw[periedge] (-2_-1_1) -- (-1_-2_0);
\draw[periedge] (-2_0_1) -- (-1_0_0);
\draw[periedge] (-1_0_0) -- (-1_0_1);
\draw[periedge] (-1_-1_1) -- (0_-1_0);
\draw[periedge] (-2_0_1) -- (-1_-1_0);
\draw[periedge] (1_-7_0) -- (1_-7_1);
\draw[periedge] (1_-7_1) -- (2_-8_0);
\draw[periedge] (1_-8_1) -- (2_-8_0);
\draw[periedge] (1_-8_0) -- (1_-8_1);
\draw[periedge] (0_-7_1) -- (1_-8_0);
\end{tikzpicture} &
\begin{tikzpicture}[scale=\scfactor]
\tikzstyle{every node} = [inner sep=1, draw, circle, fill=blue!50!white]
\tikzstyle{edge} = [draw, line width=1.0]
\tikzstyle{periedge} = [draw, line width=1.0]
\node (0_-3_0) at (-3.637307, -4.900000) {};
\node (0_-3_1) at (-2.424871, -5.600000) {};
\node (1_-4_0) at (-2.424871, -7.000000) {};
\node (0_-4_1) at (-3.637307, -7.700000) {};
\node (0_-4_0) at (-4.849742, -7.000000) {};
\node (-1_-3_1) at (-4.849742, -5.600000) {};
\node (0_0_0) at (0.000000, 1.400000) {};
\node (0_0_1) at (1.212436, 0.700000) {};
\node (1_-1_0) at (1.212436, -0.700000) {};
\node (0_-1_1) at (0.000000, -1.400000) {};
\node (0_-1_0) at (-1.212436, -0.700000) {};
\node (-1_0_1) at (-1.212436, 0.700000) {};
\node (-1_-2_0) at (-4.849742, -2.800000) {};
\node (-1_-2_1) at (-3.637307, -3.500000) {};
\node (-1_-3_0) at (-6.062178, -4.900000) {};
\node (-2_-2_1) at (-6.062178, -3.500000) {};
\node (-1_-1_0) at (-3.637307, -0.700000) {};
\node (-1_-1_1) at (-2.424871, -1.400000) {};
\node (0_-2_0) at (-2.424871, -2.800000) {};
\node (-2_-1_1) at (-4.849742, -1.400000) {};
\node (-1_0_0) at (-2.424871, 1.400000) {};
\node (-2_0_1) at (-3.637307, 0.700000) {};
\node (1_-4_1) at (-1.212436, -7.700000) {};
\node (2_-5_0) at (-1.212436, -9.100000) {};
\node (1_-5_1) at (-2.424871, -9.800000) {};
\node (1_-5_0) at (-3.637307, -9.100000) {};
\node (2_-5_1) at (0.000000, -9.800000) {};
\node (3_-6_0) at (0.000000, -11.200000) {};
\node (2_-6_1) at (-1.212436, -11.900000) {};
\node (2_-6_0) at (-2.424871, -11.200000) {};
\node (3_-7_0) at (-1.212436, -13.300000) {};
\node (3_-7_1) at (0.000000, -14.000000) {};
\node (4_-8_0) at (0.000000, -15.400000) {};
\node (3_-8_1) at (-1.212436, -16.100000) {};
\node (3_-8_0) at (-2.424871, -15.400000) {};
\node (2_-7_1) at (-2.424871, -14.000000) {};
\node (3_-6_1) at (1.212436, -11.900000) {};
\node (4_-7_0) at (1.212436, -13.300000) {};
\draw[edge] (-1_-3_1) -- (0_-3_0);
\draw[periedge] (0_-3_0) -- (0_-3_1);
\draw[periedge] (0_-3_1) -- (1_-4_0);
\draw[edge] (0_-4_1) -- (1_-4_0);
\draw[periedge] (0_-4_0) -- (0_-4_1);
\draw[periedge] (-1_-3_1) -- (0_-4_0);
\draw[periedge] (-1_0_1) -- (0_0_0);
\draw[periedge] (0_0_0) -- (0_0_1);
\draw[periedge] (0_0_1) -- (1_-1_0);
\draw[periedge] (0_-1_1) -- (1_-1_0);
\draw[periedge] (0_-1_0) -- (0_-1_1);
\draw[edge] (-1_0_1) -- (0_-1_0);
\draw[periedge] (-2_-2_1) -- (-1_-2_0);
\draw[edge] (-1_-2_0) -- (-1_-2_1);
\draw[periedge] (-1_-2_1) -- (0_-3_0);
\draw[periedge] (-1_-3_0) -- (-1_-3_1);
\draw[periedge] (-2_-2_1) -- (-1_-3_0);
\draw[periedge] (-2_-1_1) -- (-1_-1_0);
\draw[edge] (-1_-1_0) -- (-1_-1_1);
\draw[periedge] (-1_-1_1) -- (0_-2_0);
\draw[periedge] (-1_-2_1) -- (0_-2_0);
\draw[periedge] (-2_-1_1) -- (-1_-2_0);
\draw[periedge] (-2_0_1) -- (-1_0_0);
\draw[periedge] (-1_0_0) -- (-1_0_1);
\draw[periedge] (-1_-1_1) -- (0_-1_0);
\draw[periedge] (-2_0_1) -- (-1_-1_0);
\draw[periedge] (1_-4_0) -- (1_-4_1);
\draw[periedge] (1_-4_1) -- (2_-5_0);
\draw[edge] (1_-5_1) -- (2_-5_0);
\draw[periedge] (1_-5_0) -- (1_-5_1);
\draw[periedge] (0_-4_1) -- (1_-5_0);
\draw[periedge] (2_-5_0) -- (2_-5_1);
\draw[periedge] (2_-5_1) -- (3_-6_0);
\draw[edge] (2_-6_1) -- (3_-6_0);
\draw[periedge] (2_-6_0) -- (2_-6_1);
\draw[periedge] (1_-5_1) -- (2_-6_0);
\draw[periedge] (2_-7_1) -- (3_-7_0);
\draw[edge] (3_-7_0) -- (3_-7_1);
\draw[periedge] (3_-7_1) -- (4_-8_0);
\draw[periedge] (3_-8_1) -- (4_-8_0);
\draw[periedge] (3_-8_0) -- (3_-8_1);
\draw[periedge] (2_-7_1) -- (3_-8_0);
\draw[periedge] (3_-6_0) -- (3_-6_1);
\draw[periedge] (3_-6_1) -- (4_-7_0);
\draw[periedge] (3_-7_1) -- (4_-7_0);
\draw[periedge] (2_-6_1) -- (3_-7_0);
\end{tikzpicture} & 
\begin{tikzpicture}[scale=\scfactor,rotate=120]
\tikzstyle{every node} = [inner sep=1, draw, circle, fill=blue!50!white]
\tikzstyle{edge} = [draw, line width=1.0]
\tikzstyle{periedge} = [draw, line width=1.0]
\node (0_0_0) at (0.000000, 1.400000) {};
\node (0_0_1) at (1.212436, 0.700000) {};
\node (1_-1_0) at (1.212436, -0.700000) {};
\node (0_-1_1) at (0.000000, -1.400000) {};
\node (0_-1_0) at (-1.212436, -0.700000) {};
\node (-1_0_1) at (-1.212436, 0.700000) {};
\node (-1_0_0) at (-2.424871, 1.400000) {};
\node (-1_-1_1) at (-2.424871, -1.400000) {};
\node (-1_-1_0) at (-3.637307, -0.700000) {};
\node (-2_0_1) at (-3.637307, 0.700000) {};
\node (-2_0_0) at (-4.849742, 1.400000) {};
\node (-2_-1_1) at (-4.849742, -1.400000) {};
\node (-2_-1_0) at (-6.062178, -0.700000) {};
\node (-3_0_1) at (-6.062178, 0.700000) {};
\node (-3_0_0) at (-7.274613, 1.400000) {};
\node (-3_-1_1) at (-7.274613, -1.400000) {};
\node (-3_-1_0) at (-8.487049, -0.700000) {};
\node (-4_0_1) at (-8.487049, 0.700000) {};
\node (-4_1_0) at (-8.487049, 3.500000) {};
\node (-4_1_1) at (-7.274613, 2.800000) {};
\node (-4_0_0) at (-9.699485, 1.400000) {};
\node (-5_1_1) at (-9.699485, 2.800000) {};
\node (-5_1_0) at (-10.911920, 3.500000) {};
\node (-5_0_1) at (-10.911920, 0.700000) {};
\node (-5_0_0) at (-12.124356, 1.400000) {};
\node (-6_1_1) at (-12.124356, 2.800000) {};
\node (-6_2_0) at (-12.124356, 5.600000) {};
\node (-6_2_1) at (-10.911920, 4.900000) {};
\node (-6_1_0) at (-13.336791, 3.500000) {};
\node (-7_2_1) at (-13.336791, 4.900000) {};
\node (-6_3_0) at (-10.911920, 7.700000) {};
\node (-6_3_1) at (-9.699485, 7.000000) {};
\node (-5_2_0) at (-9.699485, 5.600000) {};
\node (-7_3_1) at (-12.124356, 7.000000) {};
\node (-7_2_0) at (-14.549227, 5.600000) {};
\node (-7_1_1) at (-14.549227, 2.800000) {};
\node (-7_1_0) at (-15.761662, 3.500000) {};
\node (-8_2_1) at (-15.761662, 4.900000) {};
\draw[periedge] (-1_0_1) -- (0_0_0);
\draw[periedge] (0_0_0) -- (0_0_1);
\draw[periedge] (0_0_1) -- (1_-1_0);
\draw[periedge] (0_-1_1) -- (1_-1_0);
\draw[periedge] (0_-1_0) -- (0_-1_1);
\draw[edge] (-1_0_1) -- (0_-1_0);
\draw[periedge] (-2_0_1) -- (-1_0_0);
\draw[periedge] (-1_0_0) -- (-1_0_1);
\draw[periedge] (-1_-1_1) -- (0_-1_0);
\draw[periedge] (-1_-1_0) -- (-1_-1_1);
\draw[edge] (-2_0_1) -- (-1_-1_0);
\draw[periedge] (-3_0_1) -- (-2_0_0);
\draw[periedge] (-2_0_0) -- (-2_0_1);
\draw[periedge] (-2_-1_1) -- (-1_-1_0);
\draw[periedge] (-2_-1_0) -- (-2_-1_1);
\draw[edge] (-3_0_1) -- (-2_-1_0);
\draw[edge] (-4_0_1) -- (-3_0_0);
\draw[periedge] (-3_0_0) -- (-3_0_1);
\draw[periedge] (-3_-1_1) -- (-2_-1_0);
\draw[periedge] (-3_-1_0) -- (-3_-1_1);
\draw[periedge] (-4_0_1) -- (-3_-1_0);
\draw[periedge] (-5_1_1) -- (-4_1_0);
\draw[periedge] (-4_1_0) -- (-4_1_1);
\draw[periedge] (-4_1_1) -- (-3_0_0);
\draw[periedge] (-4_0_0) -- (-4_0_1);
\draw[edge] (-5_1_1) -- (-4_0_0);
\draw[edge] (-6_1_1) -- (-5_1_0);
\draw[periedge] (-5_1_0) -- (-5_1_1);
\draw[periedge] (-5_0_1) -- (-4_0_0);
\draw[periedge] (-5_0_0) -- (-5_0_1);
\draw[periedge] (-6_1_1) -- (-5_0_0);
\draw[periedge] (-7_2_1) -- (-6_2_0);
\draw[edge] (-6_2_0) -- (-6_2_1);
\draw[periedge] (-6_2_1) -- (-5_1_0);
\draw[periedge] (-6_1_0) -- (-6_1_1);
\draw[edge] (-7_2_1) -- (-6_1_0);
\draw[periedge] (-7_3_1) -- (-6_3_0);
\draw[periedge] (-6_3_0) -- (-6_3_1);
\draw[periedge] (-6_3_1) -- (-5_2_0);
\draw[periedge] (-6_2_1) -- (-5_2_0);
\draw[periedge] (-7_3_1) -- (-6_2_0);
\draw[periedge] (-8_2_1) -- (-7_2_0);
\draw[periedge] (-7_2_0) -- (-7_2_1);
\draw[periedge] (-7_1_1) -- (-6_1_0);
\draw[periedge] (-7_1_0) -- (-7_1_1);
\draw[periedge] (-8_2_1) -- (-7_1_0);
\end{tikzpicture} \\
\text{(e)}\ 9,\, \mathbf{C}_\mathrm{s},\, 3,\, 55 & \text{(f)}\ 9,\, \mathbf{C}_\mathrm{s},\, 4,\, 59 & \text{(g)}\ 9,\, \mathbf{C}_\mathrm{s},\, 3,\, 49 & \text{(h)}\ 9,\, \mathbf{C}_\mathrm{s},\, 6,\, 65 \\[6pt]
\begin{tikzpicture}[scale=\scfactor,rotate=60]
\tikzstyle{every node} = [inner sep=1, draw, circle, fill=blue!50!white]
\tikzstyle{edge} = [draw, line width=1.0]
\tikzstyle{periedge} = [draw, line width=1.0]
\node (0_-4_0) at (-4.849742, -7.000000) {};
\node (0_-4_1) at (-3.637307, -7.700000) {};
\node (1_-5_0) at (-3.637307, -9.100000) {};
\node (0_-5_1) at (-4.849742, -9.800000) {};
\node (0_-5_0) at (-6.062178, -9.100000) {};
\node (-1_-4_1) at (-6.062178, -7.700000) {};
\node (0_0_0) at (0.000000, 1.400000) {};
\node (0_0_1) at (1.212436, 0.700000) {};
\node (1_-1_0) at (1.212436, -0.700000) {};
\node (0_-1_1) at (0.000000, -1.400000) {};
\node (0_-1_0) at (-1.212436, -0.700000) {};
\node (-1_0_1) at (-1.212436, 0.700000) {};
\node (-1_-3_0) at (-6.062178, -4.900000) {};
\node (-1_-3_1) at (-4.849742, -5.600000) {};
\node (-1_-4_0) at (-7.274613, -7.000000) {};
\node (-2_-3_1) at (-7.274613, -5.600000) {};
\node (-1_-2_0) at (-4.849742, -2.800000) {};
\node (-1_-2_1) at (-3.637307, -3.500000) {};
\node (0_-3_0) at (-3.637307, -4.900000) {};
\node (-2_-2_1) at (-6.062178, -3.500000) {};
\node (-1_-1_0) at (-3.637307, -0.700000) {};
\node (-1_-1_1) at (-2.424871, -1.400000) {};
\node (0_-2_0) at (-2.424871, -2.800000) {};
\node (-2_-1_1) at (-4.849742, -1.400000) {};
\node (-1_0_0) at (-2.424871, 1.400000) {};
\node (-2_0_1) at (-3.637307, 0.700000) {};
\node (-2_-5_0) at (-10.911920, -9.100000) {};
\node (-2_-5_1) at (-9.699485, -9.800000) {};
\node (-1_-6_0) at (-9.699485, -11.200000) {};
\node (-2_-6_1) at (-10.911920, -11.900000) {};
\node (-2_-6_0) at (-12.124356, -11.200000) {};
\node (-3_-5_1) at (-12.124356, -9.800000) {};
\node (-2_-4_0) at (-9.699485, -7.000000) {};
\node (-2_-4_1) at (-8.487049, -7.700000) {};
\node (-1_-5_0) at (-8.487049, -9.100000) {};
\node (-3_-4_1) at (-10.911920, -7.700000) {};
\node (-2_-3_0) at (-8.487049, -4.900000) {};
\node (-3_-3_1) at (-9.699485, -5.600000) {};
\draw[edge] (-1_-4_1) -- (0_-4_0);
\draw[periedge] (0_-4_0) -- (0_-4_1);
\draw[periedge] (0_-4_1) -- (1_-5_0);
\draw[periedge] (0_-5_1) -- (1_-5_0);
\draw[periedge] (0_-5_0) -- (0_-5_1);
\draw[periedge] (-1_-4_1) -- (0_-5_0);
\draw[periedge] (-1_0_1) -- (0_0_0);
\draw[periedge] (0_0_0) -- (0_0_1);
\draw[periedge] (0_0_1) -- (1_-1_0);
\draw[periedge] (0_-1_1) -- (1_-1_0);
\draw[periedge] (0_-1_0) -- (0_-1_1);
\draw[edge] (-1_0_1) -- (0_-1_0);
\draw[periedge] (-2_-3_1) -- (-1_-3_0);
\draw[edge] (-1_-3_0) -- (-1_-3_1);
\draw[periedge] (-1_-3_1) -- (0_-4_0);
\draw[periedge] (-1_-4_0) -- (-1_-4_1);
\draw[edge] (-2_-3_1) -- (-1_-4_0);
\draw[periedge] (-2_-2_1) -- (-1_-2_0);
\draw[edge] (-1_-2_0) -- (-1_-2_1);
\draw[periedge] (-1_-2_1) -- (0_-3_0);
\draw[periedge] (-1_-3_1) -- (0_-3_0);
\draw[periedge] (-2_-2_1) -- (-1_-3_0);
\draw[periedge] (-2_-1_1) -- (-1_-1_0);
\draw[edge] (-1_-1_0) -- (-1_-1_1);
\draw[periedge] (-1_-1_1) -- (0_-2_0);
\draw[periedge] (-1_-2_1) -- (0_-2_0);
\draw[periedge] (-2_-1_1) -- (-1_-2_0);
\draw[periedge] (-2_0_1) -- (-1_0_0);
\draw[periedge] (-1_0_0) -- (-1_0_1);
\draw[periedge] (-1_-1_1) -- (0_-1_0);
\draw[periedge] (-2_0_1) -- (-1_-1_0);
\draw[periedge] (-3_-5_1) -- (-2_-5_0);
\draw[edge] (-2_-5_0) -- (-2_-5_1);
\draw[periedge] (-2_-5_1) -- (-1_-6_0);
\draw[periedge] (-2_-6_1) -- (-1_-6_0);
\draw[periedge] (-2_-6_0) -- (-2_-6_1);
\draw[periedge] (-3_-5_1) -- (-2_-6_0);
\draw[periedge] (-3_-4_1) -- (-2_-4_0);
\draw[edge] (-2_-4_0) -- (-2_-4_1);
\draw[periedge] (-2_-4_1) -- (-1_-5_0);
\draw[periedge] (-2_-5_1) -- (-1_-5_0);
\draw[periedge] (-3_-4_1) -- (-2_-5_0);
\draw[periedge] (-3_-3_1) -- (-2_-3_0);
\draw[periedge] (-2_-3_0) -- (-2_-3_1);
\draw[periedge] (-2_-4_1) -- (-1_-4_0);
\draw[periedge] (-3_-3_1) -- (-2_-4_0);
\end{tikzpicture} &
\begin{tikzpicture}[scale=\scfactor,rotate=60]
\tikzstyle{every node} = [inner sep=1, draw, circle, fill=blue!50!white]
\tikzstyle{edge} = [draw, line width=1.0]
\tikzstyle{periedge} = [draw, line width=1.0]
\node (0_-6_0) at (-7.274613, -11.200000) {};
\node (0_-6_1) at (-6.062178, -11.900000) {};
\node (1_-7_0) at (-6.062178, -13.300000) {};
\node (0_-7_1) at (-7.274613, -14.000000) {};
\node (0_-7_0) at (-8.487049, -13.300000) {};
\node (-1_-6_1) at (-8.487049, -11.900000) {};
\node (0_-5_0) at (-6.062178, -9.100000) {};
\node (0_-5_1) at (-4.849742, -9.800000) {};
\node (1_-6_0) at (-4.849742, -11.200000) {};
\node (-1_-5_1) at (-7.274613, -9.800000) {};
\node (0_-4_0) at (-4.849742, -7.000000) {};
\node (0_-4_1) at (-3.637307, -7.700000) {};
\node (1_-5_0) at (-3.637307, -9.100000) {};
\node (-1_-4_1) at (-6.062178, -7.700000) {};
\node (0_-3_0) at (-3.637307, -4.900000) {};
\node (0_-3_1) at (-2.424871, -5.600000) {};
\node (1_-4_0) at (-2.424871, -7.000000) {};
\node (-1_-3_1) at (-4.849742, -5.600000) {};
\node (0_0_0) at (0.000000, 1.400000) {};
\node (0_0_1) at (1.212436, 0.700000) {};
\node (1_-1_0) at (1.212436, -0.700000) {};
\node (0_-1_1) at (0.000000, -1.400000) {};
\node (0_-1_0) at (-1.212436, -0.700000) {};
\node (-1_0_1) at (-1.212436, 0.700000) {};
\node (-1_-2_0) at (-4.849742, -2.800000) {};
\node (-1_-2_1) at (-3.637307, -3.500000) {};
\node (-1_-3_0) at (-6.062178, -4.900000) {};
\node (-2_-2_1) at (-6.062178, -3.500000) {};
\node (-1_-1_0) at (-3.637307, -0.700000) {};
\node (-1_-1_1) at (-2.424871, -1.400000) {};
\node (0_-2_0) at (-2.424871, -2.800000) {};
\node (-2_-1_1) at (-4.849742, -1.400000) {};
\node (-1_0_0) at (-2.424871, 1.400000) {};
\node (-2_0_1) at (-3.637307, 0.700000) {};
\node (1_-3_0) at (-1.212436, -4.900000) {};
\node (1_-3_1) at (-0.000000, -5.600000) {};
\node (2_-4_0) at (0.000000, -7.000000) {};
\node (1_-4_1) at (-1.212436, -7.700000) {};
\draw[periedge] (-1_-6_1) -- (0_-6_0);
\draw[edge] (0_-6_0) -- (0_-6_1);
\draw[periedge] (0_-6_1) -- (1_-7_0);
\draw[periedge] (0_-7_1) -- (1_-7_0);
\draw[periedge] (0_-7_0) -- (0_-7_1);
\draw[periedge] (-1_-6_1) -- (0_-7_0);
\draw[periedge] (-1_-5_1) -- (0_-5_0);
\draw[edge] (0_-5_0) -- (0_-5_1);
\draw[periedge] (0_-5_1) -- (1_-6_0);
\draw[periedge] (0_-6_1) -- (1_-6_0);
\draw[periedge] (-1_-5_1) -- (0_-6_0);
\draw[periedge] (-1_-4_1) -- (0_-4_0);
\draw[edge] (0_-4_0) -- (0_-4_1);
\draw[periedge] (0_-4_1) -- (1_-5_0);
\draw[periedge] (0_-5_1) -- (1_-5_0);
\draw[periedge] (-1_-4_1) -- (0_-5_0);
\draw[edge] (-1_-3_1) -- (0_-3_0);
\draw[periedge] (0_-3_0) -- (0_-3_1);
\draw[edge] (0_-3_1) -- (1_-4_0);
\draw[periedge] (0_-4_1) -- (1_-4_0);
\draw[periedge] (-1_-3_1) -- (0_-4_0);
\draw[periedge] (-1_0_1) -- (0_0_0);
\draw[periedge] (0_0_0) -- (0_0_1);
\draw[periedge] (0_0_1) -- (1_-1_0);
\draw[periedge] (0_-1_1) -- (1_-1_0);
\draw[periedge] (0_-1_0) -- (0_-1_1);
\draw[edge] (-1_0_1) -- (0_-1_0);
\draw[periedge] (-2_-2_1) -- (-1_-2_0);
\draw[edge] (-1_-2_0) -- (-1_-2_1);
\draw[periedge] (-1_-2_1) -- (0_-3_0);
\draw[periedge] (-1_-3_0) -- (-1_-3_1);
\draw[periedge] (-2_-2_1) -- (-1_-3_0);
\draw[periedge] (-2_-1_1) -- (-1_-1_0);
\draw[edge] (-1_-1_0) -- (-1_-1_1);
\draw[periedge] (-1_-1_1) -- (0_-2_0);
\draw[periedge] (-1_-2_1) -- (0_-2_0);
\draw[periedge] (-2_-1_1) -- (-1_-2_0);
\draw[periedge] (-2_0_1) -- (-1_0_0);
\draw[periedge] (-1_0_0) -- (-1_0_1);
\draw[periedge] (-1_-1_1) -- (0_-1_0);
\draw[periedge] (-2_0_1) -- (-1_-1_0);
\draw[periedge] (0_-3_1) -- (1_-3_0);
\draw[periedge] (1_-3_0) -- (1_-3_1);
\draw[periedge] (1_-3_1) -- (2_-4_0);
\draw[periedge] (1_-4_1) -- (2_-4_0);
\draw[periedge] (1_-4_0) -- (1_-4_1);
\end{tikzpicture} &
\begin{tikzpicture}[scale=\scfactor,rotate=60]
\tikzstyle{every node} = [inner sep=1, draw, circle, fill=blue!50!white]
\tikzstyle{edge} = [draw, line width=1.0]
\tikzstyle{periedge} = [draw, line width=1.0]
\node (0_0_0) at (0.000000, 1.400000) {};
\node (0_0_1) at (1.212436, 0.700000) {};
\node (1_-1_0) at (1.212436, -0.700000) {};
\node (0_-1_1) at (0.000000, -1.400000) {};
\node (0_-1_0) at (-1.212436, -0.700000) {};
\node (-1_0_1) at (-1.212436, 0.700000) {};
\node (-1_0_0) at (-2.424871, 1.400000) {};
\node (-1_-1_1) at (-2.424871, -1.400000) {};
\node (-1_-1_0) at (-3.637307, -0.700000) {};
\node (-2_0_1) at (-3.637307, 0.700000) {};
\node (-2_-3_0) at (-8.487049, -4.900000) {};
\node (-2_-3_1) at (-7.274613, -5.600000) {};
\node (-1_-4_0) at (-7.274613, -7.000000) {};
\node (-2_-4_1) at (-8.487049, -7.700000) {};
\node (-2_-4_0) at (-9.699485, -7.000000) {};
\node (-3_-3_1) at (-9.699485, -5.600000) {};
\node (-2_0_0) at (-4.849742, 1.400000) {};
\node (-2_-1_1) at (-4.849742, -1.400000) {};
\node (-2_-1_0) at (-6.062178, -0.700000) {};
\node (-3_0_1) at (-6.062178, 0.700000) {};
\node (-3_-2_0) at (-9.699485, -2.800000) {};
\node (-3_-2_1) at (-8.487049, -3.500000) {};
\node (-3_-3_0) at (-10.911920, -4.900000) {};
\node (-4_-2_1) at (-10.911920, -3.500000) {};
\node (-3_-1_0) at (-8.487049, -0.700000) {};
\node (-3_-1_1) at (-7.274613, -1.400000) {};
\node (-2_-2_0) at (-7.274613, -2.800000) {};
\node (-4_-1_1) at (-9.699485, -1.400000) {};
\node (-3_0_0) at (-7.274613, 1.400000) {};
\node (-4_0_1) at (-8.487049, 0.700000) {};
\node (-4_-2_0) at (-12.124356, -2.800000) {};
\node (-4_-3_1) at (-12.124356, -5.600000) {};
\node (-4_-3_0) at (-13.336791, -4.900000) {};
\node (-5_-2_1) at (-13.336791, -3.500000) {};
\node (-4_1_0) at (-8.487049, 3.500000) {};
\node (-4_1_1) at (-7.274613, 2.800000) {};
\node (-4_0_0) at (-9.699485, 1.400000) {};
\node (-5_1_1) at (-9.699485, 2.800000) {};
\draw[periedge] (-1_0_1) -- (0_0_0);
\draw[periedge] (0_0_0) -- (0_0_1);
\draw[periedge] (0_0_1) -- (1_-1_0);
\draw[periedge] (0_-1_1) -- (1_-1_0);
\draw[periedge] (0_-1_0) -- (0_-1_1);
\draw[edge] (-1_0_1) -- (0_-1_0);
\draw[periedge] (-2_0_1) -- (-1_0_0);
\draw[periedge] (-1_0_0) -- (-1_0_1);
\draw[periedge] (-1_-1_1) -- (0_-1_0);
\draw[periedge] (-1_-1_0) -- (-1_-1_1);
\draw[edge] (-2_0_1) -- (-1_-1_0);
\draw[edge] (-3_-3_1) -- (-2_-3_0);
\draw[periedge] (-2_-3_0) -- (-2_-3_1);
\draw[periedge] (-2_-3_1) -- (-1_-4_0);
\draw[periedge] (-2_-4_1) -- (-1_-4_0);
\draw[periedge] (-2_-4_0) -- (-2_-4_1);
\draw[periedge] (-3_-3_1) -- (-2_-4_0);
\draw[periedge] (-3_0_1) -- (-2_0_0);
\draw[periedge] (-2_0_0) -- (-2_0_1);
\draw[periedge] (-2_-1_1) -- (-1_-1_0);
\draw[periedge] (-2_-1_0) -- (-2_-1_1);
\draw[edge] (-3_0_1) -- (-2_-1_0);
\draw[periedge] (-4_-2_1) -- (-3_-2_0);
\draw[edge] (-3_-2_0) -- (-3_-2_1);
\draw[periedge] (-3_-2_1) -- (-2_-3_0);
\draw[periedge] (-3_-3_0) -- (-3_-3_1);
\draw[edge] (-4_-2_1) -- (-3_-3_0);
\draw[periedge] (-4_-1_1) -- (-3_-1_0);
\draw[edge] (-3_-1_0) -- (-3_-1_1);
\draw[periedge] (-3_-1_1) -- (-2_-2_0);
\draw[periedge] (-3_-2_1) -- (-2_-2_0);
\draw[periedge] (-4_-1_1) -- (-3_-2_0);
\draw[edge] (-4_0_1) -- (-3_0_0);
\draw[periedge] (-3_0_0) -- (-3_0_1);
\draw[periedge] (-3_-1_1) -- (-2_-1_0);
\draw[periedge] (-4_0_1) -- (-3_-1_0);
\draw[periedge] (-5_-2_1) -- (-4_-2_0);
\draw[periedge] (-4_-2_0) -- (-4_-2_1);
\draw[periedge] (-4_-3_1) -- (-3_-3_0);
\draw[periedge] (-4_-3_0) -- (-4_-3_1);
\draw[periedge] (-5_-2_1) -- (-4_-3_0);
\draw[periedge] (-5_1_1) -- (-4_1_0);
\draw[periedge] (-4_1_0) -- (-4_1_1);
\draw[periedge] (-4_1_1) -- (-3_0_0);
\draw[periedge] (-4_0_0) -- (-4_0_1);
\draw[periedge] (-5_1_1) -- (-4_0_0);
\end{tikzpicture} &
\begin{tikzpicture}[scale=\scfactor,rotate=120]
\tikzstyle{every node} = [inner sep=1, draw, circle, fill=blue!50!white]
\tikzstyle{edge} = [draw, line width=1.0]
\tikzstyle{periedge} = [draw, line width=1.0]
\node (0_0_0) at (0.000000, 1.400000) {};
\node (0_0_1) at (1.212436, 0.700000) {};
\node (1_-1_0) at (1.212436, -0.700000) {};
\node (0_-1_1) at (0.000000, -1.400000) {};
\node (0_-1_0) at (-1.212436, -0.700000) {};
\node (-1_0_1) at (-1.212436, 0.700000) {};
\node (-1_0_0) at (-2.424871, 1.400000) {};
\node (-1_-1_1) at (-2.424871, -1.400000) {};
\node (-1_-1_0) at (-3.637307, -0.700000) {};
\node (-2_0_1) at (-3.637307, 0.700000) {};
\node (-2_-1_0) at (-6.062178, -0.700000) {};
\node (-2_-1_1) at (-4.849742, -1.400000) {};
\node (-1_-2_0) at (-4.849742, -2.800000) {};
\node (-2_-2_1) at (-6.062178, -3.500000) {};
\node (-2_-2_0) at (-7.274613, -2.800000) {};
\node (-3_-1_1) at (-7.274613, -1.400000) {};
\node (-2_0_0) at (-4.849742, 1.400000) {};
\node (-3_0_1) at (-6.062178, 0.700000) {};
\node (-3_0_0) at (-7.274613, 1.400000) {};
\node (-3_-1_0) at (-8.487049, -0.700000) {};
\node (-4_0_1) at (-8.487049, 0.700000) {};
\node (-4_0_0) at (-9.699485, 1.400000) {};
\node (-4_-1_1) at (-9.699485, -1.400000) {};
\node (-4_-1_0) at (-10.911920, -0.700000) {};
\node (-5_0_1) at (-10.911920, 0.700000) {};
\node (-4_1_0) at (-8.487049, 3.500000) {};
\node (-4_1_1) at (-7.274613, 2.800000) {};
\node (-5_1_1) at (-9.699485, 2.800000) {};
\node (-5_0_0) at (-12.124356, 1.400000) {};
\node (-5_-1_1) at (-12.124356, -1.400000) {};
\node (-5_-1_0) at (-13.336791, -0.700000) {};
\node (-6_0_1) at (-13.336791, 0.700000) {};
\node (-5_2_0) at (-9.699485, 5.600000) {};
\node (-5_2_1) at (-8.487049, 4.900000) {};
\node (-5_1_0) at (-10.911920, 3.500000) {};
\node (-6_2_1) at (-10.911920, 4.900000) {};
\draw[periedge] (-1_0_1) -- (0_0_0);
\draw[periedge] (0_0_0) -- (0_0_1);
\draw[periedge] (0_0_1) -- (1_-1_0);
\draw[periedge] (0_-1_1) -- (1_-1_0);
\draw[periedge] (0_-1_0) -- (0_-1_1);
\draw[edge] (-1_0_1) -- (0_-1_0);
\draw[periedge] (-2_0_1) -- (-1_0_0);
\draw[periedge] (-1_0_0) -- (-1_0_1);
\draw[periedge] (-1_-1_1) -- (0_-1_0);
\draw[periedge] (-1_-1_0) -- (-1_-1_1);
\draw[edge] (-2_0_1) -- (-1_-1_0);
\draw[edge] (-3_-1_1) -- (-2_-1_0);
\draw[edge] (-2_-1_0) -- (-2_-1_1);
\draw[periedge] (-2_-1_1) -- (-1_-2_0);
\draw[periedge] (-2_-2_1) -- (-1_-2_0);
\draw[periedge] (-2_-2_0) -- (-2_-2_1);
\draw[periedge] (-3_-1_1) -- (-2_-2_0);
\draw[periedge] (-3_0_1) -- (-2_0_0);
\draw[periedge] (-2_0_0) -- (-2_0_1);
\draw[periedge] (-2_-1_1) -- (-1_-1_0);
\draw[edge] (-3_0_1) -- (-2_-1_0);
\draw[edge] (-4_0_1) -- (-3_0_0);
\draw[periedge] (-3_0_0) -- (-3_0_1);
\draw[periedge] (-3_-1_0) -- (-3_-1_1);
\draw[edge] (-4_0_1) -- (-3_-1_0);
\draw[periedge] (-5_0_1) -- (-4_0_0);
\draw[edge] (-4_0_0) -- (-4_0_1);
\draw[periedge] (-4_-1_1) -- (-3_-1_0);
\draw[periedge] (-4_-1_0) -- (-4_-1_1);
\draw[edge] (-5_0_1) -- (-4_-1_0);
\draw[edge] (-5_1_1) -- (-4_1_0);
\draw[periedge] (-4_1_0) -- (-4_1_1);
\draw[periedge] (-4_1_1) -- (-3_0_0);
\draw[periedge] (-5_1_1) -- (-4_0_0);
\draw[periedge] (-6_0_1) -- (-5_0_0);
\draw[periedge] (-5_0_0) -- (-5_0_1);
\draw[periedge] (-5_-1_1) -- (-4_-1_0);
\draw[periedge] (-5_-1_0) -- (-5_-1_1);
\draw[periedge] (-6_0_1) -- (-5_-1_0);
\draw[periedge] (-6_2_1) -- (-5_2_0);
\draw[periedge] (-5_2_0) -- (-5_2_1);
\draw[periedge] (-5_2_1) -- (-4_1_0);
\draw[periedge] (-5_1_0) -- (-5_1_1);
\draw[periedge] (-6_2_1) -- (-5_1_0);
\end{tikzpicture} \\
\text{(i)}\ 9,\, \mathbf{C}_\mathrm{s},\, 5,\, 62 & \text{(j)}\ 9,\, \mathbf{C}_\mathrm{s},\, 5,\, 61 & \text{(k)}\ 9,\, \mathbf{C}_\mathrm{s},\, 6,\, 76 & \text{(l)}\ 9,\, \mathbf{C}_\mathrm{s},\, 5,\, 56 \\[6pt]
\begin{tikzpicture}[scale=\scfactor,rotate=60]
\tikzstyle{every node} = [inner sep=1, draw, circle, fill=blue!50!white]
\tikzstyle{edge} = [draw, line width=1.0]
\tikzstyle{periedge} = [draw, line width=1.0]
\node (0_-5_0) at (-6.062178, -9.100000) {};
\node (0_-5_1) at (-4.849742, -9.800000) {};
\node (1_-6_0) at (-4.849742, -11.200000) {};
\node (0_-6_1) at (-6.062178, -11.900000) {};
\node (0_-6_0) at (-7.274613, -11.200000) {};
\node (-1_-5_1) at (-7.274613, -9.800000) {};
\node (0_-4_0) at (-4.849742, -7.000000) {};
\node (0_-4_1) at (-3.637307, -7.700000) {};
\node (1_-5_0) at (-3.637307, -9.100000) {};
\node (-1_-4_1) at (-6.062178, -7.700000) {};
\node (0_0_0) at (0.000000, 1.400000) {};
\node (0_0_1) at (1.212436, 0.700000) {};
\node (1_-1_0) at (1.212436, -0.700000) {};
\node (0_-1_1) at (0.000000, -1.400000) {};
\node (0_-1_0) at (-1.212436, -0.700000) {};
\node (-1_0_1) at (-1.212436, 0.700000) {};
\node (-1_-4_0) at (-7.274613, -7.000000) {};
\node (-1_-5_0) at (-8.487049, -9.100000) {};
\node (-2_-4_1) at (-8.487049, -7.700000) {};
\node (-1_-3_0) at (-6.062178, -4.900000) {};
\node (-1_-3_1) at (-4.849742, -5.600000) {};
\node (-2_-3_1) at (-7.274613, -5.600000) {};
\node (-1_-2_0) at (-4.849742, -2.800000) {};
\node (-1_-2_1) at (-3.637307, -3.500000) {};
\node (0_-3_0) at (-3.637307, -4.900000) {};
\node (-2_-2_1) at (-6.062178, -3.500000) {};
\node (-1_-1_0) at (-3.637307, -0.700000) {};
\node (-1_-1_1) at (-2.424871, -1.400000) {};
\node (0_-2_0) at (-2.424871, -2.800000) {};
\node (-2_-1_1) at (-4.849742, -1.400000) {};
\node (-1_0_0) at (-2.424871, 1.400000) {};
\node (-2_0_1) at (-3.637307, 0.700000) {};
\node (-2_-4_0) at (-9.699485, -7.000000) {};
\node (-2_-5_1) at (-9.699485, -9.800000) {};
\node (-2_-5_0) at (-10.911920, -9.100000) {};
\node (-3_-4_1) at (-10.911920, -7.700000) {};
\draw[edge] (-1_-5_1) -- (0_-5_0);
\draw[edge] (0_-5_0) -- (0_-5_1);
\draw[periedge] (0_-5_1) -- (1_-6_0);
\draw[periedge] (0_-6_1) -- (1_-6_0);
\draw[periedge] (0_-6_0) -- (0_-6_1);
\draw[periedge] (-1_-5_1) -- (0_-6_0);
\draw[edge] (-1_-4_1) -- (0_-4_0);
\draw[periedge] (0_-4_0) -- (0_-4_1);
\draw[periedge] (0_-4_1) -- (1_-5_0);
\draw[periedge] (0_-5_1) -- (1_-5_0);
\draw[edge] (-1_-4_1) -- (0_-5_0);
\draw[periedge] (-1_0_1) -- (0_0_0);
\draw[periedge] (0_0_0) -- (0_0_1);
\draw[periedge] (0_0_1) -- (1_-1_0);
\draw[periedge] (0_-1_1) -- (1_-1_0);
\draw[periedge] (0_-1_0) -- (0_-1_1);
\draw[edge] (-1_0_1) -- (0_-1_0);
\draw[periedge] (-2_-4_1) -- (-1_-4_0);
\draw[edge] (-1_-4_0) -- (-1_-4_1);
\draw[periedge] (-1_-5_0) -- (-1_-5_1);
\draw[edge] (-2_-4_1) -- (-1_-5_0);
\draw[periedge] (-2_-3_1) -- (-1_-3_0);
\draw[edge] (-1_-3_0) -- (-1_-3_1);
\draw[periedge] (-1_-3_1) -- (0_-4_0);
\draw[periedge] (-2_-3_1) -- (-1_-4_0);
\draw[periedge] (-2_-2_1) -- (-1_-2_0);
\draw[edge] (-1_-2_0) -- (-1_-2_1);
\draw[periedge] (-1_-2_1) -- (0_-3_0);
\draw[periedge] (-1_-3_1) -- (0_-3_0);
\draw[periedge] (-2_-2_1) -- (-1_-3_0);
\draw[periedge] (-2_-1_1) -- (-1_-1_0);
\draw[edge] (-1_-1_0) -- (-1_-1_1);
\draw[periedge] (-1_-1_1) -- (0_-2_0);
\draw[periedge] (-1_-2_1) -- (0_-2_0);
\draw[periedge] (-2_-1_1) -- (-1_-2_0);
\draw[periedge] (-2_0_1) -- (-1_0_0);
\draw[periedge] (-1_0_0) -- (-1_0_1);
\draw[periedge] (-1_-1_1) -- (0_-1_0);
\draw[periedge] (-2_0_1) -- (-1_-1_0);
\draw[periedge] (-3_-4_1) -- (-2_-4_0);
\draw[periedge] (-2_-4_0) -- (-2_-4_1);
\draw[periedge] (-2_-5_1) -- (-1_-5_0);
\draw[periedge] (-2_-5_0) -- (-2_-5_1);
\draw[periedge] (-3_-4_1) -- (-2_-5_0);
\end{tikzpicture} &
\begin{tikzpicture}[scale=\scfactor,rotate=60]
\tikzstyle{every node} = [inner sep=1, draw, circle, fill=blue!50!white]
\tikzstyle{edge} = [draw, line width=1.0]
\tikzstyle{periedge} = [draw, line width=1.0]
\node (0_-5_0) at (-6.062178, -9.100000) {};
\node (0_-5_1) at (-4.849742, -9.800000) {};
\node (1_-6_0) at (-4.849742, -11.200000) {};
\node (0_-6_1) at (-6.062178, -11.900000) {};
\node (0_-6_0) at (-7.274613, -11.200000) {};
\node (-1_-5_1) at (-7.274613, -9.800000) {};
\node (0_0_0) at (0.000000, 1.400000) {};
\node (0_0_1) at (1.212436, 0.700000) {};
\node (1_-1_0) at (1.212436, -0.700000) {};
\node (0_-1_1) at (0.000000, -1.400000) {};
\node (0_-1_0) at (-1.212436, -0.700000) {};
\node (-1_0_1) at (-1.212436, 0.700000) {};
\node (-1_-4_0) at (-7.274613, -7.000000) {};
\node (-1_-4_1) at (-6.062178, -7.700000) {};
\node (-1_-5_0) at (-8.487049, -9.100000) {};
\node (-2_-4_1) at (-8.487049, -7.700000) {};
\node (-1_-3_0) at (-6.062178, -4.900000) {};
\node (-1_-3_1) at (-4.849742, -5.600000) {};
\node (0_-4_0) at (-4.849742, -7.000000) {};
\node (-2_-3_1) at (-7.274613, -5.600000) {};
\node (-1_-2_0) at (-4.849742, -2.800000) {};
\node (-1_-2_1) at (-3.637307, -3.500000) {};
\node (0_-3_0) at (-3.637307, -4.900000) {};
\node (-2_-2_1) at (-6.062178, -3.500000) {};
\node (-1_-1_0) at (-3.637307, -0.700000) {};
\node (-1_-1_1) at (-2.424871, -1.400000) {};
\node (0_-2_0) at (-2.424871, -2.800000) {};
\node (-2_-1_1) at (-4.849742, -1.400000) {};
\node (-1_0_0) at (-2.424871, 1.400000) {};
\node (-2_0_1) at (-3.637307, 0.700000) {};
\node (-2_-4_0) at (-9.699485, -7.000000) {};
\node (-2_-5_1) at (-9.699485, -9.800000) {};
\node (-2_-5_0) at (-10.911920, -9.100000) {};
\node (-3_-4_1) at (-10.911920, -7.700000) {};
\node (-2_-3_0) at (-8.487049, -4.900000) {};
\node (-3_-3_1) at (-9.699485, -5.600000) {};
\draw[edge] (-1_-5_1) -- (0_-5_0);
\draw[periedge] (0_-5_0) -- (0_-5_1);
\draw[periedge] (0_-5_1) -- (1_-6_0);
\draw[periedge] (0_-6_1) -- (1_-6_0);
\draw[periedge] (0_-6_0) -- (0_-6_1);
\draw[periedge] (-1_-5_1) -- (0_-6_0);
\draw[periedge] (-1_0_1) -- (0_0_0);
\draw[periedge] (0_0_0) -- (0_0_1);
\draw[periedge] (0_0_1) -- (1_-1_0);
\draw[periedge] (0_-1_1) -- (1_-1_0);
\draw[periedge] (0_-1_0) -- (0_-1_1);
\draw[edge] (-1_0_1) -- (0_-1_0);
\draw[edge] (-2_-4_1) -- (-1_-4_0);
\draw[edge] (-1_-4_0) -- (-1_-4_1);
\draw[periedge] (-1_-4_1) -- (0_-5_0);
\draw[periedge] (-1_-5_0) -- (-1_-5_1);
\draw[edge] (-2_-4_1) -- (-1_-5_0);
\draw[periedge] (-2_-3_1) -- (-1_-3_0);
\draw[edge] (-1_-3_0) -- (-1_-3_1);
\draw[periedge] (-1_-3_1) -- (0_-4_0);
\draw[periedge] (-1_-4_1) -- (0_-4_0);
\draw[edge] (-2_-3_1) -- (-1_-4_0);
\draw[periedge] (-2_-2_1) -- (-1_-2_0);
\draw[edge] (-1_-2_0) -- (-1_-2_1);
\draw[periedge] (-1_-2_1) -- (0_-3_0);
\draw[periedge] (-1_-3_1) -- (0_-3_0);
\draw[periedge] (-2_-2_1) -- (-1_-3_0);
\draw[periedge] (-2_-1_1) -- (-1_-1_0);
\draw[edge] (-1_-1_0) -- (-1_-1_1);
\draw[periedge] (-1_-1_1) -- (0_-2_0);
\draw[periedge] (-1_-2_1) -- (0_-2_0);
\draw[periedge] (-2_-1_1) -- (-1_-2_0);
\draw[periedge] (-2_0_1) -- (-1_0_0);
\draw[periedge] (-1_0_0) -- (-1_0_1);
\draw[periedge] (-1_-1_1) -- (0_-1_0);
\draw[periedge] (-2_0_1) -- (-1_-1_0);
\draw[periedge] (-3_-4_1) -- (-2_-4_0);
\draw[edge] (-2_-4_0) -- (-2_-4_1);
\draw[periedge] (-2_-5_1) -- (-1_-5_0);
\draw[periedge] (-2_-5_0) -- (-2_-5_1);
\draw[periedge] (-3_-4_1) -- (-2_-5_0);
\draw[periedge] (-3_-3_1) -- (-2_-3_0);
\draw[periedge] (-2_-3_0) -- (-2_-3_1);
\draw[periedge] (-3_-3_1) -- (-2_-4_0);
\end{tikzpicture} &
\begin{tikzpicture}[scale=\scfactor,rotate=120]
\tikzstyle{every node} = [inner sep=1, draw, circle, fill=blue!50!white]
\tikzstyle{edge} = [draw, line width=1.0]
\tikzstyle{periedge} = [draw, line width=1.0]
\node (0_0_0) at (0.000000, 1.400000) {};
\node (0_0_1) at (1.212436, 0.700000) {};
\node (1_-1_0) at (1.212436, -0.700000) {};
\node (0_-1_1) at (0.000000, -1.400000) {};
\node (0_-1_0) at (-1.212436, -0.700000) {};
\node (-1_0_1) at (-1.212436, 0.700000) {};
\node (-1_0_0) at (-2.424871, 1.400000) {};
\node (-1_-1_1) at (-2.424871, -1.400000) {};
\node (-1_-1_0) at (-3.637307, -0.700000) {};
\node (-2_0_1) at (-3.637307, 0.700000) {};
\node (-2_-1_0) at (-6.062178, -0.700000) {};
\node (-2_-1_1) at (-4.849742, -1.400000) {};
\node (-1_-2_0) at (-4.849742, -2.800000) {};
\node (-2_-2_1) at (-6.062178, -3.500000) {};
\node (-2_-2_0) at (-7.274613, -2.800000) {};
\node (-3_-1_1) at (-7.274613, -1.400000) {};
\node (-2_0_0) at (-4.849742, 1.400000) {};
\node (-3_0_1) at (-6.062178, 0.700000) {};
\node (-3_-1_0) at (-8.487049, -0.700000) {};
\node (-3_-2_1) at (-8.487049, -3.500000) {};
\node (-3_-2_0) at (-9.699485, -2.800000) {};
\node (-4_-1_1) at (-9.699485, -1.400000) {};
\node (-3_0_0) at (-7.274613, 1.400000) {};
\node (-4_0_1) at (-8.487049, 0.700000) {};
\node (-4_0_0) at (-9.699485, 1.400000) {};
\node (-4_-1_0) at (-10.911920, -0.700000) {};
\node (-5_0_1) at (-10.911920, 0.700000) {};
\node (-5_0_0) at (-12.124356, 1.400000) {};
\node (-5_-1_1) at (-12.124356, -1.400000) {};
\node (-5_-1_0) at (-13.336791, -0.700000) {};
\node (-6_0_1) at (-13.336791, 0.700000) {};
\node (-5_1_0) at (-10.911920, 3.500000) {};
\node (-5_1_1) at (-9.699485, 2.800000) {};
\node (-6_1_1) at (-12.124356, 2.800000) {};
\draw[periedge] (-1_0_1) -- (0_0_0);
\draw[periedge] (0_0_0) -- (0_0_1);
\draw[periedge] (0_0_1) -- (1_-1_0);
\draw[periedge] (0_-1_1) -- (1_-1_0);
\draw[periedge] (0_-1_0) -- (0_-1_1);
\draw[edge] (-1_0_1) -- (0_-1_0);
\draw[periedge] (-2_0_1) -- (-1_0_0);
\draw[periedge] (-1_0_0) -- (-1_0_1);
\draw[periedge] (-1_-1_1) -- (0_-1_0);
\draw[periedge] (-1_-1_0) -- (-1_-1_1);
\draw[edge] (-2_0_1) -- (-1_-1_0);
\draw[edge] (-3_-1_1) -- (-2_-1_0);
\draw[edge] (-2_-1_0) -- (-2_-1_1);
\draw[periedge] (-2_-1_1) -- (-1_-2_0);
\draw[periedge] (-2_-2_1) -- (-1_-2_0);
\draw[periedge] (-2_-2_0) -- (-2_-2_1);
\draw[edge] (-3_-1_1) -- (-2_-2_0);
\draw[periedge] (-3_0_1) -- (-2_0_0);
\draw[periedge] (-2_0_0) -- (-2_0_1);
\draw[periedge] (-2_-1_1) -- (-1_-1_0);
\draw[edge] (-3_0_1) -- (-2_-1_0);
\draw[edge] (-4_-1_1) -- (-3_-1_0);
\draw[edge] (-3_-1_0) -- (-3_-1_1);
\draw[periedge] (-3_-2_1) -- (-2_-2_0);
\draw[periedge] (-3_-2_0) -- (-3_-2_1);
\draw[periedge] (-4_-1_1) -- (-3_-2_0);
\draw[periedge] (-4_0_1) -- (-3_0_0);
\draw[periedge] (-3_0_0) -- (-3_0_1);
\draw[edge] (-4_0_1) -- (-3_-1_0);
\draw[edge] (-5_0_1) -- (-4_0_0);
\draw[periedge] (-4_0_0) -- (-4_0_1);
\draw[periedge] (-4_-1_0) -- (-4_-1_1);
\draw[edge] (-5_0_1) -- (-4_-1_0);
\draw[periedge] (-6_0_1) -- (-5_0_0);
\draw[edge] (-5_0_0) -- (-5_0_1);
\draw[periedge] (-5_-1_1) -- (-4_-1_0);
\draw[periedge] (-5_-1_0) -- (-5_-1_1);
\draw[periedge] (-6_0_1) -- (-5_-1_0);
\draw[periedge] (-6_1_1) -- (-5_1_0);
\draw[periedge] (-5_1_0) -- (-5_1_1);
\draw[periedge] (-5_1_1) -- (-4_0_0);
\draw[periedge] (-6_1_1) -- (-5_0_0);
\end{tikzpicture} &
\begin{tikzpicture}[scale=\scfactor,rotate=120]
\tikzstyle{every node} = [inner sep=1, draw, circle, fill=blue!50!white]
\tikzstyle{edge} = [draw, line width=1.0]
\tikzstyle{periedge} = [draw, line width=1.0]
\node (0_0_0) at (0.000000, 1.400000) {};
\node (0_0_1) at (1.212436, 0.700000) {};
\node (1_-1_0) at (1.212436, -0.700000) {};
\node (0_-1_1) at (0.000000, -1.400000) {};
\node (0_-1_0) at (-1.212436, -0.700000) {};
\node (-1_0_1) at (-1.212436, 0.700000) {};
\node (-1_0_0) at (-2.424871, 1.400000) {};
\node (-1_-1_1) at (-2.424871, -1.400000) {};
\node (-1_-1_0) at (-3.637307, -0.700000) {};
\node (-2_0_1) at (-3.637307, 0.700000) {};
\node (-2_0_0) at (-4.849742, 1.400000) {};
\node (-2_-1_1) at (-4.849742, -1.400000) {};
\node (-2_-1_0) at (-6.062178, -0.700000) {};
\node (-3_0_1) at (-6.062178, 0.700000) {};
\node (-3_-1_0) at (-8.487049, -0.700000) {};
\node (-3_-1_1) at (-7.274613, -1.400000) {};
\node (-2_-2_0) at (-7.274613, -2.800000) {};
\node (-3_-2_1) at (-8.487049, -3.500000) {};
\node (-3_-2_0) at (-9.699485, -2.800000) {};
\node (-4_-1_1) at (-9.699485, -1.400000) {};
\node (-3_0_0) at (-7.274613, 1.400000) {};
\node (-4_0_1) at (-8.487049, 0.700000) {};
\node (-4_0_0) at (-9.699485, 1.400000) {};
\node (-4_-1_0) at (-10.911920, -0.700000) {};
\node (-5_0_1) at (-10.911920, 0.700000) {};
\node (-4_1_0) at (-8.487049, 3.500000) {};
\node (-4_1_1) at (-7.274613, 2.800000) {};
\node (-5_1_1) at (-9.699485, 2.800000) {};
\node (-5_0_0) at (-12.124356, 1.400000) {};
\node (-5_-1_1) at (-12.124356, -1.400000) {};
\node (-5_-1_0) at (-13.336791, -0.700000) {};
\node (-6_0_1) at (-13.336791, 0.700000) {};
\node (-5_2_0) at (-9.699485, 5.600000) {};
\node (-5_2_1) at (-8.487049, 4.900000) {};
\node (-5_1_0) at (-10.911920, 3.500000) {};
\node (-6_2_1) at (-10.911920, 4.900000) {};
\draw[periedge] (-1_0_1) -- (0_0_0);
\draw[periedge] (0_0_0) -- (0_0_1);
\draw[periedge] (0_0_1) -- (1_-1_0);
\draw[periedge] (0_-1_1) -- (1_-1_0);
\draw[periedge] (0_-1_0) -- (0_-1_1);
\draw[edge] (-1_0_1) -- (0_-1_0);
\draw[periedge] (-2_0_1) -- (-1_0_0);
\draw[periedge] (-1_0_0) -- (-1_0_1);
\draw[periedge] (-1_-1_1) -- (0_-1_0);
\draw[periedge] (-1_-1_0) -- (-1_-1_1);
\draw[edge] (-2_0_1) -- (-1_-1_0);
\draw[periedge] (-3_0_1) -- (-2_0_0);
\draw[periedge] (-2_0_0) -- (-2_0_1);
\draw[periedge] (-2_-1_1) -- (-1_-1_0);
\draw[periedge] (-2_-1_0) -- (-2_-1_1);
\draw[edge] (-3_0_1) -- (-2_-1_0);
\draw[edge] (-4_-1_1) -- (-3_-1_0);
\draw[edge] (-3_-1_0) -- (-3_-1_1);
\draw[periedge] (-3_-1_1) -- (-2_-2_0);
\draw[periedge] (-3_-2_1) -- (-2_-2_0);
\draw[periedge] (-3_-2_0) -- (-3_-2_1);
\draw[periedge] (-4_-1_1) -- (-3_-2_0);
\draw[edge] (-4_0_1) -- (-3_0_0);
\draw[periedge] (-3_0_0) -- (-3_0_1);
\draw[periedge] (-3_-1_1) -- (-2_-1_0);
\draw[edge] (-4_0_1) -- (-3_-1_0);
\draw[periedge] (-5_0_1) -- (-4_0_0);
\draw[edge] (-4_0_0) -- (-4_0_1);
\draw[periedge] (-4_-1_0) -- (-4_-1_1);
\draw[edge] (-5_0_1) -- (-4_-1_0);
\draw[edge] (-5_1_1) -- (-4_1_0);
\draw[periedge] (-4_1_0) -- (-4_1_1);
\draw[periedge] (-4_1_1) -- (-3_0_0);
\draw[periedge] (-5_1_1) -- (-4_0_0);
\draw[periedge] (-6_0_1) -- (-5_0_0);
\draw[periedge] (-5_0_0) -- (-5_0_1);
\draw[periedge] (-5_-1_1) -- (-4_-1_0);
\draw[periedge] (-5_-1_0) -- (-5_-1_1);
\draw[periedge] (-6_0_1) -- (-5_-1_0);
\draw[periedge] (-6_2_1) -- (-5_2_0);
\draw[periedge] (-5_2_0) -- (-5_2_1);
\draw[periedge] (-5_2_1) -- (-4_1_0);
\draw[periedge] (-5_1_0) -- (-5_1_1);
\draw[periedge] (-6_2_1) -- (-5_1_0);
\end{tikzpicture} \\
\text{(m)}\ 9,\, \mathbf{C}_\mathrm{s},\, 4,\, 52 & \text{(n)}\ 9,\, \mathbf{C}_\mathrm{s},\, 4,\, 52 & \text{(o)}\ 9,\, \mathbf{C}_\mathrm{s},\, 3,\, 36 & \text{(p)}\ 9,\, \mathbf{C}_\mathrm{s},\, 5,\, 58 
\end{array}$
\caption{Catafused and perifused benzenoids on $\varepsilon \leq 9$ hexagons for which the nullity jumps from $0$ to $2$ on altanisation.
The subcaption for each benzenoid lists: the number of hexagons, the maximum point group, the bay number and the number of perfect matchings.
}
\label{fig:small_excess2_benz}
\end{figure}

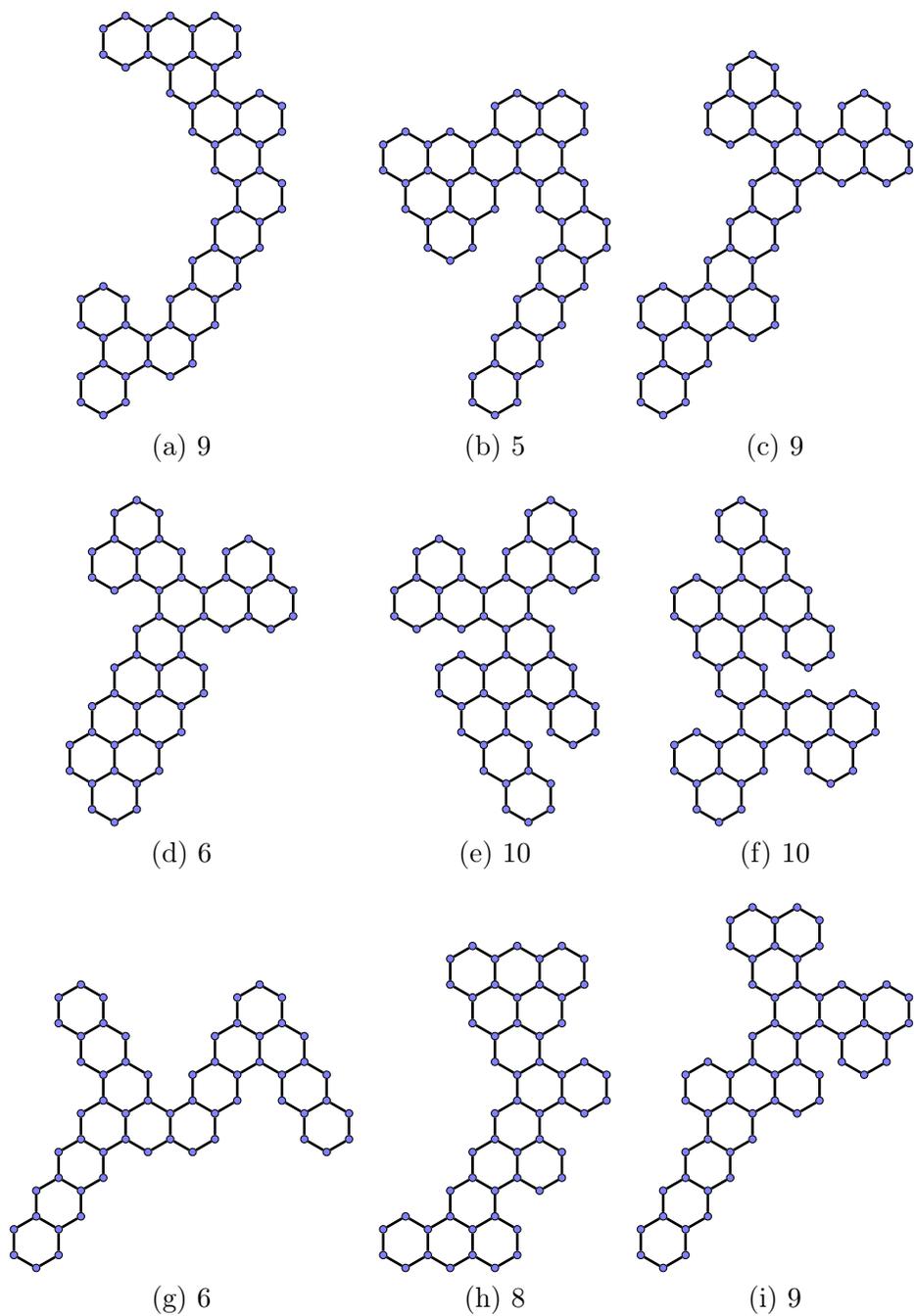
\begin{figure}[!htbp]
\centering
$\begin{array}{ccc}
\begin{tikzpicture}[scale=\scfactorx,rotate=60]
\tikzstyle{every node} = [inner sep=1, draw, circle, fill=blue!50!white]
\tikzstyle{edge} = [draw, line width=1.0]
\tikzstyle{periedge} = [draw, line width=1.0]
\node (0_0_0) at (0.000000, 1.400000) {};
\node (0_0_1) at (1.212436, 0.700000) {};
\node (1_-1_0) at (1.212436, -0.700000) {};
\node (0_-1_1) at (0.000000, -1.400000) {};
\node (0_-1_0) at (-1.212436, -0.700000) {};
\node (-1_0_1) at (-1.212436, 0.700000) {};
\node (0_1_0) at (1.212436, 3.500000) {};
\node (0_1_1) at (2.424871, 2.800000) {};
\node (1_0_0) at (2.424871, 1.400000) {};
\node (-1_1_1) at (0.000000, 2.800000) {};
\node (1_-1_1) at (2.424871, -1.400000) {};
\node (2_-2_0) at (2.424871, -2.800000) {};
\node (1_-2_1) at (1.212436, -3.500000) {};
\node (1_-2_0) at (0.000000, -2.800000) {};
\node (2_-1_0) at (3.637307, -0.700000) {};
\node (2_-1_1) at (4.849742, -1.400000) {};
\node (3_-2_0) at (4.849742, -2.800000) {};
\node (2_-2_1) at (3.637307, -3.500000) {};
\node (3_-1_0) at (6.062178, -0.700000) {};
\node (3_-1_1) at (7.274613, -1.400000) {};
\node (4_-2_0) at (7.274613, -2.800000) {};
\node (3_-2_1) at (6.062178, -3.500000) {};
\node (4_-1_0) at (8.487049, -0.700000) {};
\node (4_-1_1) at (9.699485, -1.400000) {};
\node (5_-2_0) at (9.699485, -2.800000) {};
\node (4_-2_1) at (8.487049, -3.500000) {};
\node (4_4_0) at (14.549227, 9.800000) {};
\node (4_4_1) at (15.761662, 9.100000) {};
\node (5_3_0) at (15.761662, 7.700000) {};
\node (4_3_1) at (14.549227, 7.000000) {};
\node (4_3_0) at (13.336791, 7.700000) {};
\node (3_4_1) at (13.336791, 9.100000) {};
\node (5_-1_0) at (10.911920, -0.700000) {};
\node (5_-1_1) at (12.124356, -1.400000) {};
\node (6_-2_0) at (12.124356, -2.800000) {};
\node (5_-2_1) at (10.911920, -3.500000) {};
\node (5_0_0) at (12.124356, 1.400000) {};
\node (5_0_1) at (13.336791, 0.700000) {};
\node (6_-1_0) at (13.336791, -0.700000) {};
\node (4_0_1) at (10.911920, 0.700000) {};
\node (5_1_0) at (13.336791, 3.500000) {};
\node (5_1_1) at (14.549227, 2.800000) {};
\node (6_0_0) at (14.549227, 1.400000) {};
\node (4_1_1) at (12.124356, 2.800000) {};
\node (5_2_0) at (14.549227, 5.600000) {};
\node (5_2_1) at (15.761662, 4.900000) {};
\node (6_1_0) at (15.761662, 3.500000) {};
\node (4_2_1) at (13.336791, 4.900000) {};
\node (5_3_1) at (16.974098, 7.000000) {};
\node (6_2_0) at (16.974098, 5.600000) {};
\node (6_0_1) at (15.761662, 0.700000) {};
\node (7_-1_0) at (15.761662, -0.700000) {};
\node (6_-1_1) at (14.549227, -1.400000) {};
\node (6_2_1) at (18.186533, 4.900000) {};
\node (7_1_0) at (18.186533, 3.500000) {};
\node (6_1_1) at (16.974098, 2.800000) {};
\node (-1_0_0) at (-2.424871, 1.400000) {};
\node (-1_-1_1) at (-2.424871, -1.400000) {};
\node (-1_-1_0) at (-3.637307, -0.700000) {};
\node (-2_0_1) at (-3.637307, 0.700000) {};
\draw[periedge] (-1_0_1) -- (0_0_0);
\draw[edge] (0_0_0) -- (0_0_1);
\draw[periedge] (0_0_1) -- (1_-1_0);
\draw[edge] (0_-1_1) -- (1_-1_0);
\draw[periedge] (0_-1_0) -- (0_-1_1);
\draw[edge] (-1_0_1) -- (0_-1_0);
\draw[periedge] (-1_1_1) -- (0_1_0);
\draw[periedge] (0_1_0) -- (0_1_1);
\draw[periedge] (0_1_1) -- (1_0_0);
\draw[periedge] (0_0_1) -- (1_0_0);
\draw[periedge] (-1_1_1) -- (0_0_0);
\draw[periedge] (1_-1_0) -- (1_-1_1);
\draw[edge] (1_-1_1) -- (2_-2_0);
\draw[periedge] (1_-2_1) -- (2_-2_0);
\draw[periedge] (1_-2_0) -- (1_-2_1);
\draw[periedge] (0_-1_1) -- (1_-2_0);
\draw[periedge] (1_-1_1) -- (2_-1_0);
\draw[periedge] (2_-1_0) -- (2_-1_1);
\draw[edge] (2_-1_1) -- (3_-2_0);
\draw[periedge] (2_-2_1) -- (3_-2_0);
\draw[periedge] (2_-2_0) -- (2_-2_1);
\draw[periedge] (2_-1_1) -- (3_-1_0);
\draw[periedge] (3_-1_0) -- (3_-1_1);
\draw[edge] (3_-1_1) -- (4_-2_0);
\draw[periedge] (3_-2_1) -- (4_-2_0);
\draw[periedge] (3_-2_0) -- (3_-2_1);
\draw[periedge] (3_-1_1) -- (4_-1_0);
\draw[periedge] (4_-1_0) -- (4_-1_1);
\draw[edge] (4_-1_1) -- (5_-2_0);
\draw[periedge] (4_-2_1) -- (5_-2_0);
\draw[periedge] (4_-2_0) -- (4_-2_1);
\draw[periedge] (3_4_1) -- (4_4_0);
\draw[periedge] (4_4_0) -- (4_4_1);
\draw[periedge] (4_4_1) -- (5_3_0);
\draw[edge] (4_3_1) -- (5_3_0);
\draw[periedge] (4_3_0) -- (4_3_1);
\draw[periedge] (3_4_1) -- (4_3_0);
\draw[periedge] (4_-1_1) -- (5_-1_0);
\draw[edge] (5_-1_0) -- (5_-1_1);
\draw[periedge] (5_-1_1) -- (6_-2_0);
\draw[periedge] (5_-2_1) -- (6_-2_0);
\draw[periedge] (5_-2_0) -- (5_-2_1);
\draw[periedge] (4_0_1) -- (5_0_0);
\draw[edge] (5_0_0) -- (5_0_1);
\draw[edge] (5_0_1) -- (6_-1_0);
\draw[periedge] (5_-1_1) -- (6_-1_0);
\draw[periedge] (4_0_1) -- (5_-1_0);
\draw[periedge] (4_1_1) -- (5_1_0);
\draw[edge] (5_1_0) -- (5_1_1);
\draw[periedge] (5_1_1) -- (6_0_0);
\draw[edge] (5_0_1) -- (6_0_0);
\draw[periedge] (4_1_1) -- (5_0_0);
\draw[periedge] (4_2_1) -- (5_2_0);
\draw[edge] (5_2_0) -- (5_2_1);
\draw[edge] (5_2_1) -- (6_1_0);
\draw[periedge] (5_1_1) -- (6_1_0);
\draw[periedge] (4_2_1) -- (5_1_0);
\draw[periedge] (5_3_0) -- (5_3_1);
\draw[periedge] (5_3_1) -- (6_2_0);
\draw[edge] (5_2_1) -- (6_2_0);
\draw[periedge] (4_3_1) -- (5_2_0);
\draw[periedge] (6_0_0) -- (6_0_1);
\draw[periedge] (6_0_1) -- (7_-1_0);
\draw[periedge] (6_-1_1) -- (7_-1_0);
\draw[periedge] (6_-1_0) -- (6_-1_1);
\draw[periedge] (6_2_0) -- (6_2_1);
\draw[periedge] (6_2_1) -- (7_1_0);
\draw[periedge] (6_1_1) -- (7_1_0);
\draw[periedge] (6_1_0) -- (6_1_1);
\draw[periedge] (-2_0_1) -- (-1_0_0);
\draw[periedge] (-1_0_0) -- (-1_0_1);
\draw[periedge] (-1_-1_1) -- (0_-1_0);
\draw[periedge] (-1_-1_0) -- (-1_-1_1);
\draw[periedge] (-2_0_1) -- (-1_-1_0);
\end{tikzpicture} & 
\begin{tikzpicture}[scale=\scfactorx,rotate=60]
\tikzstyle{every node} = [inner sep=1, draw, circle, fill=blue!50!white]
\tikzstyle{edge} = [draw, line width=1.0]
\tikzstyle{periedge} = [draw, line width=1.0]
\node (0_0_0) at (0.000000, 1.400000) {};
\node (0_0_1) at (1.212436, 0.700000) {};
\node (1_-1_0) at (1.212436, -0.700000) {};
\node (0_-1_1) at (0.000000, -1.400000) {};
\node (0_-1_0) at (-1.212436, -0.700000) {};
\node (-1_0_1) at (-1.212436, 0.700000) {};
\node (0_3_0) at (3.637307, 7.700000) {};
\node (0_3_1) at (4.849742, 7.000000) {};
\node (1_2_0) at (4.849742, 5.600000) {};
\node (0_2_1) at (3.637307, 4.900000) {};
\node (0_2_0) at (2.424871, 5.600000) {};
\node (-1_3_1) at (2.424871, 7.000000) {};
\node (0_4_0) at (4.849742, 9.800000) {};
\node (0_4_1) at (6.062178, 9.100000) {};
\node (1_3_0) at (6.062178, 7.700000) {};
\node (-1_4_1) at (3.637307, 9.100000) {};
\node (1_0_0) at (2.424871, 1.400000) {};
\node (1_0_1) at (3.637307, 0.700000) {};
\node (2_-1_0) at (3.637307, -0.700000) {};
\node (1_-1_1) at (2.424871, -1.400000) {};
\node (1_3_1) at (7.274613, 7.000000) {};
\node (2_2_0) at (7.274613, 5.600000) {};
\node (1_2_1) at (6.062178, 4.900000) {};
\node (2_0_0) at (4.849742, 1.400000) {};
\node (2_0_1) at (6.062178, 0.700000) {};
\node (3_-1_0) at (6.062178, -0.700000) {};
\node (2_-1_1) at (4.849742, -1.400000) {};
\node (2_1_0) at (6.062178, 3.500000) {};
\node (2_1_1) at (7.274613, 2.800000) {};
\node (3_0_0) at (7.274613, 1.400000) {};
\node (1_1_1) at (4.849742, 2.800000) {};
\node (2_2_1) at (8.487049, 4.900000) {};
\node (3_1_0) at (8.487049, 3.500000) {};
\node (2_3_0) at (8.487049, 7.700000) {};
\node (2_3_1) at (9.699485, 7.000000) {};
\node (3_2_0) at (9.699485, 5.600000) {};
\node (3_2_1) at (10.911920, 4.900000) {};
\node (4_1_0) at (10.911920, 3.500000) {};
\node (3_1_1) at (9.699485, 2.800000) {};
\node (-1_0_0) at (-2.424871, 1.400000) {};
\node (-1_-1_1) at (-2.424871, -1.400000) {};
\node (-1_-1_0) at (-3.637307, -0.700000) {};
\node (-2_0_1) at (-3.637307, 0.700000) {};
\node (-1_3_0) at (1.212436, 7.700000) {};
\node (-1_2_1) at (1.212436, 4.900000) {};
\node (-1_2_0) at (0.000000, 5.600000) {};
\node (-2_3_1) at (0.000000, 7.000000) {};
\node (-1_4_0) at (2.424871, 9.800000) {};
\node (-2_4_1) at (1.212436, 9.100000) {};
\node (-1_5_0) at (3.637307, 11.900000) {};
\node (-1_5_1) at (4.849742, 11.200000) {};
\node (-2_5_1) at (2.424871, 11.200000) {};
\node (-2_0_0) at (-4.849742, 1.400000) {};
\node (-2_-1_1) at (-4.849742, -1.400000) {};
\node (-2_-1_0) at (-6.062178, -0.700000) {};
\node (-3_0_1) at (-6.062178, 0.700000) {};
\draw[periedge] (-1_0_1) -- (0_0_0);
\draw[periedge] (0_0_0) -- (0_0_1);
\draw[edge] (0_0_1) -- (1_-1_0);
\draw[periedge] (0_-1_1) -- (1_-1_0);
\draw[periedge] (0_-1_0) -- (0_-1_1);
\draw[edge] (-1_0_1) -- (0_-1_0);
\draw[edge] (-1_3_1) -- (0_3_0);
\draw[edge] (0_3_0) -- (0_3_1);
\draw[edge] (0_3_1) -- (1_2_0);
\draw[periedge] (0_2_1) -- (1_2_0);
\draw[periedge] (0_2_0) -- (0_2_1);
\draw[edge] (-1_3_1) -- (0_2_0);
\draw[edge] (-1_4_1) -- (0_4_0);
\draw[periedge] (0_4_0) -- (0_4_1);
\draw[periedge] (0_4_1) -- (1_3_0);
\draw[edge] (0_3_1) -- (1_3_0);
\draw[edge] (-1_4_1) -- (0_3_0);
\draw[periedge] (0_0_1) -- (1_0_0);
\draw[periedge] (1_0_0) -- (1_0_1);
\draw[edge] (1_0_1) -- (2_-1_0);
\draw[periedge] (1_-1_1) -- (2_-1_0);
\draw[periedge] (1_-1_0) -- (1_-1_1);
\draw[periedge] (1_3_0) -- (1_3_1);
\draw[edge] (1_3_1) -- (2_2_0);
\draw[edge] (1_2_1) -- (2_2_0);
\draw[periedge] (1_2_0) -- (1_2_1);
\draw[periedge] (1_0_1) -- (2_0_0);
\draw[edge] (2_0_0) -- (2_0_1);
\draw[periedge] (2_0_1) -- (3_-1_0);
\draw[periedge] (2_-1_1) -- (3_-1_0);
\draw[periedge] (2_-1_0) -- (2_-1_1);
\draw[periedge] (1_1_1) -- (2_1_0);
\draw[edge] (2_1_0) -- (2_1_1);
\draw[periedge] (2_1_1) -- (3_0_0);
\draw[periedge] (2_0_1) -- (3_0_0);
\draw[periedge] (1_1_1) -- (2_0_0);
\draw[edge] (2_2_0) -- (2_2_1);
\draw[edge] (2_2_1) -- (3_1_0);
\draw[periedge] (2_1_1) -- (3_1_0);
\draw[periedge] (1_2_1) -- (2_1_0);
\draw[periedge] (1_3_1) -- (2_3_0);
\draw[periedge] (2_3_0) -- (2_3_1);
\draw[periedge] (2_3_1) -- (3_2_0);
\draw[edge] (2_2_1) -- (3_2_0);
\draw[periedge] (3_2_0) -- (3_2_1);
\draw[periedge] (3_2_1) -- (4_1_0);
\draw[periedge] (3_1_1) -- (4_1_0);
\draw[periedge] (3_1_0) -- (3_1_1);
\draw[periedge] (-2_0_1) -- (-1_0_0);
\draw[periedge] (-1_0_0) -- (-1_0_1);
\draw[periedge] (-1_-1_1) -- (0_-1_0);
\draw[periedge] (-1_-1_0) -- (-1_-1_1);
\draw[edge] (-2_0_1) -- (-1_-1_0);
\draw[periedge] (-2_3_1) -- (-1_3_0);
\draw[edge] (-1_3_0) -- (-1_3_1);
\draw[periedge] (-1_2_1) -- (0_2_0);
\draw[periedge] (-1_2_0) -- (-1_2_1);
\draw[periedge] (-2_3_1) -- (-1_2_0);
\draw[periedge] (-2_4_1) -- (-1_4_0);
\draw[edge] (-1_4_0) -- (-1_4_1);
\draw[periedge] (-2_4_1) -- (-1_3_0);
\draw[periedge] (-2_5_1) -- (-1_5_0);
\draw[periedge] (-1_5_0) -- (-1_5_1);
\draw[periedge] (-1_5_1) -- (0_4_0);
\draw[periedge] (-2_5_1) -- (-1_4_0);
\draw[periedge] (-3_0_1) -- (-2_0_0);
\draw[periedge] (-2_0_0) -- (-2_0_1);
\draw[periedge] (-2_-1_1) -- (-1_-1_0);
\draw[periedge] (-2_-1_0) -- (-2_-1_1);
\draw[periedge] (-3_0_1) -- (-2_-1_0);
\end{tikzpicture} &
\begin{tikzpicture}[scale=\scfactorx,rotate=60]
\tikzstyle{every node} = [inner sep=1, draw, circle, fill=blue!50!white]
\tikzstyle{edge} = [draw, line width=1.0]
\tikzstyle{periedge} = [draw, line width=1.0]
\node (0_0_0) at (0.000000, 1.400000) {};
\node (0_0_1) at (1.212436, 0.700000) {};
\node (1_-1_0) at (1.212436, -0.700000) {};
\node (0_-1_1) at (0.000000, -1.400000) {};
\node (0_-1_0) at (-1.212436, -0.700000) {};
\node (-1_0_1) at (-1.212436, 0.700000) {};
\node (1_-1_1) at (2.424871, -1.400000) {};
\node (2_-2_0) at (2.424871, -2.800000) {};
\node (1_-2_1) at (1.212436, -3.500000) {};
\node (1_-2_0) at (0.000000, -2.800000) {};
\node (1_0_0) at (2.424871, 1.400000) {};
\node (1_0_1) at (3.637307, 0.700000) {};
\node (2_-1_0) at (3.637307, -0.700000) {};
\node (2_0_0) at (4.849742, 1.400000) {};
\node (2_0_1) at (6.062178, 0.700000) {};
\node (3_-1_0) at (6.062178, -0.700000) {};
\node (2_-1_1) at (4.849742, -1.400000) {};
\node (3_0_0) at (7.274613, 1.400000) {};
\node (3_0_1) at (8.487049, 0.700000) {};
\node (4_-1_0) at (8.487049, -0.700000) {};
\node (3_-1_1) at (7.274613, -1.400000) {};
\node (3_2_0) at (9.699485, 5.600000) {};
\node (3_2_1) at (10.911920, 4.900000) {};
\node (4_1_0) at (10.911920, 3.500000) {};
\node (3_1_1) at (9.699485, 2.800000) {};
\node (3_1_0) at (8.487049, 3.500000) {};
\node (2_2_1) at (8.487049, 4.900000) {};
\node (4_0_0) at (9.699485, 1.400000) {};
\node (4_0_1) at (10.911920, 0.700000) {};
\node (5_-1_0) at (10.911920, -0.700000) {};
\node (4_-1_1) at (9.699485, -1.400000) {};
\node (4_1_1) at (12.124356, 2.800000) {};
\node (5_0_0) at (12.124356, 1.400000) {};
\node (4_2_0) at (12.124356, 5.600000) {};
\node (4_2_1) at (13.336791, 4.900000) {};
\node (5_1_0) at (13.336791, 3.500000) {};
\node (5_-1_1) at (12.124356, -1.400000) {};
\node (6_-2_0) at (12.124356, -2.800000) {};
\node (5_-2_1) at (10.911920, -3.500000) {};
\node (5_-2_0) at (9.699485, -2.800000) {};
\node (6_-2_1) at (13.336791, -3.500000) {};
\node (7_-3_0) at (13.336791, -4.900000) {};
\node (6_-3_1) at (12.124356, -5.600000) {};
\node (6_-3_0) at (10.911920, -4.900000) {};
\node (6_-1_0) at (13.336791, -0.700000) {};
\node (6_-1_1) at (14.549227, -1.400000) {};
\node (7_-2_0) at (14.549227, -2.800000) {};
\node (-1_0_0) at (-2.424871, 1.400000) {};
\node (-1_-1_1) at (-2.424871, -1.400000) {};
\node (-1_-1_0) at (-3.637307, -0.700000) {};
\node (-2_0_1) at (-3.637307, 0.700000) {};
\node (-1_1_0) at (-1.212436, 3.500000) {};
\node (-1_1_1) at (0.000000, 2.800000) {};
\node (-2_1_1) at (-2.424871, 2.800000) {};
\node (-2_0_0) at (-4.849742, 1.400000) {};
\node (-2_-1_1) at (-4.849742, -1.400000) {};
\node (-2_-1_0) at (-6.062178, -0.700000) {};
\node (-3_0_1) at (-6.062178, 0.700000) {};
\draw[edge] (-1_0_1) -- (0_0_0);
\draw[periedge] (0_0_0) -- (0_0_1);
\draw[edge] (0_0_1) -- (1_-1_0);
\draw[edge] (0_-1_1) -- (1_-1_0);
\draw[periedge] (0_-1_0) -- (0_-1_1);
\draw[edge] (-1_0_1) -- (0_-1_0);
\draw[edge] (1_-1_0) -- (1_-1_1);
\draw[periedge] (1_-1_1) -- (2_-2_0);
\draw[periedge] (1_-2_1) -- (2_-2_0);
\draw[periedge] (1_-2_0) -- (1_-2_1);
\draw[periedge] (0_-1_1) -- (1_-2_0);
\draw[periedge] (0_0_1) -- (1_0_0);
\draw[periedge] (1_0_0) -- (1_0_1);
\draw[edge] (1_0_1) -- (2_-1_0);
\draw[periedge] (1_-1_1) -- (2_-1_0);
\draw[periedge] (1_0_1) -- (2_0_0);
\draw[periedge] (2_0_0) -- (2_0_1);
\draw[edge] (2_0_1) -- (3_-1_0);
\draw[periedge] (2_-1_1) -- (3_-1_0);
\draw[periedge] (2_-1_0) -- (2_-1_1);
\draw[periedge] (2_0_1) -- (3_0_0);
\draw[periedge] (3_0_0) -- (3_0_1);
\draw[edge] (3_0_1) -- (4_-1_0);
\draw[periedge] (3_-1_1) -- (4_-1_0);
\draw[periedge] (3_-1_0) -- (3_-1_1);
\draw[periedge] (2_2_1) -- (3_2_0);
\draw[periedge] (3_2_0) -- (3_2_1);
\draw[edge] (3_2_1) -- (4_1_0);
\draw[edge] (3_1_1) -- (4_1_0);
\draw[periedge] (3_1_0) -- (3_1_1);
\draw[periedge] (2_2_1) -- (3_1_0);
\draw[periedge] (3_0_1) -- (4_0_0);
\draw[edge] (4_0_0) -- (4_0_1);
\draw[periedge] (4_0_1) -- (5_-1_0);
\draw[edge] (4_-1_1) -- (5_-1_0);
\draw[periedge] (4_-1_0) -- (4_-1_1);
\draw[edge] (4_1_0) -- (4_1_1);
\draw[periedge] (4_1_1) -- (5_0_0);
\draw[periedge] (4_0_1) -- (5_0_0);
\draw[periedge] (3_1_1) -- (4_0_0);
\draw[periedge] (3_2_1) -- (4_2_0);
\draw[periedge] (4_2_0) -- (4_2_1);
\draw[periedge] (4_2_1) -- (5_1_0);
\draw[periedge] (4_1_1) -- (5_1_0);
\draw[periedge] (5_-1_0) -- (5_-1_1);
\draw[edge] (5_-1_1) -- (6_-2_0);
\draw[edge] (5_-2_1) -- (6_-2_0);
\draw[periedge] (5_-2_0) -- (5_-2_1);
\draw[periedge] (4_-1_1) -- (5_-2_0);
\draw[edge] (6_-2_0) -- (6_-2_1);
\draw[periedge] (6_-2_1) -- (7_-3_0);
\draw[periedge] (6_-3_1) -- (7_-3_0);
\draw[periedge] (6_-3_0) -- (6_-3_1);
\draw[periedge] (5_-2_1) -- (6_-3_0);
\draw[periedge] (5_-1_1) -- (6_-1_0);
\draw[periedge] (6_-1_0) -- (6_-1_1);
\draw[periedge] (6_-1_1) -- (7_-2_0);
\draw[periedge] (6_-2_1) -- (7_-2_0);
\draw[periedge] (-2_0_1) -- (-1_0_0);
\draw[edge] (-1_0_0) -- (-1_0_1);
\draw[periedge] (-1_-1_1) -- (0_-1_0);
\draw[periedge] (-1_-1_0) -- (-1_-1_1);
\draw[edge] (-2_0_1) -- (-1_-1_0);
\draw[periedge] (-2_1_1) -- (-1_1_0);
\draw[periedge] (-1_1_0) -- (-1_1_1);
\draw[periedge] (-1_1_1) -- (0_0_0);
\draw[periedge] (-2_1_1) -- (-1_0_0);
\draw[periedge] (-3_0_1) -- (-2_0_0);
\draw[periedge] (-2_0_0) -- (-2_0_1);
\draw[periedge] (-2_-1_1) -- (-1_-1_0);
\draw[periedge] (-2_-1_0) -- (-2_-1_1);
\draw[periedge] (-3_0_1) -- (-2_-1_0);
\end{tikzpicture} \\    
\text{(a)}\ 9& \text{(b)}\ 5 & \text{(c)}\ 9 \\[12pt]
\begin{tikzpicture}[scale=\scfactorx,rotate=60]
\tikzstyle{every node} = [inner sep=1, draw, circle, fill=blue!50!white]
\tikzstyle{edge} = [draw, line width=1.0]
\tikzstyle{periedge} = [draw, line width=1.0]
\node (0_0_0) at (0.000000, 1.400000) {};
\node (0_0_1) at (1.212436, 0.700000) {};
\node (1_-1_0) at (1.212436, -0.700000) {};
\node (0_-1_1) at (0.000000, -1.400000) {};
\node (0_-1_0) at (-1.212436, -0.700000) {};
\node (-1_0_1) at (-1.212436, 0.700000) {};
\node (0_1_0) at (1.212436, 3.500000) {};
\node (0_1_1) at (2.424871, 2.800000) {};
\node (1_0_0) at (2.424871, 1.400000) {};
\node (-1_1_1) at (0.000000, 2.800000) {};
\node (1_0_1) at (3.637307, 0.700000) {};
\node (2_-1_0) at (3.637307, -0.700000) {};
\node (1_-1_1) at (2.424871, -1.400000) {};
\node (1_1_0) at (3.637307, 3.500000) {};
\node (1_1_1) at (4.849742, 2.800000) {};
\node (2_0_0) at (4.849742, 1.400000) {};
\node (1_3_0) at (6.062178, 7.700000) {};
\node (1_3_1) at (7.274613, 7.000000) {};
\node (2_2_0) at (7.274613, 5.600000) {};
\node (1_2_1) at (6.062178, 4.900000) {};
\node (1_2_0) at (4.849742, 5.600000) {};
\node (0_3_1) at (4.849742, 7.000000) {};
\node (2_1_0) at (6.062178, 3.500000) {};
\node (2_1_1) at (7.274613, 2.800000) {};
\node (3_0_0) at (7.274613, 1.400000) {};
\node (2_0_1) at (6.062178, 0.700000) {};
\node (2_2_1) at (8.487049, 4.900000) {};
\node (3_1_0) at (8.487049, 3.500000) {};
\node (2_3_0) at (8.487049, 7.700000) {};
\node (2_3_1) at (9.699485, 7.000000) {};
\node (3_2_0) at (9.699485, 5.600000) {};
\node (3_0_1) at (8.487049, 0.700000) {};
\node (4_-1_0) at (8.487049, -0.700000) {};
\node (3_-1_1) at (7.274613, -1.400000) {};
\node (3_-1_0) at (6.062178, -0.700000) {};
\node (4_-1_1) at (9.699485, -1.400000) {};
\node (5_-2_0) at (9.699485, -2.800000) {};
\node (4_-2_1) at (8.487049, -3.500000) {};
\node (4_-2_0) at (7.274613, -2.800000) {};
\node (4_0_0) at (9.699485, 1.400000) {};
\node (4_0_1) at (10.911920, 0.700000) {};
\node (5_-1_0) at (10.911920, -0.700000) {};
\node (-1_0_0) at (-2.424871, 1.400000) {};
\node (-1_-1_1) at (-2.424871, -1.400000) {};
\node (-1_-1_0) at (-3.637307, -0.700000) {};
\node (-2_0_1) at (-3.637307, 0.700000) {};
\node (-1_1_0) at (-1.212436, 3.500000) {};
\node (-2_1_1) at (-2.424871, 2.800000) {};
\node (-2_0_0) at (-4.849742, 1.400000) {};
\node (-2_-1_1) at (-4.849742, -1.400000) {};
\node (-2_-1_0) at (-6.062178, -0.700000) {};
\node (-3_0_1) at (-6.062178, 0.700000) {};
\node (-2_1_0) at (-3.637307, 3.500000) {};
\node (-3_1_1) at (-4.849742, 2.800000) {};
\draw[edge] (-1_0_1) -- (0_0_0);
\draw[edge] (0_0_0) -- (0_0_1);
\draw[edge] (0_0_1) -- (1_-1_0);
\draw[periedge] (0_-1_1) -- (1_-1_0);
\draw[periedge] (0_-1_0) -- (0_-1_1);
\draw[edge] (-1_0_1) -- (0_-1_0);
\draw[periedge] (-1_1_1) -- (0_1_0);
\draw[periedge] (0_1_0) -- (0_1_1);
\draw[edge] (0_1_1) -- (1_0_0);
\draw[edge] (0_0_1) -- (1_0_0);
\draw[edge] (-1_1_1) -- (0_0_0);
\draw[edge] (1_0_0) -- (1_0_1);
\draw[periedge] (1_0_1) -- (2_-1_0);
\draw[periedge] (1_-1_1) -- (2_-1_0);
\draw[periedge] (1_-1_0) -- (1_-1_1);
\draw[periedge] (0_1_1) -- (1_1_0);
\draw[periedge] (1_1_0) -- (1_1_1);
\draw[edge] (1_1_1) -- (2_0_0);
\draw[periedge] (1_0_1) -- (2_0_0);
\draw[periedge] (0_3_1) -- (1_3_0);
\draw[periedge] (1_3_0) -- (1_3_1);
\draw[edge] (1_3_1) -- (2_2_0);
\draw[edge] (1_2_1) -- (2_2_0);
\draw[periedge] (1_2_0) -- (1_2_1);
\draw[periedge] (0_3_1) -- (1_2_0);
\draw[periedge] (1_1_1) -- (2_1_0);
\draw[edge] (2_1_0) -- (2_1_1);
\draw[periedge] (2_1_1) -- (3_0_0);
\draw[edge] (2_0_1) -- (3_0_0);
\draw[periedge] (2_0_0) -- (2_0_1);
\draw[edge] (2_2_0) -- (2_2_1);
\draw[periedge] (2_2_1) -- (3_1_0);
\draw[periedge] (2_1_1) -- (3_1_0);
\draw[periedge] (1_2_1) -- (2_1_0);
\draw[periedge] (1_3_1) -- (2_3_0);
\draw[periedge] (2_3_0) -- (2_3_1);
\draw[periedge] (2_3_1) -- (3_2_0);
\draw[periedge] (2_2_1) -- (3_2_0);
\draw[periedge] (3_0_0) -- (3_0_1);
\draw[edge] (3_0_1) -- (4_-1_0);
\draw[edge] (3_-1_1) -- (4_-1_0);
\draw[periedge] (3_-1_0) -- (3_-1_1);
\draw[periedge] (2_0_1) -- (3_-1_0);
\draw[edge] (4_-1_0) -- (4_-1_1);
\draw[periedge] (4_-1_1) -- (5_-2_0);
\draw[periedge] (4_-2_1) -- (5_-2_0);
\draw[periedge] (4_-2_0) -- (4_-2_1);
\draw[periedge] (3_-1_1) -- (4_-2_0);
\draw[periedge] (3_0_1) -- (4_0_0);
\draw[periedge] (4_0_0) -- (4_0_1);
\draw[periedge] (4_0_1) -- (5_-1_0);
\draw[periedge] (4_-1_1) -- (5_-1_0);
\draw[edge] (-2_0_1) -- (-1_0_0);
\draw[edge] (-1_0_0) -- (-1_0_1);
\draw[periedge] (-1_-1_1) -- (0_-1_0);
\draw[periedge] (-1_-1_0) -- (-1_-1_1);
\draw[edge] (-2_0_1) -- (-1_-1_0);
\draw[periedge] (-2_1_1) -- (-1_1_0);
\draw[periedge] (-1_1_0) -- (-1_1_1);
\draw[edge] (-2_1_1) -- (-1_0_0);
\draw[periedge] (-3_0_1) -- (-2_0_0);
\draw[edge] (-2_0_0) -- (-2_0_1);
\draw[periedge] (-2_-1_1) -- (-1_-1_0);
\draw[periedge] (-2_-1_0) -- (-2_-1_1);
\draw[periedge] (-3_0_1) -- (-2_-1_0);
\draw[periedge] (-3_1_1) -- (-2_1_0);
\draw[periedge] (-2_1_0) -- (-2_1_1);
\draw[periedge] (-3_1_1) -- (-2_0_0);
\end{tikzpicture} &
\begin{tikzpicture}[scale=\scfactorx,rotate=120]
\tikzstyle{every node} = [inner sep=1, draw, circle, fill=blue!50!white]
\tikzstyle{edge} = [draw, line width=1.0]
\tikzstyle{periedge} = [draw, line width=1.0]
\node (0_0_0) at (0.000000, 1.400000) {};
\node (0_0_1) at (1.212436, 0.700000) {};
\node (1_-1_0) at (1.212436, -0.700000) {};
\node (0_-1_1) at (0.000000, -1.400000) {};
\node (0_-1_0) at (-1.212436, -0.700000) {};
\node (-1_0_1) at (-1.212436, 0.700000) {};
\node (1_0_0) at (2.424871, 1.400000) {};
\node (1_0_1) at (3.637307, 0.700000) {};
\node (2_-1_0) at (3.637307, -0.700000) {};
\node (1_-1_1) at (2.424871, -1.400000) {};
\node (1_1_0) at (3.637307, 3.500000) {};
\node (1_1_1) at (4.849742, 2.800000) {};
\node (2_0_0) at (4.849742, 1.400000) {};
\node (0_1_1) at (2.424871, 2.800000) {};
\node (1_2_0) at (4.849742, 5.600000) {};
\node (1_2_1) at (6.062178, 4.900000) {};
\node (2_1_0) at (6.062178, 3.500000) {};
\node (0_2_1) at (3.637307, 4.900000) {};
\node (2_-2_0) at (2.424871, -2.800000) {};
\node (2_-2_1) at (3.637307, -3.500000) {};
\node (3_-3_0) at (3.637307, -4.900000) {};
\node (2_-3_1) at (2.424871, -5.600000) {};
\node (2_-3_0) at (1.212436, -4.900000) {};
\node (1_-2_1) at (1.212436, -3.500000) {};
\node (2_-1_1) at (4.849742, -1.400000) {};
\node (3_-2_0) at (4.849742, -2.800000) {};
\node (2_1_1) at (7.274613, 2.800000) {};
\node (3_0_0) at (7.274613, 1.400000) {};
\node (2_0_1) at (6.062178, 0.700000) {};
\node (3_-2_1) at (6.062178, -3.500000) {};
\node (4_-3_0) at (6.062178, -4.900000) {};
\node (3_-3_1) at (4.849742, -5.600000) {};
\node (-1_0_0) at (-2.424871, 1.400000) {};
\node (-1_-1_1) at (-2.424871, -1.400000) {};
\node (-1_-1_0) at (-3.637307, -0.700000) {};
\node (-2_0_1) at (-3.637307, 0.700000) {};
\node (-1_1_0) at (-1.212436, 3.500000) {};
\node (-1_1_1) at (0.000000, 2.800000) {};
\node (-2_1_1) at (-2.424871, 2.800000) {};
\node (-1_2_0) at (0.000000, 5.600000) {};
\node (-1_2_1) at (1.212436, 4.900000) {};
\node (0_1_0) at (1.212436, 3.500000) {};
\node (-2_2_1) at (-1.212436, 4.900000) {};
\node (-2_0_0) at (-4.849742, 1.400000) {};
\node (-2_-1_1) at (-4.849742, -1.400000) {};
\node (-2_-1_0) at (-6.062178, -0.700000) {};
\node (-3_0_1) at (-6.062178, 0.700000) {};
\node (-2_2_0) at (-2.424871, 5.600000) {};
\node (-2_1_0) at (-3.637307, 3.500000) {};
\node (-3_2_1) at (-3.637307, 4.900000) {};
\node (-3_2_0) at (-4.849742, 5.600000) {};
\node (-3_1_1) at (-4.849742, 2.800000) {};
\node (-3_1_0) at (-6.062178, 3.500000) {};
\node (-4_2_1) at (-6.062178, 4.900000) {};
\node (-4_2_0) at (-7.274613, 5.600000) {};
\node (-4_1_1) at (-7.274613, 2.800000) {};
\node (-4_1_0) at (-8.487049, 3.500000) {};
\node (-5_2_1) at (-8.487049, 4.900000) {};
\draw[edge] (-1_0_1) -- (0_0_0);
\draw[periedge] (0_0_0) -- (0_0_1);
\draw[edge] (0_0_1) -- (1_-1_0);
\draw[periedge] (0_-1_1) -- (1_-1_0);
\draw[periedge] (0_-1_0) -- (0_-1_1);
\draw[edge] (-1_0_1) -- (0_-1_0);
\draw[periedge] (0_0_1) -- (1_0_0);
\draw[edge] (1_0_0) -- (1_0_1);
\draw[periedge] (1_0_1) -- (2_-1_0);
\draw[edge] (1_-1_1) -- (2_-1_0);
\draw[periedge] (1_-1_0) -- (1_-1_1);
\draw[periedge] (0_1_1) -- (1_1_0);
\draw[edge] (1_1_0) -- (1_1_1);
\draw[edge] (1_1_1) -- (2_0_0);
\draw[periedge] (1_0_1) -- (2_0_0);
\draw[periedge] (0_1_1) -- (1_0_0);
\draw[periedge] (0_2_1) -- (1_2_0);
\draw[periedge] (1_2_0) -- (1_2_1);
\draw[periedge] (1_2_1) -- (2_1_0);
\draw[edge] (1_1_1) -- (2_1_0);
\draw[periedge] (0_2_1) -- (1_1_0);
\draw[periedge] (1_-2_1) -- (2_-2_0);
\draw[edge] (2_-2_0) -- (2_-2_1);
\draw[edge] (2_-2_1) -- (3_-3_0);
\draw[periedge] (2_-3_1) -- (3_-3_0);
\draw[periedge] (2_-3_0) -- (2_-3_1);
\draw[periedge] (1_-2_1) -- (2_-3_0);
\draw[periedge] (2_-1_0) -- (2_-1_1);
\draw[periedge] (2_-1_1) -- (3_-2_0);
\draw[edge] (2_-2_1) -- (3_-2_0);
\draw[periedge] (1_-1_1) -- (2_-2_0);
\draw[periedge] (2_1_0) -- (2_1_1);
\draw[periedge] (2_1_1) -- (3_0_0);
\draw[periedge] (2_0_1) -- (3_0_0);
\draw[periedge] (2_0_0) -- (2_0_1);
\draw[periedge] (3_-2_0) -- (3_-2_1);
\draw[periedge] (3_-2_1) -- (4_-3_0);
\draw[periedge] (3_-3_1) -- (4_-3_0);
\draw[periedge] (3_-3_0) -- (3_-3_1);
\draw[periedge] (-2_0_1) -- (-1_0_0);
\draw[edge] (-1_0_0) -- (-1_0_1);
\draw[periedge] (-1_-1_1) -- (0_-1_0);
\draw[periedge] (-1_-1_0) -- (-1_-1_1);
\draw[edge] (-2_0_1) -- (-1_-1_0);
\draw[edge] (-2_1_1) -- (-1_1_0);
\draw[edge] (-1_1_0) -- (-1_1_1);
\draw[periedge] (-1_1_1) -- (0_0_0);
\draw[periedge] (-2_1_1) -- (-1_0_0);
\draw[periedge] (-2_2_1) -- (-1_2_0);
\draw[periedge] (-1_2_0) -- (-1_2_1);
\draw[periedge] (-1_2_1) -- (0_1_0);
\draw[periedge] (-1_1_1) -- (0_1_0);
\draw[edge] (-2_2_1) -- (-1_1_0);
\draw[periedge] (-3_0_1) -- (-2_0_0);
\draw[periedge] (-2_0_0) -- (-2_0_1);
\draw[periedge] (-2_-1_1) -- (-1_-1_0);
\draw[periedge] (-2_-1_0) -- (-2_-1_1);
\draw[periedge] (-3_0_1) -- (-2_-1_0);
\draw[periedge] (-3_2_1) -- (-2_2_0);
\draw[periedge] (-2_2_0) -- (-2_2_1);
\draw[periedge] (-2_1_0) -- (-2_1_1);
\draw[edge] (-3_2_1) -- (-2_1_0);
\draw[periedge] (-4_2_1) -- (-3_2_0);
\draw[periedge] (-3_2_0) -- (-3_2_1);
\draw[periedge] (-3_1_1) -- (-2_1_0);
\draw[periedge] (-3_1_0) -- (-3_1_1);
\draw[edge] (-4_2_1) -- (-3_1_0);
\draw[periedge] (-5_2_1) -- (-4_2_0);
\draw[periedge] (-4_2_0) -- (-4_2_1);
\draw[periedge] (-4_1_1) -- (-3_1_0);
\draw[periedge] (-4_1_0) -- (-4_1_1);
\draw[periedge] (-5_2_1) -- (-4_1_0);
\end{tikzpicture} &    
\begin{tikzpicture}[scale=\scfactorx,rotate=120]
\tikzstyle{every node} = [inner sep=1, draw, circle, fill=blue!50!white]
\tikzstyle{edge} = [draw, line width=1.0]
\tikzstyle{periedge} = [draw, line width=1.0]
\node (0_0_0) at (0.000000, 1.400000) {};
\node (0_0_1) at (1.212436, 0.700000) {};
\node (1_-1_0) at (1.212436, -0.700000) {};
\node (0_-1_1) at (0.000000, -1.400000) {};
\node (0_-1_0) at (-1.212436, -0.700000) {};
\node (-1_0_1) at (-1.212436, 0.700000) {};
\node (1_0_0) at (2.424871, 1.400000) {};
\node (1_0_1) at (3.637307, 0.700000) {};
\node (2_-1_0) at (3.637307, -0.700000) {};
\node (1_-1_1) at (2.424871, -1.400000) {};
\node (-1_0_0) at (-2.424871, 1.400000) {};
\node (-1_-1_1) at (-2.424871, -1.400000) {};
\node (-1_-1_0) at (-3.637307, -0.700000) {};
\node (-2_0_1) at (-3.637307, 0.700000) {};
\node (-1_1_0) at (-1.212436, 3.500000) {};
\node (-1_1_1) at (0.000000, 2.800000) {};
\node (-2_1_1) at (-2.424871, 2.800000) {};
\node (-1_2_0) at (0.000000, 5.600000) {};
\node (-1_2_1) at (1.212436, 4.900000) {};
\node (0_1_0) at (1.212436, 3.500000) {};
\node (-2_2_1) at (-1.212436, 4.900000) {};
\node (-2_0_0) at (-4.849742, 1.400000) {};
\node (-2_-1_1) at (-4.849742, -1.400000) {};
\node (-2_-1_0) at (-6.062178, -0.700000) {};
\node (-3_0_1) at (-6.062178, 0.700000) {};
\node (-2_2_0) at (-2.424871, 5.600000) {};
\node (-2_1_0) at (-3.637307, 3.500000) {};
\node (-3_2_1) at (-3.637307, 4.900000) {};
\node (-3_2_0) at (-4.849742, 5.600000) {};
\node (-3_1_1) at (-4.849742, 2.800000) {};
\node (-3_1_0) at (-6.062178, 3.500000) {};
\node (-4_2_1) at (-6.062178, 4.900000) {};
\node (-4_0_0) at (-9.699485, 1.400000) {};
\node (-4_0_1) at (-8.487049, 0.700000) {};
\node (-3_-1_0) at (-8.487049, -0.700000) {};
\node (-4_-1_1) at (-9.699485, -1.400000) {};
\node (-4_-1_0) at (-10.911920, -0.700000) {};
\node (-5_0_1) at (-10.911920, 0.700000) {};
\node (-4_1_0) at (-8.487049, 3.500000) {};
\node (-4_1_1) at (-7.274613, 2.800000) {};
\node (-3_0_0) at (-7.274613, 1.400000) {};
\node (-5_1_1) at (-9.699485, 2.800000) {};
\node (-4_2_0) at (-7.274613, 5.600000) {};
\node (-5_2_1) at (-8.487049, 4.900000) {};
\node (-5_1_0) at (-10.911920, 3.500000) {};
\node (-5_0_0) at (-12.124356, 1.400000) {};
\node (-6_1_1) at (-12.124356, 2.800000) {};
\node (-5_3_0) at (-8.487049, 7.700000) {};
\node (-5_3_1) at (-7.274613, 7.000000) {};
\node (-5_2_0) at (-9.699485, 5.600000) {};
\node (-6_3_1) at (-9.699485, 7.000000) {};
\node (-5_4_0) at (-7.274613, 9.800000) {};
\node (-5_4_1) at (-6.062178, 9.100000) {};
\node (-4_3_0) at (-6.062178, 7.700000) {};
\node (-6_4_1) at (-8.487049, 9.100000) {};
\node (-6_4_0) at (-9.699485, 9.800000) {};
\node (-6_3_0) at (-10.911920, 7.700000) {};
\node (-7_4_1) at (-10.911920, 9.100000) {};
\draw[edge] (-1_0_1) -- (0_0_0);
\draw[periedge] (0_0_0) -- (0_0_1);
\draw[edge] (0_0_1) -- (1_-1_0);
\draw[periedge] (0_-1_1) -- (1_-1_0);
\draw[periedge] (0_-1_0) -- (0_-1_1);
\draw[edge] (-1_0_1) -- (0_-1_0);
\draw[periedge] (0_0_1) -- (1_0_0);
\draw[periedge] (1_0_0) -- (1_0_1);
\draw[periedge] (1_0_1) -- (2_-1_0);
\draw[periedge] (1_-1_1) -- (2_-1_0);
\draw[periedge] (1_-1_0) -- (1_-1_1);
\draw[periedge] (-2_0_1) -- (-1_0_0);
\draw[edge] (-1_0_0) -- (-1_0_1);
\draw[periedge] (-1_-1_1) -- (0_-1_0);
\draw[periedge] (-1_-1_0) -- (-1_-1_1);
\draw[edge] (-2_0_1) -- (-1_-1_0);
\draw[edge] (-2_1_1) -- (-1_1_0);
\draw[edge] (-1_1_0) -- (-1_1_1);
\draw[periedge] (-1_1_1) -- (0_0_0);
\draw[periedge] (-2_1_1) -- (-1_0_0);
\draw[periedge] (-2_2_1) -- (-1_2_0);
\draw[periedge] (-1_2_0) -- (-1_2_1);
\draw[periedge] (-1_2_1) -- (0_1_0);
\draw[periedge] (-1_1_1) -- (0_1_0);
\draw[edge] (-2_2_1) -- (-1_1_0);
\draw[periedge] (-3_0_1) -- (-2_0_0);
\draw[periedge] (-2_0_0) -- (-2_0_1);
\draw[periedge] (-2_-1_1) -- (-1_-1_0);
\draw[periedge] (-2_-1_0) -- (-2_-1_1);
\draw[periedge] (-3_0_1) -- (-2_-1_0);
\draw[periedge] (-3_2_1) -- (-2_2_0);
\draw[periedge] (-2_2_0) -- (-2_2_1);
\draw[periedge] (-2_1_0) -- (-2_1_1);
\draw[edge] (-3_2_1) -- (-2_1_0);
\draw[periedge] (-4_2_1) -- (-3_2_0);
\draw[periedge] (-3_2_0) -- (-3_2_1);
\draw[periedge] (-3_1_1) -- (-2_1_0);
\draw[periedge] (-3_1_0) -- (-3_1_1);
\draw[edge] (-4_2_1) -- (-3_1_0);
\draw[edge] (-5_0_1) -- (-4_0_0);
\draw[edge] (-4_0_0) -- (-4_0_1);
\draw[periedge] (-4_0_1) -- (-3_-1_0);
\draw[periedge] (-4_-1_1) -- (-3_-1_0);
\draw[periedge] (-4_-1_0) -- (-4_-1_1);
\draw[periedge] (-5_0_1) -- (-4_-1_0);
\draw[periedge] (-5_1_1) -- (-4_1_0);
\draw[edge] (-4_1_0) -- (-4_1_1);
\draw[periedge] (-4_1_1) -- (-3_0_0);
\draw[periedge] (-4_0_1) -- (-3_0_0);
\draw[edge] (-5_1_1) -- (-4_0_0);
\draw[edge] (-5_2_1) -- (-4_2_0);
\draw[periedge] (-4_2_0) -- (-4_2_1);
\draw[periedge] (-4_1_1) -- (-3_1_0);
\draw[periedge] (-5_2_1) -- (-4_1_0);
\draw[periedge] (-6_1_1) -- (-5_1_0);
\draw[periedge] (-5_1_0) -- (-5_1_1);
\draw[periedge] (-5_0_0) -- (-5_0_1);
\draw[periedge] (-6_1_1) -- (-5_0_0);
\draw[edge] (-6_3_1) -- (-5_3_0);
\draw[edge] (-5_3_0) -- (-5_3_1);
\draw[periedge] (-5_3_1) -- (-4_2_0);
\draw[periedge] (-5_2_0) -- (-5_2_1);
\draw[periedge] (-6_3_1) -- (-5_2_0);
\draw[periedge] (-6_4_1) -- (-5_4_0);
\draw[periedge] (-5_4_0) -- (-5_4_1);
\draw[periedge] (-5_4_1) -- (-4_3_0);
\draw[periedge] (-5_3_1) -- (-4_3_0);
\draw[edge] (-6_4_1) -- (-5_3_0);
\draw[periedge] (-7_4_1) -- (-6_4_0);
\draw[periedge] (-6_4_0) -- (-6_4_1);
\draw[periedge] (-6_3_0) -- (-6_3_1);
\draw[periedge] (-7_4_1) -- (-6_3_0);
\end{tikzpicture} \\
\text{(d)}\ 6 & \text{(e)}\ 10 & \text{(f)}\ 10  \\[12pt]
\begin{tikzpicture}[scale=\scfactorx,rotate=60]
\tikzstyle{every node} = [inner sep=1, draw, circle, fill=blue!50!white]
\tikzstyle{edge} = [draw, line width=1.0]
\tikzstyle{periedge} = [draw, line width=1.0]
\node (0_0_0) at (0.000000, 1.400000) {};
\node (0_0_1) at (1.212436, 0.700000) {};
\node (1_-1_0) at (1.212436, -0.700000) {};
\node (0_-1_1) at (0.000000, -1.400000) {};
\node (0_-1_0) at (-1.212436, -0.700000) {};
\node (-1_0_1) at (-1.212436, 0.700000) {};
\node (1_0_0) at (2.424871, 1.400000) {};
\node (1_0_1) at (3.637307, 0.700000) {};
\node (2_-1_0) at (3.637307, -0.700000) {};
\node (1_-1_1) at (2.424871, -1.400000) {};
\node (2_-1_1) at (4.849742, -1.400000) {};
\node (3_-2_0) at (4.849742, -2.800000) {};
\node (2_-2_1) at (3.637307, -3.500000) {};
\node (2_-2_0) at (2.424871, -2.800000) {};
\node (2_0_0) at (4.849742, 1.400000) {};
\node (2_0_1) at (6.062178, 0.700000) {};
\node (3_-1_0) at (6.062178, -0.700000) {};
\node (2_1_0) at (6.062178, 3.500000) {};
\node (2_1_1) at (7.274613, 2.800000) {};
\node (3_0_0) at (7.274613, 1.400000) {};
\node (1_1_1) at (4.849742, 2.800000) {};
\node (2_2_0) at (7.274613, 5.600000) {};
\node (2_2_1) at (8.487049, 4.900000) {};
\node (3_1_0) at (8.487049, 3.500000) {};
\node (1_2_1) at (6.062178, 4.900000) {};
\node (3_-2_1) at (6.062178, -3.500000) {};
\node (4_-3_0) at (6.062178, -4.900000) {};
\node (3_-3_1) at (4.849742, -5.600000) {};
\node (3_-3_0) at (3.637307, -4.900000) {};
\node (4_-2_0) at (7.274613, -2.800000) {};
\node (4_-2_1) at (8.487049, -3.500000) {};
\node (5_-3_0) at (8.487049, -4.900000) {};
\node (4_-3_1) at (7.274613, -5.600000) {};
\node (5_-2_0) at (9.699485, -2.800000) {};
\node (5_-2_1) at (10.911920, -3.500000) {};
\node (6_-3_0) at (10.911920, -4.900000) {};
\node (5_-3_1) at (9.699485, -5.600000) {};
\node (6_-5_0) at (8.487049, -9.100000) {};
\node (6_-5_1) at (9.699485, -9.800000) {};
\node (7_-6_0) at (9.699485, -11.200000) {};
\node (6_-6_1) at (8.487049, -11.900000) {};
\node (6_-6_0) at (7.274613, -11.200000) {};
\node (5_-5_1) at (7.274613, -9.800000) {};
\node (6_-4_0) at (9.699485, -7.000000) {};
\node (6_-4_1) at (10.911920, -7.700000) {};
\node (7_-5_0) at (10.911920, -9.100000) {};
\node (5_-4_1) at (8.487049, -7.700000) {};
\node (6_-3_1) at (12.124356, -5.600000) {};
\node (7_-4_0) at (12.124356, -7.000000) {};
\node (6_-2_0) at (12.124356, -2.800000) {};
\node (6_-2_1) at (13.336791, -3.500000) {};
\node (7_-3_0) at (13.336791, -4.900000) {};
\node (-1_0_0) at (-2.424871, 1.400000) {};
\node (-1_-1_1) at (-2.424871, -1.400000) {};
\node (-1_-1_0) at (-3.637307, -0.700000) {};
\node (-2_0_1) at (-3.637307, 0.700000) {};
\node (-2_0_0) at (-4.849742, 1.400000) {};
\node (-2_-1_1) at (-4.849742, -1.400000) {};
\node (-2_-1_0) at (-6.062178, -0.700000) {};
\node (-3_0_1) at (-6.062178, 0.700000) {};
\draw[periedge] (-1_0_1) -- (0_0_0);
\draw[periedge] (0_0_0) -- (0_0_1);
\draw[edge] (0_0_1) -- (1_-1_0);
\draw[periedge] (0_-1_1) -- (1_-1_0);
\draw[periedge] (0_-1_0) -- (0_-1_1);
\draw[edge] (-1_0_1) -- (0_-1_0);
\draw[periedge] (0_0_1) -- (1_0_0);
\draw[periedge] (1_0_0) -- (1_0_1);
\draw[edge] (1_0_1) -- (2_-1_0);
\draw[edge] (1_-1_1) -- (2_-1_0);
\draw[periedge] (1_-1_0) -- (1_-1_1);
\draw[edge] (2_-1_0) -- (2_-1_1);
\draw[periedge] (2_-1_1) -- (3_-2_0);
\draw[edge] (2_-2_1) -- (3_-2_0);
\draw[periedge] (2_-2_0) -- (2_-2_1);
\draw[periedge] (1_-1_1) -- (2_-2_0);
\draw[periedge] (1_0_1) -- (2_0_0);
\draw[edge] (2_0_0) -- (2_0_1);
\draw[periedge] (2_0_1) -- (3_-1_0);
\draw[periedge] (2_-1_1) -- (3_-1_0);
\draw[periedge] (1_1_1) -- (2_1_0);
\draw[edge] (2_1_0) -- (2_1_1);
\draw[periedge] (2_1_1) -- (3_0_0);
\draw[periedge] (2_0_1) -- (3_0_0);
\draw[periedge] (1_1_1) -- (2_0_0);
\draw[periedge] (1_2_1) -- (2_2_0);
\draw[periedge] (2_2_0) -- (2_2_1);
\draw[periedge] (2_2_1) -- (3_1_0);
\draw[periedge] (2_1_1) -- (3_1_0);
\draw[periedge] (1_2_1) -- (2_1_0);
\draw[periedge] (3_-2_0) -- (3_-2_1);
\draw[edge] (3_-2_1) -- (4_-3_0);
\draw[periedge] (3_-3_1) -- (4_-3_0);
\draw[periedge] (3_-3_0) -- (3_-3_1);
\draw[periedge] (2_-2_1) -- (3_-3_0);
\draw[periedge] (3_-2_1) -- (4_-2_0);
\draw[periedge] (4_-2_0) -- (4_-2_1);
\draw[edge] (4_-2_1) -- (5_-3_0);
\draw[periedge] (4_-3_1) -- (5_-3_0);
\draw[periedge] (4_-3_0) -- (4_-3_1);
\draw[periedge] (4_-2_1) -- (5_-2_0);
\draw[periedge] (5_-2_0) -- (5_-2_1);
\draw[edge] (5_-2_1) -- (6_-3_0);
\draw[edge] (5_-3_1) -- (6_-3_0);
\draw[periedge] (5_-3_0) -- (5_-3_1);
\draw[periedge] (5_-5_1) -- (6_-5_0);
\draw[edge] (6_-5_0) -- (6_-5_1);
\draw[periedge] (6_-5_1) -- (7_-6_0);
\draw[periedge] (6_-6_1) -- (7_-6_0);
\draw[periedge] (6_-6_0) -- (6_-6_1);
\draw[periedge] (5_-5_1) -- (6_-6_0);
\draw[periedge] (5_-4_1) -- (6_-4_0);
\draw[edge] (6_-4_0) -- (6_-4_1);
\draw[periedge] (6_-4_1) -- (7_-5_0);
\draw[periedge] (6_-5_1) -- (7_-5_0);
\draw[periedge] (5_-4_1) -- (6_-5_0);
\draw[edge] (6_-3_0) -- (6_-3_1);
\draw[periedge] (6_-3_1) -- (7_-4_0);
\draw[periedge] (6_-4_1) -- (7_-4_0);
\draw[periedge] (5_-3_1) -- (6_-4_0);
\draw[periedge] (5_-2_1) -- (6_-2_0);
\draw[periedge] (6_-2_0) -- (6_-2_1);
\draw[periedge] (6_-2_1) -- (7_-3_0);
\draw[periedge] (6_-3_1) -- (7_-3_0);
\draw[periedge] (-2_0_1) -- (-1_0_0);
\draw[periedge] (-1_0_0) -- (-1_0_1);
\draw[periedge] (-1_-1_1) -- (0_-1_0);
\draw[periedge] (-1_-1_0) -- (-1_-1_1);
\draw[edge] (-2_0_1) -- (-1_-1_0);
\draw[periedge] (-3_0_1) -- (-2_0_0);
\draw[periedge] (-2_0_0) -- (-2_0_1);
\draw[periedge] (-2_-1_1) -- (-1_-1_0);
\draw[periedge] (-2_-1_0) -- (-2_-1_1);
\draw[periedge] (-3_0_1) -- (-2_-1_0);
\end{tikzpicture} &  
\begin{tikzpicture}[scale=\scfactorx]
\tikzstyle{every node} = [inner sep=1, draw, circle, fill=blue!50!white]
\tikzstyle{edge} = [draw, line width=1.0]
\tikzstyle{periedge} = [draw, line width=1.0]
\node (0_0_0) at (0.000000, 1.400000) {};
\node (0_0_1) at (1.212436, 0.700000) {};
\node (1_-1_0) at (1.212436, -0.700000) {};
\node (0_-1_1) at (0.000000, -1.400000) {};
\node (0_-1_0) at (-1.212436, -0.700000) {};
\node (-1_0_1) at (-1.212436, 0.700000) {};
\node (0_2_0) at (2.424871, 5.600000) {};
\node (0_2_1) at (3.637307, 4.900000) {};
\node (1_1_0) at (3.637307, 3.500000) {};
\node (0_1_1) at (2.424871, 2.800000) {};
\node (0_1_0) at (1.212436, 3.500000) {};
\node (-1_2_1) at (1.212436, 4.900000) {};
\node (0_4_0) at (4.849742, 9.800000) {};
\node (0_4_1) at (6.062178, 9.100000) {};
\node (1_3_0) at (6.062178, 7.700000) {};
\node (0_3_1) at (4.849742, 7.000000) {};
\node (0_3_0) at (3.637307, 7.700000) {};
\node (-1_4_1) at (3.637307, 9.100000) {};
\node (-1_0_0) at (-2.424871, 1.400000) {};
\node (-1_-1_1) at (-2.424871, -1.400000) {};
\node (-1_-1_0) at (-3.637307, -0.700000) {};
\node (-2_0_1) at (-3.637307, 0.700000) {};
\node (-1_1_0) at (-1.212436, 3.500000) {};
\node (-1_1_1) at (0.000000, 2.800000) {};
\node (-2_1_1) at (-2.424871, 2.800000) {};
\node (-1_2_0) at (0.000000, 5.600000) {};
\node (-2_2_1) at (-1.212436, 4.900000) {};
\node (-1_3_0) at (1.212436, 7.700000) {};
\node (-1_3_1) at (2.424871, 7.000000) {};
\node (-2_3_1) at (0.000000, 7.000000) {};
\node (-1_4_0) at (2.424871, 9.800000) {};
\node (-2_4_1) at (1.212436, 9.100000) {};
\node (-2_0_0) at (-4.849742, 1.400000) {};
\node (-2_-1_1) at (-4.849742, -1.400000) {};
\node (-2_-1_0) at (-6.062178, -0.700000) {};
\node (-3_0_1) at (-6.062178, 0.700000) {};
\node (-2_5_0) at (1.212436, 11.900000) {};
\node (-2_5_1) at (2.424871, 11.200000) {};
\node (-2_4_0) at (0.000000, 9.800000) {};
\node (-3_5_1) at (-0.000000, 11.200000) {};
\node (-2_6_0) at (2.424871, 14.000000) {};
\node (-2_6_1) at (3.637307, 13.300000) {};
\node (-1_5_0) at (3.637307, 11.900000) {};
\node (-3_6_1) at (1.212436, 13.300000) {};
\node (-2_7_0) at (3.637307, 16.100000) {};
\node (-2_7_1) at (4.849742, 15.400000) {};
\node (-1_6_0) at (4.849742, 14.000000) {};
\node (-3_7_1) at (2.424871, 15.400000) {};
\node (-3_6_0) at (0.000000, 14.000000) {};
\node (-3_5_0) at (-1.212436, 11.900000) {};
\node (-4_6_1) at (-1.212436, 13.300000) {};
\node (-3_7_0) at (1.212436, 16.100000) {};
\node (-4_7_1) at (-0.000000, 15.400000) {};
\node (-4_7_0) at (-1.212436, 16.100000) {};
\node (-4_6_0) at (-2.424871, 14.000000) {};
\node (-5_7_1) at (-2.424871, 15.400000) {};
\draw[edge] (-1_0_1) -- (0_0_0);
\draw[periedge] (0_0_0) -- (0_0_1);
\draw[periedge] (0_0_1) -- (1_-1_0);
\draw[periedge] (0_-1_1) -- (1_-1_0);
\draw[periedge] (0_-1_0) -- (0_-1_1);
\draw[edge] (-1_0_1) -- (0_-1_0);
\draw[edge] (-1_2_1) -- (0_2_0);
\draw[periedge] (0_2_0) -- (0_2_1);
\draw[periedge] (0_2_1) -- (1_1_0);
\draw[periedge] (0_1_1) -- (1_1_0);
\draw[periedge] (0_1_0) -- (0_1_1);
\draw[edge] (-1_2_1) -- (0_1_0);
\draw[periedge] (-1_4_1) -- (0_4_0);
\draw[periedge] (0_4_0) -- (0_4_1);
\draw[periedge] (0_4_1) -- (1_3_0);
\draw[periedge] (0_3_1) -- (1_3_0);
\draw[periedge] (0_3_0) -- (0_3_1);
\draw[edge] (-1_4_1) -- (0_3_0);
\draw[periedge] (-2_0_1) -- (-1_0_0);
\draw[edge] (-1_0_0) -- (-1_0_1);
\draw[periedge] (-1_-1_1) -- (0_-1_0);
\draw[periedge] (-1_-1_0) -- (-1_-1_1);
\draw[edge] (-2_0_1) -- (-1_-1_0);
\draw[periedge] (-2_1_1) -- (-1_1_0);
\draw[edge] (-1_1_0) -- (-1_1_1);
\draw[periedge] (-1_1_1) -- (0_0_0);
\draw[periedge] (-2_1_1) -- (-1_0_0);
\draw[periedge] (-2_2_1) -- (-1_2_0);
\draw[edge] (-1_2_0) -- (-1_2_1);
\draw[periedge] (-1_1_1) -- (0_1_0);
\draw[periedge] (-2_2_1) -- (-1_1_0);
\draw[periedge] (-2_3_1) -- (-1_3_0);
\draw[edge] (-1_3_0) -- (-1_3_1);
\draw[periedge] (-1_3_1) -- (0_2_0);
\draw[periedge] (-2_3_1) -- (-1_2_0);
\draw[edge] (-2_4_1) -- (-1_4_0);
\draw[periedge] (-1_4_0) -- (-1_4_1);
\draw[periedge] (-1_3_1) -- (0_3_0);
\draw[periedge] (-2_4_1) -- (-1_3_0);
\draw[periedge] (-3_0_1) -- (-2_0_0);
\draw[periedge] (-2_0_0) -- (-2_0_1);
\draw[periedge] (-2_-1_1) -- (-1_-1_0);
\draw[periedge] (-2_-1_0) -- (-2_-1_1);
\draw[periedge] (-3_0_1) -- (-2_-1_0);
\draw[edge] (-3_5_1) -- (-2_5_0);
\draw[edge] (-2_5_0) -- (-2_5_1);
\draw[periedge] (-2_5_1) -- (-1_4_0);
\draw[periedge] (-2_4_0) -- (-2_4_1);
\draw[periedge] (-3_5_1) -- (-2_4_0);
\draw[edge] (-3_6_1) -- (-2_6_0);
\draw[edge] (-2_6_0) -- (-2_6_1);
\draw[periedge] (-2_6_1) -- (-1_5_0);
\draw[periedge] (-2_5_1) -- (-1_5_0);
\draw[edge] (-3_6_1) -- (-2_5_0);
\draw[periedge] (-3_7_1) -- (-2_7_0);
\draw[periedge] (-2_7_0) -- (-2_7_1);
\draw[periedge] (-2_7_1) -- (-1_6_0);
\draw[periedge] (-2_6_1) -- (-1_6_0);
\draw[edge] (-3_7_1) -- (-2_6_0);
\draw[edge] (-4_6_1) -- (-3_6_0);
\draw[edge] (-3_6_0) -- (-3_6_1);
\draw[periedge] (-3_5_0) -- (-3_5_1);
\draw[periedge] (-4_6_1) -- (-3_5_0);
\draw[periedge] (-4_7_1) -- (-3_7_0);
\draw[periedge] (-3_7_0) -- (-3_7_1);
\draw[edge] (-4_7_1) -- (-3_6_0);
\draw[periedge] (-5_7_1) -- (-4_7_0);
\draw[periedge] (-4_7_0) -- (-4_7_1);
\draw[periedge] (-4_6_0) -- (-4_6_1);
\draw[periedge] (-5_7_1) -- (-4_6_0);
\end{tikzpicture} &
\begin{tikzpicture}[scale=\scfactorx,rotate=60]
\tikzstyle{every node} = [inner sep=1, draw, circle, fill=blue!50!white]
\tikzstyle{edge} = [draw, line width=1.0]
\tikzstyle{periedge} = [draw, line width=1.0]
\node (0_0_0) at (0.000000, 1.400000) {};
\node (0_0_1) at (1.212436, 0.700000) {};
\node (1_-1_0) at (1.212436, -0.700000) {};
\node (0_-1_1) at (0.000000, -1.400000) {};
\node (0_-1_0) at (-1.212436, -0.700000) {};
\node (-1_0_1) at (-1.212436, 0.700000) {};
\node (1_0_0) at (2.424871, 1.400000) {};
\node (1_0_1) at (3.637307, 0.700000) {};
\node (2_-1_0) at (3.637307, -0.700000) {};
\node (1_-1_1) at (2.424871, -1.400000) {};
\node (1_1_0) at (3.637307, 3.500000) {};
\node (1_1_1) at (4.849742, 2.800000) {};
\node (2_0_0) at (4.849742, 1.400000) {};
\node (0_1_1) at (2.424871, 2.800000) {};
\node (2_0_1) at (6.062178, 0.700000) {};
\node (3_-1_0) at (6.062178, -0.700000) {};
\node (2_-1_1) at (4.849742, -1.400000) {};
\node (3_-1_1) at (7.274613, -1.400000) {};
\node (4_-2_0) at (7.274613, -2.800000) {};
\node (3_-2_1) at (6.062178, -3.500000) {};
\node (3_-2_0) at (4.849742, -2.800000) {};
\node (3_0_0) at (7.274613, 1.400000) {};
\node (3_0_1) at (8.487049, 0.700000) {};
\node (4_-1_0) at (8.487049, -0.700000) {};
\node (4_0_0) at (9.699485, 1.400000) {};
\node (4_0_1) at (10.911920, 0.700000) {};
\node (5_-1_0) at (10.911920, -0.700000) {};
\node (4_-1_1) at (9.699485, -1.400000) {};
\node (4_1_0) at (10.911920, 3.500000) {};
\node (4_1_1) at (12.124356, 2.800000) {};
\node (5_0_0) at (12.124356, 1.400000) {};
\node (3_1_1) at (9.699485, 2.800000) {};
\node (4_2_0) at (12.124356, 5.600000) {};
\node (4_2_1) at (13.336791, 4.900000) {};
\node (5_1_0) at (13.336791, 3.500000) {};
\node (3_2_1) at (10.911920, 4.900000) {};
\node (5_-2_0) at (9.699485, -2.800000) {};
\node (5_-2_1) at (10.911920, -3.500000) {};
\node (6_-3_0) at (10.911920, -4.900000) {};
\node (5_-3_1) at (9.699485, -5.600000) {};
\node (5_-3_0) at (8.487049, -4.900000) {};
\node (4_-2_1) at (8.487049, -3.500000) {};
\node (5_-1_1) at (12.124356, -1.400000) {};
\node (6_-2_0) at (12.124356, -2.800000) {};
\node (5_1_1) at (14.549227, 2.800000) {};
\node (6_0_0) at (14.549227, 1.400000) {};
\node (5_0_1) at (13.336791, 0.700000) {};
\node (6_-2_1) at (13.336791, -3.500000) {};
\node (7_-3_0) at (13.336791, -4.900000) {};
\node (6_-3_1) at (12.124356, -5.600000) {};
\node (-1_0_0) at (-2.424871, 1.400000) {};
\node (-1_-1_1) at (-2.424871, -1.400000) {};
\node (-1_-1_0) at (-3.637307, -0.700000) {};
\node (-2_0_1) at (-3.637307, 0.700000) {};
\node (-2_0_0) at (-4.849742, 1.400000) {};
\node (-2_-1_1) at (-4.849742, -1.400000) {};
\node (-2_-1_0) at (-6.062178, -0.700000) {};
\node (-3_0_1) at (-6.062178, 0.700000) {};
\draw[periedge] (-1_0_1) -- (0_0_0);
\draw[periedge] (0_0_0) -- (0_0_1);
\draw[edge] (0_0_1) -- (1_-1_0);
\draw[periedge] (0_-1_1) -- (1_-1_0);
\draw[periedge] (0_-1_0) -- (0_-1_1);
\draw[edge] (-1_0_1) -- (0_-1_0);
\draw[periedge] (0_0_1) -- (1_0_0);
\draw[edge] (1_0_0) -- (1_0_1);
\draw[edge] (1_0_1) -- (2_-1_0);
\draw[periedge] (1_-1_1) -- (2_-1_0);
\draw[periedge] (1_-1_0) -- (1_-1_1);
\draw[periedge] (0_1_1) -- (1_1_0);
\draw[periedge] (1_1_0) -- (1_1_1);
\draw[periedge] (1_1_1) -- (2_0_0);
\draw[edge] (1_0_1) -- (2_0_0);
\draw[periedge] (0_1_1) -- (1_0_0);
\draw[periedge] (2_0_0) -- (2_0_1);
\draw[edge] (2_0_1) -- (3_-1_0);
\draw[edge] (2_-1_1) -- (3_-1_0);
\draw[periedge] (2_-1_0) -- (2_-1_1);
\draw[edge] (3_-1_0) -- (3_-1_1);
\draw[periedge] (3_-1_1) -- (4_-2_0);
\draw[periedge] (3_-2_1) -- (4_-2_0);
\draw[periedge] (3_-2_0) -- (3_-2_1);
\draw[periedge] (2_-1_1) -- (3_-2_0);
\draw[periedge] (2_0_1) -- (3_0_0);
\draw[periedge] (3_0_0) -- (3_0_1);
\draw[edge] (3_0_1) -- (4_-1_0);
\draw[periedge] (3_-1_1) -- (4_-1_0);
\draw[periedge] (3_0_1) -- (4_0_0);
\draw[edge] (4_0_0) -- (4_0_1);
\draw[periedge] (4_0_1) -- (5_-1_0);
\draw[edge] (4_-1_1) -- (5_-1_0);
\draw[periedge] (4_-1_0) -- (4_-1_1);
\draw[periedge] (3_1_1) -- (4_1_0);
\draw[edge] (4_1_0) -- (4_1_1);
\draw[edge] (4_1_1) -- (5_0_0);
\draw[periedge] (4_0_1) -- (5_0_0);
\draw[periedge] (3_1_1) -- (4_0_0);
\draw[periedge] (3_2_1) -- (4_2_0);
\draw[periedge] (4_2_0) -- (4_2_1);
\draw[periedge] (4_2_1) -- (5_1_0);
\draw[edge] (4_1_1) -- (5_1_0);
\draw[periedge] (3_2_1) -- (4_1_0);
\draw[periedge] (4_-2_1) -- (5_-2_0);
\draw[edge] (5_-2_0) -- (5_-2_1);
\draw[edge] (5_-2_1) -- (6_-3_0);
\draw[periedge] (5_-3_1) -- (6_-3_0);
\draw[periedge] (5_-3_0) -- (5_-3_1);
\draw[periedge] (4_-2_1) -- (5_-3_0);
\draw[periedge] (5_-1_0) -- (5_-1_1);
\draw[periedge] (5_-1_1) -- (6_-2_0);
\draw[edge] (5_-2_1) -- (6_-2_0);
\draw[periedge] (4_-1_1) -- (5_-2_0);
\draw[periedge] (5_1_0) -- (5_1_1);
\draw[periedge] (5_1_1) -- (6_0_0);
\draw[periedge] (5_0_1) -- (6_0_0);
\draw[periedge] (5_0_0) -- (5_0_1);
\draw[periedge] (6_-2_0) -- (6_-2_1);
\draw[periedge] (6_-2_1) -- (7_-3_0);
\draw[periedge] (6_-3_1) -- (7_-3_0);
\draw[periedge] (6_-3_0) -- (6_-3_1);
\draw[periedge] (-2_0_1) -- (-1_0_0);
\draw[periedge] (-1_0_0) -- (-1_0_1);
\draw[periedge] (-1_-1_1) -- (0_-1_0);
\draw[periedge] (-1_-1_0) -- (-1_-1_1);
\draw[edge] (-2_0_1) -- (-1_-1_0);
\draw[periedge] (-3_0_1) -- (-2_0_0);
\draw[periedge] (-2_0_0) -- (-2_0_1);
\draw[periedge] (-2_-1_1) -- (-1_-1_0);
\draw[periedge] (-2_-1_0) -- (-2_-1_1);
\draw[periedge] (-3_0_1) -- (-2_-1_0);
\end{tikzpicture} \\
\text{(g)}\ 6 & \text{(h)}\ 8 & \text{(i)}\ 9 
\end{array}$
\caption{The benzenoids on $\varepsilon = 15$ hexagons for which the nullity jumps from $2$ to $4$ on altanisation.
The subcaption for each benzenoid gives the bay number. All of them have the trivial maximum point group $\mathbf{C}_\mathrm{s}$.
}
\label{fig:nullity2_excess2_benz}
\end{figure}

If we limit consideration to the subclass of convex 
benzenoids \cite{gutman2012}, we can readily take the analysis much further, 
and Table~\ref{table:convex_benzenoids} shows the results for  cases with up to $3000$ hexagons. 
In all the $760511$ examples included in this survey, the excess nullity is either $0$ or $1$. 
The paractical chemical consequence is clear. This exhaustive computation rules out the possibility that an altan of a convex benzenoid 
(which, as noted above, contains faces of length $5$ and $6$ only) will be found in small-molecule synthetic chemistry.
It also makes it tempting to make the conjecture that {\it no} convex benzenoid has excess nullity $2$,
and in particular that no Kekulean convex benzenoid has an altan with nullity $2$.
We have developed a proof strategy for this conjecture, but for reasons of space we do not give it here. 
The proof will be reported in full elsewhere. 
We also note that within the range of computations, excess nullity of $1$ in Table~\ref{table:convex_benzenoids}
is in fact found for parent nullities of $0$ or $2$ only, even though parents with higher nullities occur.

Benzenoids, as bipartite patches, have a strong connection between nullity of the parent patch
and the size of its natural attachment set, as made explicit in Corollary~\ref{cor:corollary3}.
This connection is broken for more general patches in which some faces are odd, and in particular 
for the pentagonal, heptagonal and mixed pentagon/hexagon patches treated in Tables~\ref{tbl:purepenta}, \ref{tbl:purehepta} and \ref{tbl:penthexptch}.
In these cases, the results for small numbers of faces and natural attachment sets of even size again show the 
tendency for the altans with the lowest possible excess nullity to predominate, but the cases with excess nullity $2$ 
appear to be more common than for benzenoids. The pent-hex patches include examples that appear 
as induced subgraphs in fullerenes \cite{Graver2010,Graver2014,PentaCluster}, and the class
as a whole exhibits richer behaviour than the all-hexagon benzenoids.
Some examples indicating the richer behaviour of the non-benzenoid patches are illustrated in 
Figures~\ref{fig:purepatches} and~\ref{fig:mixedpatches}.

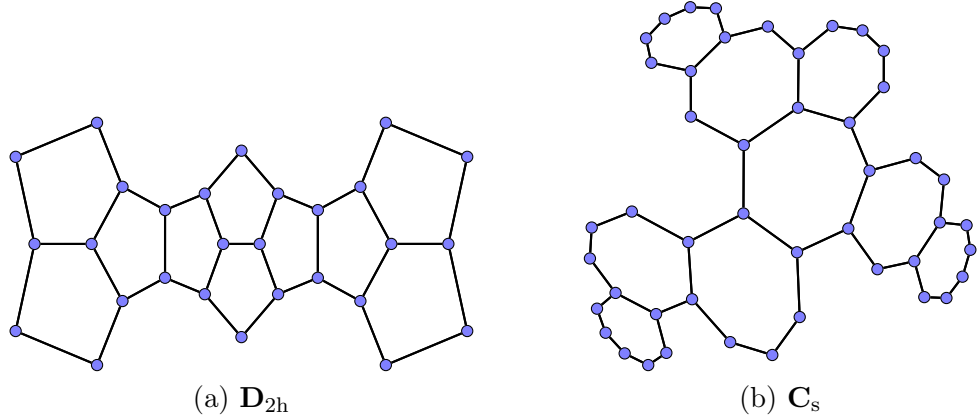
\begin{figure}[!htbp]
\centering
$\begin{array}{ccc}
\begin{tikzpicture}[scale=0.6,rotate=-22.5]
\tikzstyle{edge}=[draw, thick]
\tikzstyle{every node}=[draw, circle, fill=blue!50!white, inner sep=1.5pt]
\node (v1) at (4.236, 1.755) {};
\node (v2) at (3.068, 1.271) {};
\node (v3) at (1.951, 2.179) {};
\node (v4) at (2.92, -0.162) {};
\node (v5) at (1.277, 1.341) {};
\node (v6) at (1.851, -0.046) {};
\node (v7) at (0.324, 1.341) {};
\node (v8) at (1.177, -0.719) {};
\node (v9) at (0.374, 0.155) {};
\node (v10) at (-1.177, 0.72) {};
\node (v11) at (-0.323, -1.341) {};
\node (v12) at (-0.374, -0.155) {};
\node (v13) at (-1.851, 0.046) {};
\node (v14) at (-1.277, -1.341) {};
\node (v15) at (-2.92, 0.162) {};
\node (v16) at (-1.951, -2.179) {};
\node (v17) at (-3.068, -1.271) {};
\node (v18) at (-4.236, -1.755) {};
\node (v19) at (3.883, 3.699) {};
\node (v20) at (1.928, 3.706) {};
\node (v21) at (5.361, 0.13) {};
\node (v22) at (3.984, -1.257) {};
\node (v23) at (-0.791, 1.91) {};
\node (v24) at (0.791, -1.909) {};
\node (v25) at (-3.984, 1.257) {};
\node (v26) at (-5.361, -0.13) {};
\node (v27) at (-1.928, -3.706) {};
\node (v28) at (-3.883, -3.699) {};
\path[edge] (v1) -- (v2);
\path[edge] (v1) -- (v19);
\path[edge] (v1) -- (v21);
\path[edge] (v2) -- (v1);
\path[edge] (v2) -- (v4);
\path[edge] (v2) -- (v3);
\path[edge] (v3) -- (v2);
\path[edge] (v3) -- (v5);
\path[edge] (v3) -- (v20);
\path[edge] (v4) -- (v2);
\path[edge] (v4) -- (v22);
\path[edge] (v4) -- (v6);
\path[edge] (v5) -- (v6);
\path[edge] (v5) -- (v7);
\path[edge] (v5) -- (v3);
\path[edge] (v6) -- (v5);
\path[edge] (v6) -- (v4);
\path[edge] (v6) -- (v8);
\path[edge] (v7) -- (v9);
\path[edge] (v7) -- (v23);
\path[edge] (v7) -- (v5);
\path[edge] (v8) -- (v9);
\path[edge] (v8) -- (v6);
\path[edge] (v8) -- (v24);
\path[edge] (v9) -- (v7);
\path[edge] (v9) -- (v8);
\path[edge] (v9) -- (v12);
\path[edge] (v10) -- (v12);
\path[edge] (v10) -- (v13);
\path[edge] (v10) -- (v23);
\path[edge] (v11) -- (v12);
\path[edge] (v11) -- (v24);
\path[edge] (v11) -- (v14);
\path[edge] (v12) -- (v9);
\path[edge] (v12) -- (v11);
\path[edge] (v12) -- (v10);
\path[edge] (v13) -- (v14);
\path[edge] (v13) -- (v15);
\path[edge] (v13) -- (v10);
\path[edge] (v14) -- (v13);
\path[edge] (v14) -- (v11);
\path[edge] (v14) -- (v16);
\path[edge] (v15) -- (v17);
\path[edge] (v15) -- (v25);
\path[edge] (v15) -- (v13);
\path[edge] (v16) -- (v17);
\path[edge] (v16) -- (v14);
\path[edge] (v16) -- (v27);
\path[edge] (v17) -- (v15);
\path[edge] (v17) -- (v16);
\path[edge] (v17) -- (v18);
\path[edge] (v18) -- (v17);
\path[edge] (v18) -- (v28);
\path[edge] (v18) -- (v26);
\path[edge] (v19) -- (v1);
\path[edge] (v19) -- (v20);
\path[edge] (v20) -- (v19);
\path[edge] (v20) -- (v3);
\path[edge] (v21) -- (v1);
\path[edge] (v21) -- (v22);
\path[edge] (v22) -- (v21);
\path[edge] (v22) -- (v4);
\path[edge] (v23) -- (v7);
\path[edge] (v23) -- (v10);
\path[edge] (v24) -- (v8);
\path[edge] (v24) -- (v11);
\path[edge] (v25) -- (v15);
\path[edge] (v25) -- (v26);
\path[edge] (v26) -- (v25);
\path[edge] (v26) -- (v18);
\path[edge] (v27) -- (v16);
\path[edge] (v27) -- (v28);
\path[edge] (v28) -- (v27);
\path[edge] (v28) -- (v18);
\end{tikzpicture} & \qquad \qquad &
\begin{tikzpicture}[scale=0.4,rotate=85]
\tikzstyle{edge}=[draw, thick]
\tikzstyle{every node}=[draw, circle, fill=blue!50!white, inner sep=1.5pt]
\node (v1) at (5.031, 1.944) {};
\node (v2) at (3.815, 2.979) {};
\node (v3) at (4.714, -0.538) {};
\node (v4) at (1.518, 1.011) {};
\node (v5) at (2.912, -0.677) {};
\node (v6) at (2.566, -2.428) {};
\node (v7) at (-0.751, 0.824) {};
\node (v8) at (1.032, -3.227) {};
\node (v9) at (-1.881, -1.063) {};
\node (v10) at (-1.857, 2.568) {};
\node (v11) at (-0.945, -2.692) {};
\node (v12) at (-0.478, -5.707) {};
\node (v13) at (-3.74, 2.276) {};
\node (v14) at (-3.753, 4.828) {};
\node (v15) at (-1.838, -4.973) {};
\node (v16) at (-4.338, 3.426) {};
\node (v17) at (5.966, 2.33) {};
\node (v18) at (5.93, 3.181) {};
\node (v19) at (5.476, 3.997) {};
\node (v20) at (4.772, 4.555) {};
\node (v21) at (3.976, 4.31) {};
\node (v22) at (5.515, 0.558) {};
\node (v23) at (2.308, 2.845) {};
\node (v24) at (5.72, -1.581) {};
\node (v25) at (5.659, -2.588) {};
\node (v26) at (5.045, -3.352) {};
\node (v27) at (3.842, -3.456) {};
\node (v28) at (1.583, -4.72) {};
\node (v29) at (0.954, -5.719) {};
\node (v30) at (-4.019, -1.339) {};
\node (v31) at (-5.356, -0.542) {};
\node (v32) at (-5.052, 0.89) {};
\node (v33) at (-1.01, 4.524) {};
\node (v34) at (-1.656, 5.807) {};
\node (v35) at (-2.686, 5.776) {};
\node (v36) at (-2.209, -3.774) {};
\node (v37) at (-0.495, -6.597) {};
\node (v38) at (-1.297, -6.794) {};
\node (v39) at (-2.21, -6.613) {};
\node (v40) at (-2.963, -6.142) {};
\node (v41) at (-2.998, -5.407) {};
\node (v42) at (-4.329, 5.396) {};
\node (v43) at (-5.101, 5.041) {};
\node (v44) at (-5.731, 4.351) {};
\node (v45) at (-6.049, 3.537) {};
\node (v46) at (-5.592, 2.972) {};
\path[edge] (v1) -- (v2);
\path[edge] (v1) -- (v17);
\path[edge] (v1) -- (v22);
\path[edge] (v2) -- (v1);
\path[edge] (v2) -- (v23);
\path[edge] (v2) -- (v21);
\path[edge] (v3) -- (v5);
\path[edge] (v3) -- (v22);
\path[edge] (v3) -- (v24);
\path[edge] (v4) -- (v5);
\path[edge] (v4) -- (v7);
\path[edge] (v4) -- (v23);
\path[edge] (v5) -- (v3);
\path[edge] (v5) -- (v6);
\path[edge] (v5) -- (v4);
\path[edge] (v6) -- (v5);
\path[edge] (v6) -- (v27);
\path[edge] (v6) -- (v8);
\path[edge] (v7) -- (v9);
\path[edge] (v7) -- (v10);
\path[edge] (v7) -- (v4);
\path[edge] (v8) -- (v11);
\path[edge] (v8) -- (v6);
\path[edge] (v8) -- (v28);
\path[edge] (v9) -- (v7);
\path[edge] (v9) -- (v11);
\path[edge] (v9) -- (v30);
\path[edge] (v10) -- (v13);
\path[edge] (v10) -- (v33);
\path[edge] (v10) -- (v7);
\path[edge] (v11) -- (v8);
\path[edge] (v11) -- (v36);
\path[edge] (v11) -- (v9);
\path[edge] (v12) -- (v15);
\path[edge] (v12) -- (v29);
\path[edge] (v12) -- (v37);
\path[edge] (v13) -- (v10);
\path[edge] (v13) -- (v32);
\path[edge] (v13) -- (v16);
\path[edge] (v14) -- (v16);
\path[edge] (v14) -- (v42);
\path[edge] (v14) -- (v35);
\path[edge] (v15) -- (v12);
\path[edge] (v15) -- (v41);
\path[edge] (v15) -- (v36);
\path[edge] (v16) -- (v14);
\path[edge] (v16) -- (v13);
\path[edge] (v16) -- (v46);
\path[edge] (v17) -- (v1);
\path[edge] (v17) -- (v18);
\path[edge] (v18) -- (v17);
\path[edge] (v18) -- (v19);
\path[edge] (v19) -- (v18);
\path[edge] (v19) -- (v20);
\path[edge] (v20) -- (v19);
\path[edge] (v20) -- (v21);
\path[edge] (v21) -- (v20);
\path[edge] (v21) -- (v2);
\path[edge] (v22) -- (v1);
\path[edge] (v22) -- (v3);
\path[edge] (v23) -- (v2);
\path[edge] (v23) -- (v4);
\path[edge] (v24) -- (v3);
\path[edge] (v24) -- (v25);
\path[edge] (v25) -- (v24);
\path[edge] (v25) -- (v26);
\path[edge] (v26) -- (v25);
\path[edge] (v26) -- (v27);
\path[edge] (v27) -- (v26);
\path[edge] (v27) -- (v6);
\path[edge] (v28) -- (v8);
\path[edge] (v28) -- (v29);
\path[edge] (v29) -- (v28);
\path[edge] (v29) -- (v12);
\path[edge] (v30) -- (v9);
\path[edge] (v30) -- (v31);
\path[edge] (v31) -- (v30);
\path[edge] (v31) -- (v32);
\path[edge] (v32) -- (v31);
\path[edge] (v32) -- (v13);
\path[edge] (v33) -- (v10);
\path[edge] (v33) -- (v34);
\path[edge] (v34) -- (v33);
\path[edge] (v34) -- (v35);
\path[edge] (v35) -- (v34);
\path[edge] (v35) -- (v14);
\path[edge] (v36) -- (v11);
\path[edge] (v36) -- (v15);
\path[edge] (v37) -- (v12);
\path[edge] (v37) -- (v38);
\path[edge] (v38) -- (v37);
\path[edge] (v38) -- (v39);
\path[edge] (v39) -- (v38);
\path[edge] (v39) -- (v40);
\path[edge] (v40) -- (v39);
\path[edge] (v40) -- (v41);
\path[edge] (v41) -- (v40);
\path[edge] (v41) -- (v15);
\path[edge] (v42) -- (v14);
\path[edge] (v42) -- (v43);
\path[edge] (v43) -- (v42);
\path[edge] (v43) -- (v44);
\path[edge] (v44) -- (v43);
\path[edge] (v44) -- (v45);
\path[edge] (v45) -- (v44);
\path[edge] (v45) -- (v46);
\path[edge] (v46) -- (v45);
\path[edge] (v46) -- (v16);
\end{tikzpicture} \\
\text{(a)}\ \mathbf{D}_\mathrm{2h} & & \text{(b)}\  \mathbf{C}_\mathrm{s} 
\end{array}$
\caption{Smallest patches with nullity combination $\eta(\Pi) = 1$, $\eta(\altan(\Pi)) = 3$, which 
cannot occur for benzenoids: (a) a pure pentagonal patch with $10$ pentagons; (b) 
a pure heptagonal patch with $9$ heptagons.}
\label{fig:purepatches}
\end{figure}

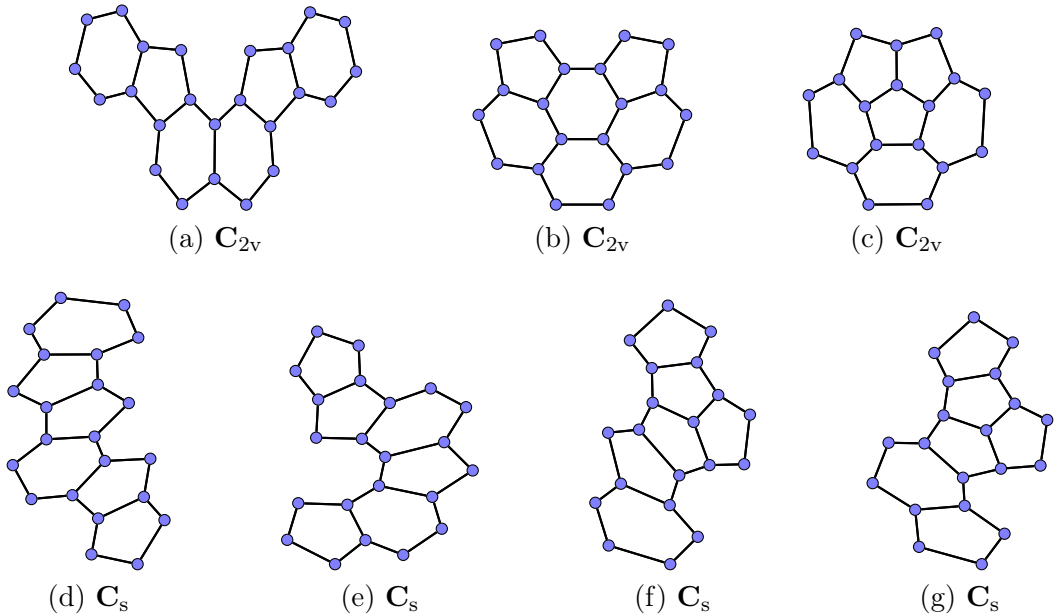
\begin{figure}[!htb]
\centering
$\begin{array}{ccccc}
\begin{tikzpicture}[scale=0.4,rotate=-52]
\tikzstyle{edge}=[draw, thick]
\tikzstyle{every node}=[draw, circle, fill=blue!50!white, inner sep=1.5pt]
\node (v1) at (1.654, 2.253) {};
\node (v2) at (0.253, 2.882) {};
\node (v3) at (1.937, 0.808) {};
\node (v4) at (0.633, 0.556) {};
\node (v5) at (2.172, -1.731) {};
\node (v6) at (0.752, -0.6) {};
\node (v7) at (-0.355, -2.067) {};
\node (v8) at (-0.402, -0.74) {};
\node (v9) at (-1.825, -2.113) {};
\node (v10) at (-2.755, -0.89) {};
\node (v11) at (2.524, 2.897) {};
\node (v12) at (2.254, 4.174) {};
\node (v13) at (0.73, 4.901) {};
\node (v14) at (-0.236, 4.184) {};
\node (v15) at (-0.449, 1.827) {};
\node (v16) at (3.173, -0.021) {};
\node (v17) at (3.495, -1.424) {};
\node (v18) at (2.169, -3.09) {};
\node (v19) at (0.729, -3.09) {};
\node (v20) at (-1.886, 0.031) {};
\node (v21) at (-2.258, -3.106) {};
\node (v22) at (-3.565, -3.128) {};
\node (v23) at (-4.612, -1.807) {};
\node (v24) at (-4.132, -0.706) {};
\path[edge] (v1) -- (v2);
\path[edge] (v1) -- (v11);
\path[edge] (v1) -- (v3);
\path[edge] (v2) -- (v1);
\path[edge] (v2) -- (v15);
\path[edge] (v2) -- (v14);
\path[edge] (v3) -- (v4);
\path[edge] (v3) -- (v1);
\path[edge] (v3) -- (v16);
\path[edge] (v4) -- (v3);
\path[edge] (v4) -- (v6);
\path[edge] (v4) -- (v15);
\path[edge] (v5) -- (v6);
\path[edge] (v5) -- (v17);
\path[edge] (v5) -- (v18);
\path[edge] (v6) -- (v5);
\path[edge] (v6) -- (v8);
\path[edge] (v6) -- (v4);
\path[edge] (v7) -- (v8);
\path[edge] (v7) -- (v19);
\path[edge] (v7) -- (v9);
\path[edge] (v8) -- (v7);
\path[edge] (v8) -- (v20);
\path[edge] (v8) -- (v6);
\path[edge] (v9) -- (v10);
\path[edge] (v9) -- (v7);
\path[edge] (v9) -- (v21);
\path[edge] (v10) -- (v9);
\path[edge] (v10) -- (v24);
\path[edge] (v10) -- (v20);
\path[edge] (v11) -- (v1);
\path[edge] (v11) -- (v12);
\path[edge] (v12) -- (v11);
\path[edge] (v12) -- (v13);
\path[edge] (v13) -- (v12);
\path[edge] (v13) -- (v14);
\path[edge] (v14) -- (v13);
\path[edge] (v14) -- (v2);
\path[edge] (v15) -- (v2);
\path[edge] (v15) -- (v4);
\path[edge] (v16) -- (v3);
\path[edge] (v16) -- (v17);
\path[edge] (v17) -- (v16);
\path[edge] (v17) -- (v5);
\path[edge] (v18) -- (v5);
\path[edge] (v18) -- (v19);
\path[edge] (v19) -- (v18);
\path[edge] (v19) -- (v7);
\path[edge] (v20) -- (v8);
\path[edge] (v20) -- (v10);
\path[edge] (v21) -- (v9);
\path[edge] (v21) -- (v22);
\path[edge] (v22) -- (v21);
\path[edge] (v22) -- (v23);
\path[edge] (v23) -- (v22);
\path[edge] (v23) -- (v24);
\path[edge] (v24) -- (v23);
\path[edge] (v24) -- (v10);
\end{tikzpicture} & \qquad\qquad &
\begin{tikzpicture}[scale=0.4]
\tikzstyle{edge}=[draw, thick]
\tikzstyle{every node}=[draw, circle, fill=blue!50!white, inner sep=1.5pt]
\node (v1) at (2.62, 0.776) {};
\node (v2) at (1.298, 0.32) {};
\node (v3) at (0.601, 1.483) {};
\node (v4) at (1.462, -1.852) {};
\node (v5) at (0.696, -0.856) {};
\node (v6) at (-0.602, 1.482) {};
\node (v7) at (-1.461, -1.852) {};
\node (v8) at (-0.697, -0.855) {};
\node (v9) at (-1.303, 0.321) {};
\node (v10) at (-2.621, 0.776) {};
\node (v11) at (2.864, 2.3) {};
\node (v12) at (1.404, 2.581) {};
\node (v13) at (3.447, -0.05) {};
\node (v14) at (2.829, -1.674) {};
\node (v15) at (0.88, -3.032) {};
\node (v16) at (-0.88, -3.032) {};
\node (v17) at (-1.401, 2.582) {};
\node (v18) at (-2.861, 2.302) {};
\node (v19) at (-2.828, -1.673) {};
\node (v20) at (-3.447, -0.048) {};
\path[edge] (v1) -- (v2);
\path[edge] (v1) -- (v11);
\path[edge] (v1) -- (v13);
\path[edge] (v2) -- (v1);
\path[edge] (v2) -- (v5);
\path[edge] (v2) -- (v3);
\path[edge] (v3) -- (v2);
\path[edge] (v3) -- (v6);
\path[edge] (v3) -- (v12);
\path[edge] (v4) -- (v5);
\path[edge] (v4) -- (v14);
\path[edge] (v4) -- (v15);
\path[edge] (v5) -- (v2);
\path[edge] (v5) -- (v4);
\path[edge] (v5) -- (v8);
\path[edge] (v6) -- (v9);
\path[edge] (v6) -- (v17);
\path[edge] (v6) -- (v3);
\path[edge] (v7) -- (v8);
\path[edge] (v7) -- (v16);
\path[edge] (v7) -- (v19);
\path[edge] (v8) -- (v5);
\path[edge] (v8) -- (v7);
\path[edge] (v8) -- (v9);
\path[edge] (v9) -- (v6);
\path[edge] (v9) -- (v8);
\path[edge] (v9) -- (v10);
\path[edge] (v10) -- (v9);
\path[edge] (v10) -- (v20);
\path[edge] (v10) -- (v18);
\path[edge] (v11) -- (v1);
\path[edge] (v11) -- (v12);
\path[edge] (v12) -- (v11);
\path[edge] (v12) -- (v3);
\path[edge] (v13) -- (v1);
\path[edge] (v13) -- (v14);
\path[edge] (v14) -- (v13);
\path[edge] (v14) -- (v4);
\path[edge] (v15) -- (v4);
\path[edge] (v15) -- (v16);
\path[edge] (v16) -- (v15);
\path[edge] (v16) -- (v7);
\path[edge] (v17) -- (v6);
\path[edge] (v17) -- (v18);
\path[edge] (v18) -- (v17);
\path[edge] (v18) -- (v10);
\path[edge] (v19) -- (v7);
\path[edge] (v19) -- (v20);
\path[edge] (v20) -- (v19);
\path[edge] (v20) -- (v10);
\end{tikzpicture} &  \qquad\qquad &
\begin{tikzpicture}[scale=0.4,rotate=68]
\tikzstyle{edge}=[draw, thick]
\tikzstyle{every node}=[draw, circle, fill=blue!50!white, inner sep=1.5pt]
\node (v1) at (2.177, 0.903) {};
\node (v2) at (0.96, 0.398) {};
\node (v3) at (0.421, 2.246) {};
\node (v4) at (1.886, -1.29) {};
\node (v5) at (0.724, -0.855) {};
\node (v6) at (-0.092, 1.117) {};
\node (v7) at (-2.153, 0.732) {};
\node (v8) at (-1.006, -2.043) {};
\node (v9) at (-0.6, -1.01) {};
\node (v10) at (-1.136, 0.288) {};
\node (v11) at (2.038, 2.287) {};
\node (v12) at (3.057, -0.175) {};
\node (v13) at (-0.447, 2.984) {};
\node (v14) at (-2.184, 2.133) {};
\node (v15) at (1.795, -2.424) {};
\node (v16) at (-0.036, -3.053) {};
\node (v17) at (-3.074, -0.218) {};
\node (v18) at (-2.329, -2.02) {};
\path[edge] (v1) -- (v2);
\path[edge] (v1) -- (v11);
\path[edge] (v1) -- (v12);
\path[edge] (v2) -- (v1);
\path[edge] (v2) -- (v5);
\path[edge] (v2) -- (v6);
\path[edge] (v3) -- (v6);
\path[edge] (v3) -- (v13);
\path[edge] (v3) -- (v11);
\path[edge] (v4) -- (v5);
\path[edge] (v4) -- (v12);
\path[edge] (v4) -- (v15);
\path[edge] (v5) -- (v2);
\path[edge] (v5) -- (v4);
\path[edge] (v5) -- (v9);
\path[edge] (v6) -- (v2);
\path[edge] (v6) -- (v10);
\path[edge] (v6) -- (v3);
\path[edge] (v7) -- (v10);
\path[edge] (v7) -- (v17);
\path[edge] (v7) -- (v14);
\path[edge] (v8) -- (v9);
\path[edge] (v8) -- (v16);
\path[edge] (v8) -- (v18);
\path[edge] (v9) -- (v5);
\path[edge] (v9) -- (v8);
\path[edge] (v9) -- (v10);
\path[edge] (v10) -- (v6);
\path[edge] (v10) -- (v9);
\path[edge] (v10) -- (v7);
\path[edge] (v11) -- (v1);
\path[edge] (v11) -- (v3);
\path[edge] (v12) -- (v1);
\path[edge] (v12) -- (v4);
\path[edge] (v13) -- (v3);
\path[edge] (v13) -- (v14);
\path[edge] (v14) -- (v13);
\path[edge] (v14) -- (v7);
\path[edge] (v15) -- (v4);
\path[edge] (v15) -- (v16);
\path[edge] (v16) -- (v15);
\path[edge] (v16) -- (v8);
\path[edge] (v17) -- (v7);
\path[edge] (v17) -- (v18);
\path[edge] (v18) -- (v17);
\path[edge] (v18) -- (v8);
\end{tikzpicture} \\
\text{(a)}\ \mathbf{C}_\mathrm{2v} && \text{(b)}\ \mathbf{C}_\mathrm{2v} && \text{(c)}\ \mathbf{C}_\mathrm{2v} \\[12pt]
\end{array}$

$\begin{array}{ccccccc}
\begin{tikzpicture}[scale=0.4]
\tikzstyle{edge}=[draw, thick]
\tikzstyle{every node}=[draw, circle, fill=blue!50!white, inner sep=1.5pt]
\node (v1) at (0.306, 2.536) {};
\node (v2) at (-1.451, 2.529) {};
\node (v3) at (0.338, 1.531) {};
\node (v4) at (-1.372, 0.778) {};
\node (v5) at (0.23, -0.199) {};
\node (v6) at (-1.371, -0.289) {};
\node (v7) at (0.571, -1.009) {};
\node (v8) at (-0.507, -2.148) {};
\node (v9) at (1.868, -2.205) {};
\node (v10) at (0.348, -2.921) {};
\node (v11) at (1.682, 3.095) {};
\node (v12) at (1.215, 4.173) {};
\node (v13) at (-0.893, 4.415) {};
\node (v14) at (-1.934, 3.38) {};
\node (v15) at (-2.464, 1.311) {};
\node (v16) at (1.368, 0.917) {};
\node (v17) at (-2.488, -1.148) {};
\node (v18) at (-1.869, -2.276) {};
\node (v19) at (2.083, -0.936) {};
\node (v20) at (2.554, -2.976) {};
\node (v21) at (1.66, -4.442) {};
\node (v22) at (0.127, -4.116) {};
\path[edge] (v1) -- (v2);
\path[edge] (v1) -- (v11);
\path[edge] (v1) -- (v3);
\path[edge] (v2) -- (v1);
\path[edge] (v2) -- (v15);
\path[edge] (v2) -- (v14);
\path[edge] (v3) -- (v4);
\path[edge] (v3) -- (v1);
\path[edge] (v3) -- (v16);
\path[edge] (v4) -- (v3);
\path[edge] (v4) -- (v6);
\path[edge] (v4) -- (v15);
\path[edge] (v5) -- (v6);
\path[edge] (v5) -- (v16);
\path[edge] (v5) -- (v7);
\path[edge] (v6) -- (v5);
\path[edge] (v6) -- (v17);
\path[edge] (v6) -- (v4);
\path[edge] (v7) -- (v8);
\path[edge] (v7) -- (v5);
\path[edge] (v7) -- (v19);
\path[edge] (v8) -- (v7);
\path[edge] (v8) -- (v10);
\path[edge] (v8) -- (v18);
\path[edge] (v9) -- (v10);
\path[edge] (v9) -- (v19);
\path[edge] (v9) -- (v20);
\path[edge] (v10) -- (v9);
\path[edge] (v10) -- (v22);
\path[edge] (v10) -- (v8);
\path[edge] (v11) -- (v1);
\path[edge] (v11) -- (v12);
\path[edge] (v12) -- (v11);
\path[edge] (v12) -- (v13);
\path[edge] (v13) -- (v12);
\path[edge] (v13) -- (v14);
\path[edge] (v14) -- (v13);
\path[edge] (v14) -- (v2);
\path[edge] (v15) -- (v2);
\path[edge] (v15) -- (v4);
\path[edge] (v16) -- (v3);
\path[edge] (v16) -- (v5);
\path[edge] (v17) -- (v6);
\path[edge] (v17) -- (v18);
\path[edge] (v18) -- (v17);
\path[edge] (v18) -- (v8);
\path[edge] (v19) -- (v7);
\path[edge] (v19) -- (v9);
\path[edge] (v20) -- (v9);
\path[edge] (v20) -- (v21);
\path[edge] (v21) -- (v20);
\path[edge] (v21) -- (v22);
\path[edge] (v22) -- (v21);
\path[edge] (v22) -- (v10);
\end{tikzpicture} &  \qquad\qquad &
\begin{tikzpicture}[scale=0.4,rotate=50]
\tikzstyle{edge}=[draw, thick]
\tikzstyle{every node}=[draw, circle, fill=blue!50!white, inner sep=1.5pt]
\node (v1) at (1.595, 1.892) {};
\node (v2) at (0.22, 2.564) {};
\node (v3) at (1.665, 0.667) {};
\node (v4) at (0.154, 0.611) {};
\node (v5) at (1.826, -1.555) {};
\node (v6) at (0.19, -0.369) {};
\node (v7) at (0.177, -2.419) {};
\node (v8) at (-0.678, -0.83) {};
\node (v9) at (-2.353, -1.647) {};
\node (v10) at (-1.838, -0.428) {};
\node (v11) at (2.459, 2.687) {};
\node (v12) at (1.931, 4.042) {};
\node (v13) at (0.466, 3.749) {};
\node (v14) at (-0.795, 1.799) {};
\node (v15) at (2.908, -0.064) {};
\node (v16) at (3.184, -1.363) {};
\node (v17) at (1.714, -2.886) {};
\node (v18) at (-0.43, -3.346) {};
\node (v19) at (-1.92, -2.923) {};
\node (v20) at (-3.614, -1.386) {};
\node (v21) at (-4.016, 0.353) {};
\node (v22) at (-2.847, 0.854) {};
\path[edge] (v1) -- (v2);
\path[edge] (v1) -- (v11);
\path[edge] (v1) -- (v3);
\path[edge] (v2) -- (v1);
\path[edge] (v2) -- (v14);
\path[edge] (v2) -- (v13);
\path[edge] (v3) -- (v4);
\path[edge] (v3) -- (v1);
\path[edge] (v3) -- (v15);
\path[edge] (v4) -- (v3);
\path[edge] (v4) -- (v6);
\path[edge] (v4) -- (v14);
\path[edge] (v5) -- (v6);
\path[edge] (v5) -- (v16);
\path[edge] (v5) -- (v17);
\path[edge] (v6) -- (v5);
\path[edge] (v6) -- (v8);
\path[edge] (v6) -- (v4);
\path[edge] (v7) -- (v8);
\path[edge] (v7) -- (v17);
\path[edge] (v7) -- (v18);
\path[edge] (v8) -- (v7);
\path[edge] (v8) -- (v10);
\path[edge] (v8) -- (v6);
\path[edge] (v9) -- (v10);
\path[edge] (v9) -- (v19);
\path[edge] (v9) -- (v20);
\path[edge] (v10) -- (v9);
\path[edge] (v10) -- (v22);
\path[edge] (v10) -- (v8);
\path[edge] (v11) -- (v1);
\path[edge] (v11) -- (v12);
\path[edge] (v12) -- (v11);
\path[edge] (v12) -- (v13);
\path[edge] (v13) -- (v12);
\path[edge] (v13) -- (v2);
\path[edge] (v14) -- (v2);
\path[edge] (v14) -- (v4);
\path[edge] (v15) -- (v3);
\path[edge] (v15) -- (v16);
\path[edge] (v16) -- (v15);
\path[edge] (v16) -- (v5);
\path[edge] (v17) -- (v5);
\path[edge] (v17) -- (v7);
\path[edge] (v18) -- (v7);
\path[edge] (v18) -- (v19);
\path[edge] (v19) -- (v18);
\path[edge] (v19) -- (v9);
\path[edge] (v20) -- (v9);
\path[edge] (v20) -- (v21);
\path[edge] (v21) -- (v20);
\path[edge] (v21) -- (v22);
\path[edge] (v22) -- (v21);
\path[edge] (v22) -- (v10);
\end{tikzpicture} & \qquad\qquad &
\begin{tikzpicture}[scale=0.4,rotate=20]
\tikzstyle{edge}=[draw, thick]
\tikzstyle{every node}=[draw, circle, fill=blue!50!white, inner sep=1.5pt]
\node (v1) at (1.672, 2.046) {};
\node (v2) at (0.198, 2.38) {};
\node (v3) at (1.965, 0.785) {};
\node (v4) at (-0.17, 1.288) {};
\node (v5) at (0.889, 0.223) {};
\node (v6) at (0.896, -1.308) {};
\node (v7) at (-0.9, 0.603) {};
\node (v8) at (-0.166, -1.284) {};
\node (v9) at (-2.083, -0.882) {};
\node (v10) at (-0.802, -2.106) {};
\node (v11) at (2.455, 2.85) {};
\node (v12) at (1.41, 4.151) {};
\node (v13) at (-0.171, 3.497) {};
\node (v14) at (2.732, -0.191) {};
\node (v15) at (1.988, -1.674) {};
\node (v16) at (-1.902, 0.857) {};
\node (v17) at (-3.097, -1.188) {};
\node (v18) at (-3.168, -2.674) {};
\node (v19) at (-1.457, -3.958) {};
\node (v20) at (-0.287, -3.415) {};
\path[edge] (v1) -- (v2);
\path[edge] (v1) -- (v11);
\path[edge] (v1) -- (v3);
\path[edge] (v2) -- (v1);
\path[edge] (v2) -- (v4);
\path[edge] (v2) -- (v13);
\path[edge] (v3) -- (v5);
\path[edge] (v3) -- (v1);
\path[edge] (v3) -- (v14);
\path[edge] (v4) -- (v5);
\path[edge] (v4) -- (v7);
\path[edge] (v4) -- (v2);
\path[edge] (v5) -- (v3);
\path[edge] (v5) -- (v6);
\path[edge] (v5) -- (v4);
\path[edge] (v6) -- (v5);
\path[edge] (v6) -- (v15);
\path[edge] (v6) -- (v8);
\path[edge] (v7) -- (v8);
\path[edge] (v7) -- (v16);
\path[edge] (v7) -- (v4);
\path[edge] (v8) -- (v7);
\path[edge] (v8) -- (v6);
\path[edge] (v8) -- (v10);
\path[edge] (v9) -- (v10);
\path[edge] (v9) -- (v17);
\path[edge] (v9) -- (v16);
\path[edge] (v10) -- (v9);
\path[edge] (v10) -- (v8);
\path[edge] (v10) -- (v20);
\path[edge] (v11) -- (v1);
\path[edge] (v11) -- (v12);
\path[edge] (v12) -- (v11);
\path[edge] (v12) -- (v13);
\path[edge] (v13) -- (v12);
\path[edge] (v13) -- (v2);
\path[edge] (v14) -- (v3);
\path[edge] (v14) -- (v15);
\path[edge] (v15) -- (v14);
\path[edge] (v15) -- (v6);
\path[edge] (v16) -- (v7);
\path[edge] (v16) -- (v9);
\path[edge] (v17) -- (v9);
\path[edge] (v17) -- (v18);
\path[edge] (v18) -- (v17);
\path[edge] (v18) -- (v19);
\path[edge] (v19) -- (v18);
\path[edge] (v19) -- (v20);
\path[edge] (v20) -- (v19);
\path[edge] (v20) -- (v10);
\end{tikzpicture} & \qquad\qquad &
\begin{tikzpicture}[scale=0.4,rotate=20]
\tikzstyle{edge}=[draw, thick]
\tikzstyle{every node}=[draw, circle, fill=blue!50!white, inner sep=1.5pt]
\node (v1) at (1.7, 1.871) {};
\node (v2) at (0.146, 2.16) {};
\node (v3) at (1.959, 0.711) {};
\node (v4) at (-0.396, 1.154) {};
\node (v5) at (0.769, 0.203) {};
\node (v6) at (0.791, -1.163) {};
\node (v7) at (-1.309, 0.496) {};
\node (v8) at (-0.496, -1.015) {};
\node (v9) at (-2.354, -1.474) {};
\node (v10) at (-0.764, -1.912) {};
\node (v11) at (2.597, 2.647) {};
\node (v12) at (1.668, 3.874) {};
\node (v13) at (0.029, 3.21) {};
\node (v14) at (2.791, -0.166) {};
\node (v15) at (2.061, -1.516) {};
\node (v16) at (-2.422, 0.924) {};
\node (v17) at (-3.38, -0.158) {};
\node (v18) at (-2.655, -2.676) {};
\node (v19) at (-0.895, -3.93) {};
\node (v20) at (0.163, -3.238) {};
\path[edge] (v1) -- (v2);
\path[edge] (v1) -- (v11);
\path[edge] (v1) -- (v3);
\path[edge] (v2) -- (v1);
\path[edge] (v2) -- (v4);
\path[edge] (v2) -- (v13);
\path[edge] (v3) -- (v5);
\path[edge] (v3) -- (v1);
\path[edge] (v3) -- (v14);
\path[edge] (v4) -- (v5);
\path[edge] (v4) -- (v7);
\path[edge] (v4) -- (v2);
\path[edge] (v5) -- (v3);
\path[edge] (v5) -- (v6);
\path[edge] (v5) -- (v4);
\path[edge] (v6) -- (v5);
\path[edge] (v6) -- (v15);
\path[edge] (v6) -- (v8);
\path[edge] (v7) -- (v8);
\path[edge] (v7) -- (v16);
\path[edge] (v7) -- (v4);
\path[edge] (v8) -- (v7);
\path[edge] (v8) -- (v6);
\path[edge] (v8) -- (v10);
\path[edge] (v9) -- (v10);
\path[edge] (v9) -- (v18);
\path[edge] (v9) -- (v17);
\path[edge] (v10) -- (v9);
\path[edge] (v10) -- (v8);
\path[edge] (v10) -- (v20);
\path[edge] (v11) -- (v1);
\path[edge] (v11) -- (v12);
\path[edge] (v12) -- (v11);
\path[edge] (v12) -- (v13);
\path[edge] (v13) -- (v12);
\path[edge] (v13) -- (v2);
\path[edge] (v14) -- (v3);
\path[edge] (v14) -- (v15);
\path[edge] (v15) -- (v14);
\path[edge] (v15) -- (v6);
\path[edge] (v16) -- (v7);
\path[edge] (v16) -- (v17);
\path[edge] (v17) -- (v16);
\path[edge] (v17) -- (v9);
\path[edge] (v18) -- (v9);
\path[edge] (v18) -- (v19);
\path[edge] (v19) -- (v18);
\path[edge] (v19) -- (v20);
\path[edge] (v20) -- (v19);
\path[edge] (v20) -- (v10);
\end{tikzpicture} \\
\text{(d)}\ \mathbf{C}_\mathrm{s}  & & \text{(e)}\ \mathbf{C}_\mathrm{s}  & & \text{(f)}\ \mathbf{C}_\mathrm{s}  & & \text{(g)}\ \mathbf{C}_\mathrm{s}  
\end{array}$
\caption{Smallest pent-hex patches with excess nullity $\eta(\altan(\Pi))-\eta(\Pi) = 2$. All patches have six faces and are shown in the order in which they occur in Table \ref{tbl:penthexptch}.}
\label{fig:mixedpatches}
\end{figure}

\begin{table}[!htbp]
\centering
\begin{tabular}{rr|rr|*{1}{rrr|}rr||rrr}
\multicolumn{2}{l|}{$\eta(B)$} & 0 & 0  & 2 & 2 & 2 & 4 & 4 & 1 & 3 & 5 \\
\multicolumn{2}{l|}{$\eta(\altan(B))$} & 1 & 2 & 2 & 3 & 4 & 4 & 5 & 1 & 3 & 5 \\
\hline 
\hline 
 \multirow{1}{*}{$\varepsilon$} & 1 & 1  & & & &  & & \\
& 2 & 1 & & &  & & & & \\
& 3 & 2 & & &  & &  & & 1 \\
& 4 & 6 & & &  & & & & 1 \\
& 5 & 14 & 1 & & &  & & & 7 \\
& 6 & 51 &  &  1 & 1 & & & & 28 \\
& 7 & 187 & 3 &  7 & & & & & 134 \\
& 8 & 764 & &  51 & 1 & & & & 619 \\
& 9 & 3211 & 12 &  318 & 4  & & & & 2957 & 3 \\
& 10 & 14073 & 34 & 1913 & 3  & &  & & 14024 & 39 \\
& 11 &62782 & 97 & 10916 & 14 & & & &  67046 & 374  \\
& 12 & 284552 & 366 &  60760 & 43 &  & 14 & & 320859 & 2990 \\
& 13 & 1303460 & 1401 & 331219 & 116 &  & 211 & &  1540174 &  21675 \\
& 14 &	6026901 & 4741 & 1778968 & 461 & & 2587 & 1 & 7408359 & 145559 \\
& 15 &	28066662 & 15891 & 9447084 & 1589 & 9 & 25084 & & 35721421 & 930122 & 48   
\end{tabular}
\caption{Excess nullity for altan-benzenoids based on natural attachment sets. Notation:
$\eta(B)$ is the
nullity of the parent, $\eta(\altan(B))$ is the nullity of its altan, and $\varepsilon$ is the number of hexagons
in the parent. The columns display the counts of cases with a given $(\eta(B), \eta(\altan(B)))$ pair. Double vertical lines separate cases with even and odd attachment sets.}
\label{table:benz}
\end{table}

\begin{table}[!htbp]
\centering
\begin{tabular}{rr|rr}
\multicolumn{2}{l|}{$\eta(B)$} & 0 & 0  \\
\multicolumn{2}{l|}{$\eta(\altan(B))$} & 1 & 2  \\
\hline 
\hline 
 \multirow{1}{*}{$\varepsilon$} & 1 & &  \\
& 2 & 1 & \\
& 3 & 2 & \\
& 4 & 5 & \\
& 5 & 11 & 1 \\
& 6 & 36 & \\
& 7 & 117 & 1 \\
& 8 & 411 & \\
& 9 & 1482 & 7 \\
& 10 & 5560 & 12 \\
& 11 & 21090 & 25 \\
& 12 & 81067 & 54 \\
& 13 & 313933 & 142 \\
& 14 & 1224021 & 507 \\
& 15 & 4798445 & 760 \\
 & 16 & 18894445 & 2536 
\end{tabular}
\caption{Excess nullity for altans of catafused benzenoids. Notation as in Table~\ref{table:benz}.}
\label{table:catabenz}
\end{table}

\begin{table}[!h]
\centering
\begin{tabular}{rr|rr||rr}
\multicolumn{2}{l|}{$\eta(B)$} & odd & even & 0 & 2 \\
\multicolumn{2}{l|}{$\eta(\altan(B))$} & $\eta(B)$ & $\eta(B)$ & 1 & 3  \\
\hline 
\hline 
 \multirow{1}{*}{$\varepsilon$}
 & 10 & 6 & 1 & 17 & 1 \\
 & 20 & 23 & 8 & 43 & 2 \\
 & 30 & 49 & 24 & 77 & 2 \\
 & 40 & 85 & 44 & 115 & 3 \\
 & 50 & 131 & 73 & 157 & 3 \\
 & 60 & 184 & 109 & 204 & 4 \\
 & 70 & 245 & 153 & 255 & 4 \\
 & 80 & 320 & 201 & 308 & 4 \\
 & 90 & 401 & 260 & 365 & 5 \\
 & 100 & 486 & 324 & 425 & 5 \\
 & 200 & 1842 & 1374 & 1150 & 7 \\
 & 300 & 4042 & 3160 & 2073 & 9 \\
 & 400 & 7039 & 5692 & 3152 & 11 \\
 & 500 & 10853 & 8968 & 4359 & 12 \\
 & 600 & 15485 & 13006 & 5697 & 13 \\
 & 700 & 20924 & 17767 & 7135 & 14 \\
 & 800 & 27146 & 23295 & 8675 & 15 \\
 & 900 & 34182 & 29599 & 10320 & 16 \\
 & 1000 & 41993 & 36603 & 12032 & 17 \\
 & 2000 & 163687 & 148394 & 33381 & 25 \\
 & 3000 & 364019 & 335729 & 60732 & 31
\end{tabular}
\caption{Excess nullity for altans of convex benzenoids. Notation as in Table~\ref{table:benz}.
The entries in each row display cumulative counts of convex benzenoids on up to $\varepsilon$ hexagons that have respectively
odd attachment set, even attachment set with excess nullity 0, and, after the double vertical lines, 
nullities $0$ and $2$ with excess
nullity $1$. }
\label{table:convex_benzenoids}
\end{table}

\begin{table}[!htbp]
\centering
\begin{tabular}{rr|rr|rrr|rrr|rrr|r}
\multicolumn{2}{l|}{$\eta(\Pi)$} & 0 & 0 & 1 & 1 & 1 & 2 & 2 & 2 & 3 & 3 & 3 & 4 \\
\multicolumn{2}{l|}{$\eta(\altan(\Pi))$} & 1 & 2 & 1 & 2 & 3 & 2 & 3 & 4 & 3 & 4 & 5 & 4 \\
\hline 
\hline 
 \multirow{1}{*}{$\pi$}
 & 2 & & & & 1 & & & & & & & & \\
 & 3 & 1 & & & & & & & & & & & \\
 & 4 & 3 & & & & & & & & & & & \\
 & 5 & 3 & & & & & & & & & & & \\
 & 6 & 7 & & 2 & 1 & & 1 & & & & & & \\
 & 7 & 13 & &  4 & & & & & & & & & \\
 & 8 & 36 & 4 & 5 & 1 & & 1  & & & & & & \\
 & 9 & 99 & 1 & 7 & & & & & & & & & \\
 & 10 & 241 & 3 & 13 & 5 & 1 & & 1 & & & & & \\
 & 11 & 636 & 2 & 38 &10 & & 1 & & & & & & 1 \\
 & 12 & 1614 & 16 & 90 & 35 &  1 & & & & & & &  \\
 & 13 & 4363 & 43 & 214 & 38 & 1 & 5 & & & & & & \\
 & 14 & 11796 & 78 & 483 & 86 & 3 & 9 & & & & & & \\
 & 15 & 31415 & 175 & 1269 & 176 & 1 & 6 & 1 & & 1 & & & \\
 & 16 & 85773 & 467 & 3682 & 532 & 6 & 55 & 2 & & 1 & & & \\
 & 17 & 230077 & 891 & 9816 & 1335 & 10 & 171 & 2 & & 2 & & & \\
 & 18 & 634433 & 2757 & 27844 & 3166 & 29 & 450 & 15 & 1 & 7 & & & \\
 & 19 & 1726528 & 7111 & 74818 & 8452 & 74 & 1019 & 13 & & 9 & & & \\
& 20 & 4775761 & 15562 & 206502 & 21357 & 184 & 3042 & 45 & 4 &  22 & & &
\end{tabular}
\caption{Excess nullity for altans of pure pentagonal patches. Notation:
$\eta(\Pi)$ is the
nullity of the parent, $\eta(\altan(\Pi))$ is the nullity of its altan, and $\pi$ is the number of pentagons
in the parent. The columns display the counts of cases with a given $(\eta(\Pi), \eta(\altan(\Pi)))$ pair. 
Only those patches with even $h$ are included in the table.}
\label{tbl:purepenta}
\end{table}

\begin{table}[!htbp]
\centering
\begin{tabular}{rr|rr|rrr|rrr|rrr}
\multicolumn{2}{l|}{$\eta(\Pi)$} & 0 & 0 & 1 & 1 & 1 & 2 & 2 & 2 & 3 & 3 & 3  \\
\multicolumn{2}{l|}{$\eta(\altan(\Pi))$} & 1 & 2 & 1 & 2 & 3 & 2 & 3 & 4 & 3 & 4 & 5  \\
\hline 
\hline 
 \multirow{1}{*}{$\sigma$}
 & 2 &   & & 1 & & & & & & & & \\
 & 3 & 1 & & & & & & & & & & \\
 & 4 & 8 & & & & & & & & & & \\
 & 5 & 14 & & & & & & & & & & \\
 & 6 & 131 & 3 & 20 & 7 & & 1 & & & & & \\
 & 7 & 520 & 2 & 76 & 11 & & & & & & & \\
 & 8 & 5140 & 24 & 480 & 77 & & 2 & 2 & & & & \\
 & 9 & 28326 & 114 & 2107 & 268 & 1 & 29  & & & & & \\
 & 10 & 235776 & 907 & 19021 &  2401 & 5 & 327 & 2 & & 5 & 3 & \\
 & 11 & 1506162 & 3912 & 121424 & 14523 & 46 &  2589 & 17 & & 11 & & \\
 & 12 & 12006702 & 26415 & 982165 & 112734 & 242 & 19050 & 285 & 3 & 263 & 2 & 
\end{tabular}
\caption{Excess nullity for altans of pure heptagonal patches. The number of heptagons is denoted $\sigma$
and all other notation is as in Table~\ref{tbl:purepenta}. Only those patches with even $h$ 
are included in the table.}
\label{tbl:purehepta}
\end{table}

\begin{table}[!htbp]
\centering
\begin{tabular}{rrr|rr|rrr|rrr}
\multicolumn{3}{l|}{$\eta(\Pi)$} & 0 & 0 & 1 & 1 & 1 & 2 & 2 & 2  \\
\multicolumn{3}{l|}{$\eta(\altan(\Pi))$} & 1 & 2 & 1 & 2 & 3 & 2 & 3 & 4  \\
\hline 
\hline 
 \multirow{1}{*}{$\pi$, $\varepsilon$}
 & 1 & 2 & 1 & & & & & & & \\
 & 2 & 1 & 2 & & & 1 & & & & \\
\cline{2-11}
 & 1 & 3 & 4 & & & & & & & \\
 & 2 & 2 & 12 & & 1 & 3 & & & & \\
 & 3 & 1 & 3 & & & & & & & \\
\cline{2-11}
 & 1 & 4 & 25 & & & & & & & \\
 & 2 & 3 & 67 & & 2 & 3 & & & & \\
 & 3 & 2 & 32 & & 2 & & & & & \\
 & 4 & 1 & 21 & & & 1 & & & & \\
\cline{2-11}
 & 1 & 5 & 132 & & & & & & & \\
 & 2 & 4 & 378 & 2 & 11 & 2 & 1 & & &  \\
 & 3 & 3 & 257 & 1 & 8 & 1 & & & & \\
 & 4 & 2 & 206 & 2 & 6 & 3 & & & & \\
 & 5 & 1 & 43 & 2 & 1 & & & & & \\
\cline{2-11}
& 1 & 6 & 722 & & & & & & &  \\
 & 2 & 5 & 2092 & 1 & 45 & 8 & & & & \\
 & 3 & 4 & 1987 & 3 & 64 & 5 & & & & \\
 & 4 & 3 & 1765 & 8 & 40 & 7 & 1 & & & \\
 & 5 & 2 & 643 & 3 & 11 & 2 & & & & \\
 & 6 & 1 & 158 & 1 & 12 & 1 & 1 & & & \\
\cline{2-11}
 & 1 & 7 & 3865 & 11 & 4 & 1 & & & & \\
 & 2 & 6 & 11987 & 14 & 234 & 17 & & 9 & &  \\
 & 3 & 5 & 14290 & 22 & 386 & 27 & & & & \\
 & 4 & 4 & 14759 & 29 & 260 & 28 & 1 & & & \\
 & 5 & 3 & 7272 & 9 & 109 & 7 & & & & \\
 & 6 & 2 & 2587 & 16 & 94 & 14 & 1 & 1 & & \\
 & 7 & 1 & 445 & 3 & 29 & 4 & & 1 & & \\
\cline{2-11}
 & 1 & 8 & 20867 & 43 & 66 & 9 & & 28 & &  \\
 & 2 & 7 & 68824 & 31 & 1289 & 76 & & 105 & 1 & \\
 & 3 & 6 & 99388 & 84 & 2447 & 127 & &  & &  \\
 & 4 & 5 & 115988 & 122 & 1719 & 129  & &   & &  \\
 & 5 & 4 & 72812 & 106 & 961 & 62 & & 3  & &  \\
 & 6 & 3 & 33541 & 74 & 651 & 59 & 1 & 2  & &  \\
 & 7 & 2 & 8917 & 18 & 340 & 29 & & 3  & &  \\
 & 8 & 1 & 1337 & 11 & 70 & 13  & &   & & 
\end{tabular}
\caption{
Excess nullity for altans of mixed pent-hex patches. The numbers of pentagons and hexagons are denoted $\pi$ and $\varepsilon$, respectively.
All other notation is as in Table~\ref{tbl:purepenta}. Only those patches with even $h$ 
are included in the table.}
\label{tbl:penthexptch}
\end{table}

\clearpage
\section{Relationship to previous work} 

The theorems in Section \ref{sec:mainsec} allow some comment on previous work on the systematics of altanisation.
Various observations made in the chemical literature on altanisation and the number of non-bonding orbitals were 
generalised and put into a formal mathematical  framework in two papers by Gutman \cite{Gutman2014a,Gutman2014b}
and have been taken up again in a recent paper by Dickens and Mallion on the predicted magnetic response properties
of altans \cite{Dickens2021}.

The paper \cite{Gutman2014a} concentrates on altans of benzenoids, and gives a number of rules for the face sizes, perimeter length,
Kekul\'e count, and numbers of NBMOs. Of these, one requires a change in the light of Theorem \ref{thm:armes2}. Rule 7 \cite{Gutman2014a}
states that ``Altan-benzenoid hydrocarbons have NBMOs. Kekulean altan-benzenoids have a unique NBMO.'' Unique NBMOs are
illustrated for two small cases in an accompanying figure. However, a Kekulean benzenoid has nullity zero, and therefore the present Theorem \ref{thm:armes2}
also allows the existence of altans of Kekulean benzenoids that have nullity two (but no more). 
The case illustrated in the present Figure \ref{fig:bendbenznegl} in fact exhibits nullity two, so the second part of Rule 7 should be amended to
``Kekulean altan-benzenoids can have either one
or two independent NBMOs.'' The {\lq unique\rq} NBMO referred to in \cite{Gutman2014a} is our {\lq special one\rq}, but as we have shown
there can be an additional NBMO in some cases.

The second paper \cite{Gutman2014b} generalises the theory to altans of general graphs, and gives key definitions that we have
followed here. Salient results  given as Theorems and Corollaries include a prior proof of Theorem~\ref{thm:minigutman}. 
Some confusion has arisen in later work \cite{Dickens2021} between proven results from \cite{Gutman2014b} and some cases that were being quoted there
merely as examples. In Corollary 11 of \cite{Gutman2014b} (which refers to cases with even $h = |H|$) it is noted that
``The case $\eta(G\sp\dagger) = \eta(G) +1$ is encountered if $G$ is non-singular.''
(In our notation, $G\sp\dagger$ would be $\altan(G, H)$.) The statement does not exclude 
the possibility of other values of 
$\eta(G\sp\dagger)$  being encountered, nor does it imply anything about singular graphs $G$.  
The reading that led to
the tentative proposal in \cite{Dickens2021} of a possible general rule of increase of nullity by $1$ for an altan of any $G$, 
is therefore not correct.
Indeed, that proposal runs into a contradiction when $G$ is itself an altan, as the authors themselves notice,
and has many exceptions even for non-altan $G$, as our tables show.

\section{Chemical implications}

To be concrete, our present understanding of the relations between the
numbers of NBMOs predicted by H{\"u}ckel  Theory for a parent hydrocarbon 
and its successive altans, assuming separate $\sigma$ and $\pi$ structures,
can be summarised in a short set of chemical rules. 
Let $\eta$ and $\eta(\altan)$ be the numbers of NBMOs
of the parent and altan respectively, and $h$ be the number of $sp\sp2$ CH sites on the perimeter of the 
parent molecule (the natural attachment set). Then:
\begin{enumerate}[label=(\roman*)]
\item 
If $h$ is odd, the number of $\pi$ NBMOs is the same for the parent and the first and all subsequent
iterated altans.

\item 
If $h$ is even, the number of $\pi$ NBMOs in the first altan has 
one of the three values, $\eta(\altan)=\eta$ or $\eta+1$ or $\eta+2$, 
subject to the proviso
 that if $\eta =0$, the altan has $\eta(\altan) = 1$ or  $2$.
All subsequent altans of the given parent have the same number of $\pi$ NBMOs as the first.

\item 
For the important specific case of a Kekulean benzenoid parent, for which 
the natural attachment set necessarily has even size and the nullity is $\eta = 0$, the rule is that   
$\eta(\altan)=1$ or $2$, and all iterated altans have 
$\eta(\altan)$ $\pi$ NBMOs.

\item 
Likewise, for a non-Kekulean benzenoid parent with odd nullity, 
the parent, first altan and all iterated altans have  the same number of $\pi$ NBMOs.
For a non-Kekulean benzenoid parent with even nullity, 
the nullity may increase (by at most $2$) 
on first altanisation, 
but then remains constant for all further iterated altans.

\item 
Finally, for the class of convex benzenoids, the conjecture, based on extensive computation, is that 
$\eta(\altan)$ in these cases is $\eta$ or $\eta +1$. For odd $h$, $\eta(\altan) = \eta$. 
The conjecture implies that the altan of a Kekulean convex benzenoid has $\eta =1$.
Notwithstanding any proof or disproof, exhaustive calculations show that there are no exceptions
		with fewer than 3000 hexagonal rings. 
\end{enumerate}
Two caveats about the applicability of these rules are in order. 
Chemists will want to bear in mind that the rules are derived within 
the H{\"u}ckel model for planar systems, as noted in the preamble.
The first is that when $h$ is even, they only give bounds on $\eta(\altan)$.
Those bounds are strict, but they still allow a small range of values.
As the proof strategies used in Section~\ref{sec:mainsec} show, the occurence of the more
unusual (i.e.\ higher) values of the difference $\eta(\altan)-\eta$ can be 
traced back to specific properties of the NBMOs of the parent graph.  
So far, we have not
attempted to give a complete characterisation of all such cases, though this would be an 
interesting direction for future research.

\section{Conclusion}
\label{sec:conclusion}

This paper has given some precise theorems
about the spectra of altan graphs. 
To assess their chemical significance, several factors must be considered.

The H{\"u}ckel  model itself is clearly a relatively crude model for 
planar or nearly planar $\pi$ systems.
Ordering of orbital energies may change with geometric factors, total charge, or
increase in the level of theory and use of
calculations based on more sophisticated quantum 
chemical methods.

For example,
the assumptions of H{\"u}ckel  theory in its simplest form, 
that resonance integrals for all  bonded pairs are equal
and that Coulomb integrals for all carbon centres are equal,
will become less appropriate for the geometrically non-planar structures
produced by iterated altanisation, and they may be 
expected to exhibit local differences in reactivity 
and changes in overall stability, as occur for example in fullerenes \cite{Atlas}. 

Even within the H{\"u}ckel  regime appropriate to geometrically
planar systems, there may be little practical difference between 
zero and near-zero graph eigenvalues (between non-bonding
and very weakly bonding or antibonding orbitals).
The crowding of the entire spectrum of a chemical graph into the open interval between 
$+3$ and $-3$ suggests that large graphs, such as those
produced by iterated altanisation may typically have several eigenvalues 
of very small magnitude that are not detected by the 
rules for counting strict zeroes but may have a similar impact on 
molecular properties.

Finally, the chemical significance of a non-bonding orbital depends
to an extent on the position of the zero eigenvalue within the 
the ordered graph spectrum.
H{\"u}ckel  Theory is acknowledged to perform at its best when applied to
neutral or near-neutral systems, where the numbers of positive
and negative eigenvalues are both close to $n/2$. If an NBMO occurs too early or too late
in the order, its properties are less likely to be relevant to 
physically realistic descriptions of the system.
Nevertheless, having precise rules for the number of such orbitals 
is a useful first step in building qualitative models of stability, reactivity and magnetic response. 

\section*{Acknowledgements}

The work of Nino Bašić is supported in part by the Slovenian Research Agency
(Research Programme P1-0294 and Research Projects J1-9187, J1-1691, N1-0140
and J1-2481).
PWF thanks the Leverhulme Trust for an
Emeritus Fellowship on the theme of
`Modelling molecular currents, conduction and aromaticity'.

\section*{ORCID iDs}
Nino Ba{\v s}i{\'c} \includegraphics[scale=0.05]{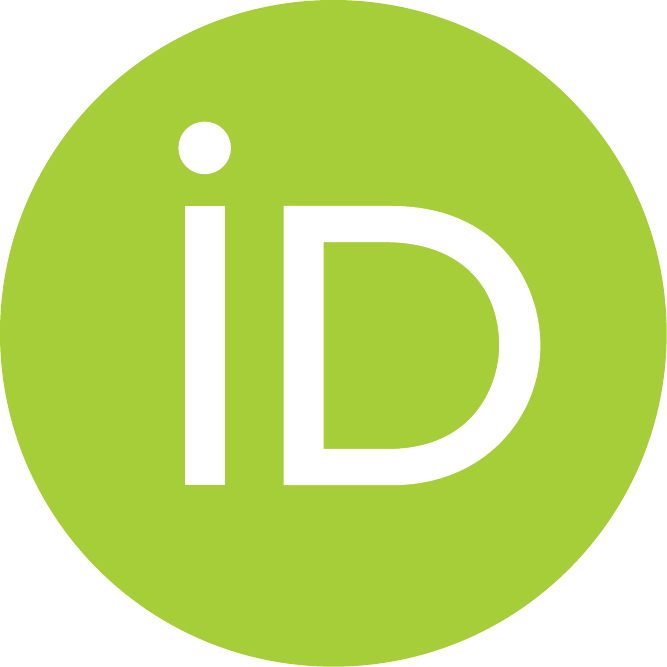} \href{https://orcid.org/0000-0002-6555-8668}{https://orcid.org/0000-0002-6555-8668}

\noindent
Patrick~W.~Fowler \includegraphics[scale=0.05]{ORCID_icon.pdf} \href{https://orcid.org/0000-0003-2106-1104}{https://orcid.org/0000-0003-2106-1104}

\bibliographystyle{amcunsrt}
\bibliography{references}

\begin{thebibliography}{10}
\expandafter\ifx\csname urlstyle\endcsname\relax
  \providecommand{\doi}[1]{doi:\discretionary{}{}{}#1}\else
  \providecommand{\doi}{doi:\discretionary{}{}{}\begingroup
  \urlstyle{rm}\Url}\fi

\bibitem{Schleyer2001}
P.~{\noop{Schleyer}}von Ragu{\'e} Schleyer (Guest~Ed.), Aromaticity [{S}pecial
  {I}ssue], \emph{Chem. Rev.} \textbf{101} (2001), 1115--1566.

\bibitem{Schleyer2005}
P.~{\noop{Schleyer}}von Ragu{\'e} Schleyer (Guest~Ed.), Delocalization---{P}i
  and {S}igma [{S}pecial {I}ssue], \emph{Chem. Rev.} \textbf{105} (2005),
  3433--3947.

\bibitem{Martin2015}
N.~Mart{\'\i}n and L.~T. Scott (Guest~Eds), Challenges in aromaticity: 150
  years after {K}ekul{\'e}'s benzene [{T}hemed {C}ollection], \emph{Chem. Soc.
  Rev.} \textbf{44} (2015), 6397--6643.

\bibitem{Peeks2020}
M.~Rickhaus, M.~Jir{\'a}sek, L.~Tejerina, H.~Gotfredsen, M.~D. Peeks, R.~Haver,
  H.-W. Jiang, T.~D.~W. Claridge and H.~L. Anderson, Global aromaticity at the
  nanoscale, \emph{Nat. Chem.} \textbf{12} (2020), 236--242,
  \doi{10.1038/s41557-019-0398-3}.

\bibitem{Peeks2021}
M.~Jir{\'a}sek, H.~L. Anderson and M.~D. Peeks, From macrocycles to quantum
  rings: {D}oes aromaticity have a size limit?, \emph{Acc. Chem. Res.}
  \textbf{54} (2021), 3241--3251, \doi{10.1021/acs.accounts.1c00323}.

\bibitem{London1937}
F.~London, Th{\'e}orie quantique des courants interatomiques dans les
  combinaisons aromatiques, \emph{J. Phys. Radium} \textbf{8} (1937), 397--409,
  \doi{10.1051/jphysrad:01937008010039700}.

\bibitem{Pople1958}
J.~A. Pople, Molecular orbital theory of aromatic ring currents, \emph{Mol.
  Phys.} \textbf{1} (1958), 175--180, \doi{10.1080/00268975800100211}.

\bibitem{McWeeny1958}
R.~McWeeny, Ring currents and proton magnetic resonance in aromatic molecules,
  \emph{Mol. Phys.} \textbf{1} (1958), 311--321,
  \doi{10.1080/00268975800100381}.

\bibitem{Schleyer1996}
P.~{\noop{Schleyer}}von Ragu{\'e}~Schleyer and H.~Jiao, What is aromaticity?,
  \emph{Pure Appl. Chem.} \textbf{68} (1996), 209--218,
  \doi{10.1351/pac199668020209}.

\bibitem{Monaco2012}
G.~Monaco and R.~Zanasi, Three contra-rotating currents from a rational design
  of polycyclic aromatic hydrocarbons: altan-corannulene and altan-coronene,
  \emph{J. Phys. Chem. A} \textbf{116} (2012), 9020--9026,
  \doi{10.1021/jp302635j}.

\bibitem{Stepien2018}
M.~St{\k e}pie{\'n}, An aromatic riddle: Decoupling annulene conjugation in
  coronoid macrocycles, \emph{Chem} \textbf{4} (2018), 1481--1483,
  \doi{10.1016/j.chempr.2018.06.008}.

\bibitem{Lawton1971}
W.~E. Barth and R.~G. Lawton, Synthesis of corannulene, \emph{J. Am. Chem.
  Soc.} \textbf{93} (1971), 1730--1745, \doi{10.1021/ja00736a028}.

\bibitem{Ege1972}
G.~Ege and H.~Vogler, Zur {K}onjugation in makrocyclischen {B}indungssystemen
  {XX}. {C}harakterordnung, magnetische {S}uszeptabilit{\"a}ten und chemische
  {V}erschiebungen von {C}orannulenen, \emph{Theor. Chim. Acta} \textbf{26}
  (1972), 55--65, \doi{10.1007/bf00527653}.

\bibitem{Diederich1978}
F.~Diederich and H.~A. Staab, Benzenoid versus annulenoid aromaticity:
  synthesis and properties of kekulene, \emph{Angewandte Chemie - International
  Edition in English} \textbf{17} (1978), 372--374,
  \doi{10.1002/anie.197803721}.

\bibitem{Aihara2014}
J.~Aihara, Validity and limitations of the annulene-within-an-annulene (awa)
  model for macrocyclic {$\pi$}-systems, \emph{RSC Adv.} \textbf{4} (2014),
  7256--7265, \doi{10.1039/c3ra45874a}.

\bibitem{Steiner2001}
E.~Steiner, P.~W. Fowler and L.~W. Jenneskens, Counter-rotating ring currents
  in coronene and corannulene, \emph{Angew. Chem. Int. Ed.} \textbf{40} (2001),
  362--366, \doi{10.1002/1521-3773(20010119)40:2<362::aid-anie362>3.3.co;2-q}.

\bibitem{Monaco2006}
G.~Monaco, R.~G. Viglione, R.~Zanasi and P.~W. Fowler, Designing ring-current
  patterns: [10,5]-coronene, a circulene with inverted rim and hub currents,
  \emph{J. Phys. Chem. A} \textbf{110} (2006), 7447--7452,
  \doi{10.1021/jp0600559}.

\bibitem{Monaco2007}
G.~Monaco, P.~W. Fowler, M.~Lillington and R.~Zanasi, Designing paramagnetic
  circulenes, \emph{Angew. Chem. Int. Ed.} \textbf{46} (2007), 1889--1892,
  \doi{10.1002/ange.200604261}.

\bibitem{coronene}
E.~Steiner, P.~W. Fowler and L.~W. Jenneskens, Counter-rotating ring currents
  in coronene and corannulene, \emph{Angew. Chem. Int. Ed.} \textbf{40} (2001),
  362--366, \doi{10.1002/1521-3773(20010119)40:2<362::aid-anie362>3.0.co;2-z}.

\bibitem{Monaco2013a}
G.~Monaco, M.~Memoli and R.~Zanasi, Additivity of current density patterns in
  altan-molecules, \emph{J. Phys. Org. Chem.} \textbf{26} (2013), 109--114,
  \doi{10.1002/poc.2958}.

\bibitem{Zanasi2016}
R.~Zanasi, P.~Della~Porta and G.~Monaco, The intriguing class of
  altan-molecules, \emph{J. Phys. Org. Chem.} \textbf{29} (2016), 793--798,
  \doi{10.1002/poc.3558}.

\bibitem{Monaco2013b}
G.~Monaco and R.~Zanasi, Anionic derivatives of altan-corannulene, \emph{J.
  Phys. Org. Chem.} \textbf{26} (2013), 730--736, \doi{10.1002/poc.3117}.

\bibitem{Dickens2014a}
T.~K. Dickens and R.~B. Mallion, Ring-current assessment of the
  annulene-within-an-annulene model for some large coupled super-ring
  conjugated-systems, \emph{Croat. Chem. Acta} \textbf{87} (2014), 221--232,
  \doi{10.5562/cca2397}.

\bibitem{Dickens2014b}
T.~K. Dickens and R.~B. Mallion, {$\pi$}-electron ring-currents and
  bond-currents in some conjugated altan-structures, \emph{J. Phys. Chem. A}
  \textbf{118} (2014), 3688--3697, \doi{10.1021/jp502585f}.

\bibitem{Dickens2014c}
T.~K. Dickens and R.~B. Mallion, Topological
  {H}{\"u}ckel-{L}ondon-{P}ople-{M}c{W}eeny ring currents and bond currents in
  altan-corannulene and altan-coronene, \emph{J. Phys. Chem. A} \textbf{118}
  (2014), 933--939, \doi{10.1021/jp411524k}.

\bibitem{Dickens2015a}
T.~K. Dickens and R.~B. Mallion, Topological ring-currents and bond-currents in
  the altan-{[r,s]}-coronenes, \emph{Chem. Commun.} \textbf{51} (2015),
  1819--1822, \doi{10.1039/c4cc07322c}.

\bibitem{Dickens2015b}
T.~K. Dickens and R.~B. Mallion, Topological ring-current and bond-current
  properties of the altans of certain {K}-factorizable conjugated systems
  containing ``fixed'' single-bonds, \emph{J. Phys. Chem. A} \textbf{119}
  (2015), 5019--5025, \doi{10.1021/acs.jpca.5b02826}.

\bibitem{Dickens2018}
T.~K. Dickens and R.~B. Mallion, Topological ring currents and bond currents in
  some neutral and anionic altans and iterated altans of corannulene and
  coronene, \emph{J. Phys. Chem. A} \textbf{122} (2018), 7666--7678,
  \doi{10.1021/acs.jpca.8b06862}.

\bibitem{Piccardo2020}
M.~Piccardo, A.~Soncini, P.~W. Fowler, G.~Monaco and R.~Zanasi, Design of
  annulene-within-an-annulene systems by the altanisation approach. a study of
  altan-{[n]}annulenes, \emph{Phys. Chem. Chem. Phys.} \textbf{22} (2020),
  5476--5486, \doi{10.1039/c9cp06835j}.

\bibitem{Dickens2020}
T.~K. Dickens and R.~B. Mallion, Topological ring-currents and bond-currents in
  hexaanionic altans and iterated altans of corannulene and coronene, \emph{J.
  Phys. Chem. A} \textbf{124} (2020), 7973--7990,
  \doi{10.1021/acs.jpca.0c04606}.

\bibitem{Dickens2021}
T.~K. Dickens and R.~B. Mallion, Topological ring currents and bond currents in
  neutral and dianionic altans and iterated altans of benzene, naphthalene, and
  azulene, \emph{J. Phys. Chem. A} \textbf{125} (2021), 10485--10499,
  \doi{10.1021/acs.jpca.1c07846}.

\bibitem{Gutman2014a}
I.~Gutman, Topological properties of altan-benzenoid hydrocarbons, \emph{J.
  Serbian Chem. Soc.} \textbf{79} (2014), 1515--1521,
  \doi{10.2298/jsc140619080g}.

\bibitem{Gutman2014b}
I.~Gutman, Altan derivatives of a graph, \emph{Iranian J. Math. Chem.}
  \textbf{5} (2014), 85--90, \doi{10.22052/ijmc.2014.6127}.

\bibitem{Basic2015}
N.~Ba\v{s}i\'{c} and T.~Pisanski, Iterated altans and their properties,
  \emph{MATCH Commun. Math. Comput. Chem.} \textbf{74} (2015), 645--658.

\bibitem{Basic2016}
N.~Ba\v{s}i\'{c}, P.~W. Fowler and T.~Pisanski, Coronoids, patches and
  generalised altans, \emph{J. Math. Chem.} \textbf{54} (2016), 977--1009,
  \doi{10.1007/s10910-016-0599-6}.

\bibitem{Trinajstic1992}
N.~Trinajsti\'{c}, \emph{Chemical {G}raph {T}heory}, Mathematical Chemistry
  Series, CRC Press, Boca Raton, FL, 2nd edition, 1992,
  \doi{10.1007/s10910-008-9464-6}.

\bibitem{Gutman2021a}
I.~Gutman, Computing the dependence of graph energy on nullity: the method of
  siblings, \emph{Discrete Math. Lett.} \textbf{7} (2021), 30--33,
  \doi{10.47443/dml.2021.0038}, \url{https://doi.org/10.47443/dml.2021.0038}.

\bibitem{Gutman2021b}
I.~Gutman, Graph energy and nullity, \emph{Open J. Discrete Appl. Math.}
  \textbf{4} (2021), 25--28.

\bibitem{Triantafillou2016}
I.~Gutman and I.~Triantafillou, Dependence of graph energy on nullity: a case
  study, \emph{MATCH Commun. Math. Comput. Chem.} \textbf{76} (2016), 761--769.

\bibitem{Cvetkovic1995}
D.~M. Cvetkovi\'{c}, M.~Doob and H.~Sachs, \emph{Spectra of {G}raphs: {T}heory
  and {A}pplications}, Johann Ambrosius Barth, Heidelberg, 3rd edition, 1995.

\bibitem{Haemers2012}
A.~E. Brouwer and W.~H. Haemers, \emph{Spectra of {G}raphs}, Universitext,
  Springer, New York, 2012, \doi{10.1007/978-1-4614-1939-6}.

\bibitem{Cvetkovic1997}
D.~M. Cvetkovi\'{c}, P.~Rowlinson and S.~Simi\'{c}, \emph{Eigenspaces of
  {G}raphs}, volume~66 of \emph{Encyclopedia of Mathematics and its
  Applications}, Cambridge University Press, Cambridge, 1997,
  \doi{10.1017/cbo9781139086547}.

\bibitem{Cvetkovic2010}
D.~M. Cvetkovi\'{c}, P.~Rowlinson and S.~Simi\'{c}, \emph{An introduction to
  the theory of graph spectra}, volume~75 of \emph{London Mathematical Society
  Student Texts}, Cambridge University Press, Cambridge, 2010.

\bibitem{Chung1997}
F.~R.~K. Chung, \emph{Spectral {G}raph {T}heory}, volume~92 of \emph{CBMS
  Regional Conference Series in Mathematics}, American Mathematical Society,
  Providence, RI, 1997.

\bibitem{Borovicanin2011}
I.~Gutman and B.~Borovi\'{c}anin, Nullity of graphs: an updated survey,
  \emph{Zb. Rad. (Beogr.)} \textbf{14(22)} (2011), 137--154.

\bibitem{Graver2010}
J.~E. Graver and C.~M. Graves, Fullerene patches {I}, \emph{Ars Math. Contemp.}
  \textbf{3} (2010), 109--120, \doi{10.26493/1855-3974.135.29d}.

\bibitem{Graver2014}
J.~E. Graver, C.~Graves and S.~J. Graves, Fullerene patches {II}, \emph{Ars
  Math. Contemp.} \textbf{7} (2014), 405--421,
  \doi{10.26493/1855-3974.391.a0d}.

\bibitem{Klein2006}
D.~J. Klein and A.~T. Balaban, The eight classes of positive-curvature
  graphitic nanocones, \emph{J. Chem. Inf. Model.} \textbf{46} (2006),
  307--320, \doi{10.1021/ci0503356}.

\bibitem{Brinkmann2005}
G.~Brinkmann, O.~Delgado-Friedrichs and U.~von Nathusius, Numbers of faces and
  boundary encodings of patches, in: \emph{Graphs and {D}iscovery}, Amer. Math.
  Soc., Providence, RI, volume~69 of \emph{DIMACS Ser. Discrete Math. Theoret.
  Comput. Sci.}, pp. 27--38, 2005, \doi{10.1090/dimacs/069/03},
  \url{https://doi.org/10.1090/dimacs/069/03}.

\bibitem{Gutman1989}
I.~Gutman and S.~J. Cyvin, \emph{Introduction to the {T}heory of {B}enzenoid
  {H}ydrocarbons}, Springer-Verlag, Berlin, 1989.

\bibitem{gutman2012}
R.~Cruz, I.~Gutman and J.~Rada, Convex hexagonal systems and their topological
  indices, \emph{MATCH Commun. Math. Comput. Chem.} \textbf{68} (2012),
  97--108.

\bibitem{Monaco2015}
G.~Monaco, On the diatropic perimeter of iterated altan-molecules,
  \emph{Physical Chemistry Chemical Physics} \textbf{17} (2015), 28544--28547,
  \doi{10.1039/c5cp05520b}.

\bibitem{CaGe}
G.~Brinkmann, O.~Delgado~Friedrichs, S.~Lisken, A.~Peeters and N.~Van~Cleemput,
  Ca{G}e---a virtual environment for studying some special classes of plane
  graphs---an update, \emph{MATCH Commun. Math. Comput. Chem.} \textbf{63}
  (2010), 533--552.

\bibitem{Basic2018}
N.~Ba\v{s}i\'{c}, P.~W. Fowler and T.~Pisanski, Stratified enumeration of
  convex benzenoids, \emph{MATCH Commun. Math. Comput. Chem.} \textbf{80}
  (2018), 153--172.

\bibitem{SageMath}
{The Sage Developers}, \emph{{S}ageMath, the {S}age {M}athematics {S}oftware
  {S}ystem ({V}ersion 9.4)}, 2021, {\tt https://www.sagemath.org}.

\bibitem{PentaCluster}
N.~Ba\v{s}i\'{c}, G.~Brinkmann, P.~W. Fowler, T.~Pisanski and N.~Van~Cleemput,
  Sizes of pentagonal clusters in fullerenes, \emph{J. Math. Chem.} \textbf{55}
  (2017), 1669--1682, \doi{10.1007/s10910-017-0754-8}.

\bibitem{Atlas}
P.~W. Fowler and D.~E. Manolopoulos, \emph{An {A}tlas of {F}ullerenes}, Dover,
  Mineola, New York, 2006.

\end{thebibliography}

\end{document}